\documentclass[11pt]{article}

\usepackage{amsmath}
\usepackage{amssymb}
\usepackage{amsthm}
\usepackage{amscd}
\usepackage{amsfonts}
\usepackage[all]{xy}
\usepackage{ifthen}
\usepackage{a4}

\long\def\forget#1{}

\newdir^{ (}{{}*!/-3pt/\dir^{(}}    
\newdir^{  }{{}*!/-3pt/\dir^{}}    
\newdir_{ (}{{}*!/-3pt/\dir_{(}}    
\newdir_{  }{{}*!/-3pt/\dir_{}}

\def\quotes#1{{''#1''}} 

\catcode`\ä=\active
\catcode`\ö=\active
\catcode`\ü=\active
\catcode`\Ä=\active
\catcode`\Ö=\active
\catcode`\Ü=\active
\catcode`\ß=\active
\catcode`\é=\active
\catcode`\è=\active
\catcode`\É=\active

\defä{\"a}
\defö{\"o}
\defü{\"u}
\defÄ{\"A}
\defÖ{\"O}
\defÜ{\"U}
\letß=\ss
\defé{\'{e}}
\defè{\`{e}}
\defÉ{\'{E}}

\DeclareMathOperator{\Quot}{Quot}
\DeclareMathOperator{\Spec}{Spec}

\DeclareMathOperator{\id}{id}

\DeclareMathOperator{\coker}{coker}
\DeclareMathOperator{\im}{im}

\newcommand{\sep}{{\rm sep}}
\newcommand{\alg}{{\rm alg}}
\DeclareMathOperator{\inv}{inv}
\newcommand{\op}{{\rm op}}
\DeclareMathOperator{\weight}{wt}
\DeclareMathOperator{\supp}{supp}

\DeclareMathOperator{\Hom}{Hom}
\DeclareMathOperator{\End}{End}
\DeclareMathOperator{\Isog}{Isog}
\DeclareMathOperator{\QHom}{QHom}
\DeclareMathOperator{\QEnd}{QEnd}
\DeclareMathOperator{\QIsog}{QIsog}
\DeclareMathOperator{\Gal}{Gal}

\newtheoremstyle{statement}%
{2\parskip}
{\parskip}
{\itshape}
{}
{\bfseries}
{}
{\newline}
{}

\newtheoremstyle{notice}%
{2\parskip}
{\parskip}
{}
{}
{\itshape}
{}
{\newline}
{}

\newtheorem{all}{all}[section]

\theoremstyle{plain}
\newtheorem{Definition}[all]{Definition}
\newtheorem{Lemma}[all]{Lemma}
\newtheorem{Proposition}[all]{Proposition}
\newtheorem{Corollary}[all]{Corollary}
\newtheorem{Theorem}[all]{Theorem}

\theoremstyle{remark}

\theoremstyle{definition}
\newtheorem{Remark}[all]{Remark}
\newtheorem{Example}[all]{Example}

\def\Proofof#1{Proof $($#1\,$)$.} 

\newenvironment{suchthat}
{\setlength{\parskip}{0ex}%
\begin{enumerate}\setlength{\parskip}{0ex}\setlength{\itemsep}{0ex}%
}{%
\end{enumerate}%
}%

\def\II#1{{[\,#1\,]}}  
\def\dual#1{{#1}^\vee} 

\def\Z{\mathbb{Z}} 
\def\N{\mathbb{N}} 
\def\Ff{{\mathbb{F}}} 
\def\Fq{{\mathbb{F}_q}} 
\def\Fs{{\mathbb{F}_s}} 
\def\PP{\mathbb{P}} 

\def\O{{\cal O}} 
\def\I{{\cal J}} 
\def\F{{\cal F}} 

\def\OC{{\O_C}} 
\def\OS{{\O_S}}

\def\FF{{\underline{\F}}} 
\def\ZZ{{\underline{0}}} 

\def\ulM{{\underline{M\!}\,}{}}
\def\ulN{{\underline{N\!}\,}{}}
\def\ulTM{{\underline{\,\,\wt{\!\!M}\!}\,}{}}
\def\ulTN{{\underline{\wt{N}\!}\,}{}}
\def\ulHM{{\underline{\hat M\!}\,}{}}
\def\ulHN{{\underline{\hat N\!}\,}{}}

\def\P{{\mit\Pi}} 
\def\t{{\mit\tau}} 
\def\s{{\sigma^*}} 

\def\TP{\widetilde{\P}} 
\def\Tt{\tilde{\tau}}

\def\TFF{\widetilde{\FF}}

\def\chr{\varepsilon} 
\def\otimesidOCL#1{\!\otimes1} 

\def\matr#1#2#3#4{\left(\genfrac{}{}{0pt}{}{#1}{#2}\,\genfrac{}{}{0pt}{}{#3}{#4}\right)}
\def\tmatr#1#2#3#4{{\textstyle\Big(\genfrac{}{}{0pt}{}{#1}{#2}\,\genfrac{}{}{0pt}{}{#3}{#4}\Big)}}
\def\smatr#1#2#3#4{{\scriptstyle\big(\genfrac{}{}{0pt}{}{#1}{#2}\,\genfrac{}{}{0pt}{}{#3}{#4}\big)}}

\def\tvect#1#2{{\textstyle\big(\genfrac{}{}{0pt}{}{#1}{#2}\big)}}

\def\TA{\tilde{A}}
\def\Lsep{{L^{\sep}}}

\def\CFs{{C_\Fs}} 
\def\Av{{A_v}} 
\def\AvG{{\Av\II{G}}} 
\def\Qv{{Q_v}} 
\def\QvG{{\Qv\II{G}}} 
\def\cl#1{{\overline{#1}}} 
\def\FS{{\cl{\Ff}_s}} 
\def\VvFF{{V_v\FF}} 
\def\GalFSFs{{\Gal(\FS/\Fs)}} 

\def\xQv{\otimes_Q\Qv}
\def\Fss{{\Ff_{s'}}}
\def\xFss{\otimes_\Fs\Fss}

\def\smallexact#1#2#3#4#5#6#7#8#9{%
{\,#1\rightarrow#3\rightarrow#5%
\ifthenelse{\equal{#7}{}}{}{\rightarrow#7%
\ifthenelse{\equal{#9}{}}{}{\rightarrow#9}}%
\,}}

\def\exact#1#2#3#4#5#6#7#8#9{%
{#1\longrightarrow#3\longrightarrow#5%
\ifthenelse{\equal{#7}{}}{}{\longrightarrow#7%
\ifthenelse{\equal{#9}{}}{}{\longrightarrow#9}}%
\,}}

\def\bigexact#1#2#3#4#5#6#7#8#9{%
\begin{CD}%
#1 @>{#2}>> #3 @>{#4}>> #5 
\ifthenelse{\equal{#7}{}}{}{@>{#6}>> #7 
\ifthenelse{\equal{#9}{}}{}{@>{#8}>> #9 }}
\end{CD}%
}

\newcommand{\DS}{\displaystyle}
\newcommand{\TS}{\textstyle}
\newcommand{\SC}{\scriptstyle}
\newcommand{\SSC}{\scriptscriptstyle}
\let\setminus\smallsetminus

\DeclareMathOperator{\Id}{Id}
\DeclareMathOperator{\charakt}{char}
\def\isoto{\stackrel{\;\sim\es}{\longrightarrow}}
\newcommand{\longto}{\longrightarrow}
\newcommand{\onto}{\mbox{\mathsurround=0pt \;$\longrightarrow \hspace{-0.7em} \to$\;}}

\newcommand{\es}{\enspace}

\DeclareMathOperator{\rk}{rk}
\newcommand{\wt}{\widetilde}
\newcommand{\wh}{\widehat}
\newcommand{\Fa}{{\mathfrak{a}}}
\newcommand{\ulK}{{\ul K}}
\newcommand{\BF}{{\mathbb{F}}}
\newcommand{\CM}{{\cal{M}}}
\newcommand{\ul}[1]{{\underline{#1}}}
\newcommand{\dbl}{{\mathchoice{\mbox{\rm [\hspace{-0.15em}[}}
                              {\mbox{\rm [\hspace{-0.15em}[}}
                              {\mbox{\scriptsize\rm [\hspace{-0.15em}[}}
                              {\mbox{\tiny\rm [\hspace{-0.15em}[}}}}
\newcommand{\dbr}{{\mathchoice{\mbox{\rm ]\hspace{-0.15em}]}}
                              {\mbox{\rm ]\hspace{-0.15em}]}}
                              {\mbox{\scriptsize\rm ]\hspace{-0.15em}]}}
                              {\mbox{\tiny\rm ]\hspace{-0.15em}]}}}}
\newcommand{\dpl}{{\mathchoice{\mbox{\rm (\hspace{-0.15em}(}}
                              {\mbox{\rm (\hspace{-0.15em}(}}
                              {\mbox{\scriptsize\rm (\hspace{-0.15em}(}}
                              {\mbox{\tiny\rm (\hspace{-0.15em}(}}}}
\newcommand{\dpr}{{\mathchoice{\mbox{\rm )\hspace{-0.15em})}}
                              {\mbox{\rm )\hspace{-0.15em})}}
                              {\mbox{\scriptsize\rm )\hspace{-0.15em})}}
                              {\mbox{\tiny\rm )\hspace{-0.15em})}}}}

\def\?{\ 
???\ \immediate\write16{}
\immediate\write16{Warning: There was still a question mark . . . }
\immediate\write16{}}

\newcommand{\BHAaa}{Proposition~7.4}   
\newcommand{\BHAbb}{Proposition~4.2}   
\newcommand{\BHAcc}{Remark~5.5}        
\newcommand{\BHAdd}{Proposition~9.11}  
\newcommand{\BHAee}{Theorem~9.9}       
\newcommand{\BHAff}{Proposition~1.6}   
\newcommand{\BHAgghh}{Theorems~8.6 and 8.7}
\newcommand{\BHAii}{Theorem~7.8}       
\newcommand{\BHAjj}{Proposition~7.3}   
\newcommand{\BHAkk}{Proposition~6.10}  
\newcommand{\BHAmm}{Proposition~9.4}   
\newcommand{\BHAllmm}{Propositions~6.5 and 9.4}
\newcommand{\BHAnn}{Proposition~5.1}   
\newcommand{\BHAoo}{Theorem~3.1}       
\newcommand{\BHApp}{Corollary~3.5}     
\newcommand{\BHAqq}{Theorem~9.5}       
\newcommand{\BHArr}{Corollary~5.4}     
\newcommand{\BHAss}{Proposition~1.2}   
\newcommand{\BHAtt}{Theorem~9.8}       
\newcommand{\BHAuu}{Definition~3.2}    

\begin{document}

\author{Matthias Bornhofen, Urs Hartl
\footnote{The second author acknowledges support of the Deutsche Forschungsgemeinschaft in form of DFG-grant HA3006/2-1 and SFB 478}
}

\title{Pure Anderson Motives over Finite Fields}

\date{October 30, 2008} 

\maketitle

\begin{abstract}
\noindent
In the arithmetic of function fields Drinfeld modules play the role that elliptic curves take on in the arithmetic of number fields. As higher dimensional generalizations of Drinfeld modules, and as the appropriate analogues of abelian varieties, G.~Anderson introduced pure $t$-motives. In this article we study the arithmetic of the latter. We investigate which pure $t$-motives are semisimple, that is, isogenous to direct sums of simple ones. We give examples for pure $t$-motives which are not semisimple. Over finite fields the semisimplicity is equivalent to the semisimplicity of the endomorphism algebra, but also this fails over infinite fields. Still over finite fields we study the Zeta function and the endomorphism rings of pure $t$-motives and criteria for the existence of isogenies. We obtain answers which are similar to Tate's famous results for abelian varieties. 

\end{abstract}


\thispagestyle{empty}

\section*{Introduction}
\addcontentsline{toc}{section}{Introduction}

In the last decades the Arithmetic of Function Fields has acquired great impetus caused by Drinfeld's~\cite{Drinfeld,Drinfeld3} invention of the concepts of \emph{elliptic modules} (today called \emph{Drinfeld modules}) and \emph{elliptic sheaves} in the 1970s. Both are analogues of elliptic curves. The latter live in the Arithmetic of Number Fields, like their higher dimensional generalizations abelian varieties. In \cite{BH_A,Hl} we claimed that pure Anderson motives (a slight generalization of the pure $t$-motives introduced by Anderson~\cite{Anderson}) and abelian $\tau$-sheaves should be viewed as the appropriate analogues for abelian varieties and higher dimensional generalizations of elliptic sheaves or modules. We want to further support this claim in the present article by developing the theory of pure Anderson motives over finite fields.

To give the definition of pure Anderson motives let $C$ be a connected smooth projective curve over $\Ff_q$, let $\infty\in C(\Ff_q)$ be a fixed point, and let $A=\Gamma(C\setminus\{\infty\},\O_C)$. For a field $L\supset\Fq$ let $\s$ be the endomorphism of $A_L:=A\otimes_\Fq L$ sending $a\otimes b$ to $a\otimes b^q$ for $a\in A$ and $b\in L$. Let $c^\ast:A\to L$ be an $\Fq$-homomorphism and let $J=(a\otimes 1-1\otimes c^\ast(a):a\in A)\subset A_L$. A \emph{pure Anderson motive $\ulM=(M,\tau)$ of rank $r$, dimension $d$ and characteristic $c^\ast$} consists of a locally free $A_L$-module $M$ of rank $r$ and an $A_L$-homomorphism $\t:\s M:= M\otimes_{A_L,\s}A_L\to M$ with $\dim_L\coker\t=d$ and $J^d\cdot\coker\t=0$, such that $M$ possesses an extension to a locally free sheaf $\CM$ on $C\times_\Fq L$ on which $\t^l:(\s)^l\CM\to\CM(k\cdot\infty)$ is an isomorphism near $\infty$ for some positive integers $k$ and $l$. The last condition is the purity condition. The ratio $\frac{k}{l}$ equals $\frac{d}{r}$ and is called the \emph{weight of $\ulM$}. Anderson's definition of pure $t$-motives~\cite{Anderson} is recovered by setting $C=\PP^1_\Fq$ and $A=\Fq[t]$. 
In the first two sections we recall the definition of morphisms and isogenies between pure Anderson motives as well as some facts from \cite{BH_A}. Also for an isogeny $f$ between pure Anderson motives we define the degree of $f$ as an ideal of $A$ (\ref{Def1.7.6}) which annihilates $\coker f$ (\ref{Prop3.28a}). If $\ulM$ is a semisimple (see below) pure Anderson motive over a finite field, the degree of any isogeny $f:\ulM\to\ulM$ is a principal ideal and has a canonical generator (\ref{Prop3.4.1}). In particular $f$ has a canonical dual.

Next we address the question whether every pure Anderson motive is \emph{semisimple}, that is, isogenous to a direct sum of simple pure Anderson motives. A pure Anderson motive is called \emph{simple} if it has no non-trivial factor motives. This question is the analogue of the classical theorem of Poincar\'e-Weil on the semisimplicity of abelian varieties. By giving a counterexample (Example~\ref{Ex3.1}) we demonstrate that the answer to this question is negative in general. On the positive side we show that every pure Anderson motive over a finite base field becomes semisimple after a field extension whose degree is a power of $q$ (\ref{Thm3.8b}), and then stays semisimple after any further field extension (\ref{Cor3.8c}). Let $Q$ be the function field of $C$. Then the endomorphism $Q$-algebra $\QEnd(\ulM):=\End(\ulM)\otimes_A Q$ of a semisimple pure Anderson motive is semisimple (\ref{QEND-DIVISION-MATRIX}) and over a finite field also the converse is true (\ref{Thm3.8}). This is false however over an infinite field (\ref{Ex3.10c}).

Like for abelian varieties the behavior of a pure Anderson motive $\ulM$ over a finite field is controlled by its Frobenius endomorphism $\pi$ (defined in \ref{Def2.19b}). If $\ulM$ is semisimple we determine the dimension and the local Hasse invariants of its endomorphism $Q$-algebra $\QEnd(\ulM)$ in terms of $\pi$ (\ref{Thm3.5a}, \ref{THEOREM-2}). We define a Zeta function $Z_\ulM$  for a pure Anderson motive $\ulM$ (Definition~\ref{DefZetaFkt}) and we show that it satisfies the Riemann hypothesis (\ref{ThmRH}), and has an expression in terms of the degrees $\deg(1-\pi^i)$ for all $i$ if $\ulM$ is semisimple (\ref{ThmZeta}). We prove the following isogeny criterion.

\bigskip

\noindent
{\bfseries Theorem \ref{THEOREM-1}.} {\it Let $\ulM$ and $\ulM'$ be semisimple pure Anderson motives over a finite field and let $\pi$, respectively $\pi'$, be their Frobenius endomorphisms. Then the following are equivalent:
\begin{enumerate}
\item 
$\ulM$ and $\ulM'$ are isogenous.
\item 
The characteristic polynomials of $\pi$ and $\pi'$ acting on the $v$-adic Tate modules of $\ulM$, respectively $\ulM'$, coincide for some (any) place $v\in \Spec A$.
\item 
There exists an isomorphism of $Q$-algebras $\QEnd(\ulM)\cong\QEnd(\ulM')$ mapping $\pi$ to $\pi'$.
\item 
$Z_\ulM=Z_{\ulM'}$.
\end{enumerate}
}

In the last section we sketch a few results for the question, which orders of $\QEnd(\ulM)$ occur as the endomorphism rings of pure Anderson motives (\ref{ThmW3.13}, \ref{ThmW6.1}). There is a relation between the breaking up of the isogeny class of a semisimple pure Anderson motive into isomorphism classes, and the arithmetic of $\QEnd(\ulM)$. We indicate this by treating the case of pure Anderson motives defined over the minimal field $\Fq$. In this case $\QEnd(\ulM)$ is commutative (\ref{ThmW6.1}).
Many of our results parallel Tate's celebrated article \cite{Tat} on abelian varieties over finite fields. To prove them, a major tool are the Tate modules and local shtuka attached to pure Anderson motives, which we recall in Sections~\ref{SectTateModules} and \ref{SectLS}, and the analogue~\cite{Taguchi95b,Tam} of Tate's conjecture on endomorphisms. 
These local structures behave like in the classical case of abelian varieties, local shtuka playing the role of the $p$-divisible groups of the abelian varieties. The only difference is that $p$-divisible groups are only useful for abelian varieties in characteristic $p$, whereas the local shtuka at any place of $Q$ are important for the investigation of abelian $\tau$-sheaves and pure Anderson motives. One of the aims of this article is to demonstrate the utility of local shtuka. For instance we apply them in the computation of the Hasse invariants of $\QEnd(\ulM)$ in Theorem~\ref{THEOREM-2}. We also used them in \cite{BH_A} to reprove the standard fact that the set of morphisms between two pure Anderson motives is a projective $A$-module (\ref{ThmT.3}). 
Scattered in the text are several interesting examples displaying various phenomena 
(\ref{Ex3.1}, \ref{Ex3.10c}, \ref{LAST-EXAMPLE}, \ref{Ex3.15}).
Note that there is a two in one version \cite{BH} of the present article and \cite{BH_A} on the arXiv.

\setcounter{tocdepth}{1}
\tableofcontents

\section*{Notation}

In this article we denote by

\vspace{2mm}
\noindent
\begin{tabular}{@{}p{0.25\textwidth}@{}p{0.75\textwidth}@{}}
$\Fq$& the finite field with $q$ elements and characteristic $p$, \\
$C$& a smooth projective geometrically irreducible curve over $\Fq$, \\
$\infty\in C(\Fq)$& a fixed $\Fq$-rational point on $C$, \\
$A = \Gamma(C\setminus\{\infty\},\OC)$& the ring of regular functions on $C$ outside $\infty$, \\
$Q = \Fq(C) =\Quot(A)$& the function field of $C$, \\
$Q_v$& the completion of $Q$ at the place $v\in C$, \\
$A_v$& the ring of integers in $Q_v$. For $v\ne\infty$ it is the completion of $A$ at $v$.\\
$\BF_v$ & the residue field of $A_v$. In particular $\BF_\infty\cong\Fq$.
\end{tabular}

\vspace{2mm}
\noindent
For a field $L$ containing $\Fq$ we write

\vspace{2mm}
\noindent
\begin{tabular}{@{}p{0.25\textwidth}@{}p{0.75\textwidth}@{}}
$C_L=C\times_{\Spec\Fq}\Spec L$,\\[1mm]
$A_L=A\otimes_\Fq L$,\\[1mm]
$Q_L=Q\otimes_\Fq L$,\\[1mm]
$A_{v,L}=A_v\wh\otimes_\Fq L$ & for the completion of $\O_{C_L}$ at the closed subscheme $v\times\Spec L$,\\[1mm]
$Q_{v,L}=A_{v,L}[\frac{1}{v}]$.& Note that this is not a field if $\BF_v\cap L\supsetneq\Fq$,\\[1mm]
\end{tabular}
\begin{tabular}{@{}p{0.25\textwidth}@{}p{0.75\textwidth}@{}}
${\rm Frob}_q:L\to L$ & for the $q$-Frobenius endomorphism mapping $x$ to $x^q$,\\[1mm]
$\sigma = \id_C\times\Spec({\rm Frob}_q)$& for the endomorphism of $C_L$ which acts as the identity on the points and on $\O_C$ and as the $q$-Frobenius on $L$,\\
$\s$ & for the endomorphisms induced by $\sigma$ on all the above rings. For instance $\s(a\otimes b)=a\otimes b^q$ for $a\in A$ and $b\in L$.\\
$\s M=M\otimes_{A_L,\s}A_L$ & for an $A_L$-module $M$ and similarly for the other rings.
\end{tabular}

\forget{
\vspace{2mm}
\noindent
All schemes, as well as their products and morphisms, are supposed to be over $\Spec\Fq$. 
Let $S$ be a scheme. We denote by

\vspace{2mm}
\noindent
\begin{tabular}{@{}p{0.25\textwidth}@{}p{0.75\textwidth}@{}}
$\sigma_S: S\rightarrow S$& its $q$-Frobenius endomorphism which acts identically on points of $S$ and as the $q$-power map on the structure sheaf $\OS$, \\
$C_S = C\times S$& the base extension of $C$ from $\Spec\Fq$ to $S$, \\
$\sigma = \id_C\times\sigma_S$& the endomorphism on $C_S$ which acts identically on $C$ and as the $q$-Frobenius on $S$. 
\end{tabular}
}

\vspace{2mm}
\noindent
For a divisor $D$ on $C$ we denote by $\O_{C_L}(D)$ the invertible sheaf on $C_L$ whose sections $\varphi$ have divisor $(\varphi)\ge-D$. 
For a coherent sheaf $\F$ on $C_L$ we set $\F(D) := \F\otimes_{\O_{C_L}}\O_{C_L}(D)$. This notation applies in particular to the divisor $D = n\cdot\infty$ for $n\in\Z$.

\vspace{2mm}
\noindent
We will fix the further notation $\pi,F,E,\mu_\pi,\pi_v,F_v,E_v$, and $\chi_v$
in formula (\ref{Eq3.1}) on page~\pageref{Eq3.1}.

\bigskip


\section{Pure Anderson Motives and Abelian $\tau$-Sheaves}
\setcounter{equation}{0}

Pure Anderson motives were introduced by G.\ Anderson~\cite{Anderson} under the name \emph{pure $t$-motives} in the case where $A=\BF_p[t]$. They were further studied in \cite{BH_A}. To give their definition let $L$ be a field extension of $\Fq$ and fix an $\Fq$-homomorphism $c^\ast:A\to L$. Let $J\subset A_L$ be the ideal generated by $a\otimes 1-1\otimes c^\ast(a)$ for all $a\in A$. 

\begin{Definition}[pure Anderson motives] \label{Def1'.1}
A \emph{pure Anderson motive $\ulM=(M,\tau)$ of rank $r$, dimension $d$, and characteristic $c^\ast$ over $L$} consists of a locally free $A_L$-module $M$ of rank $r$ and an injective $A_L$-module homomorphism $\tau:\s M\to M$ such that
\begin{suchthat}
\item 
the cokernel of $\t$ is an $L$-vector space of dimension $d$ and annihilated by $J^d$, and 
\item 
$M$ extends to a locally free sheaf $\CM$ of rank $r$ on $C_L$ such that for some positive integers $k,l$ the map $\tau^l:=\tau\circ\s(\tau)\circ\ldots\circ(\s)^{l-1}(\tau):(\s)^l M\to M$ induces an isomorphism $(\s)^l\CM_\infty\to\CM(k\cdot\infty)_\infty$ of the stalks at $\infty$. 
\end{suchthat}
We call $\chr:=\ker c^\ast\in \Spec A$ the \emph{characteristic point} of $\ulM$ and we say that $\ulM$ has \emph{finite characteristic} (respectively \emph{generic characteristic}) if $\chr$ is a closed (respectively the generic) point. The ratio $\weight(M,\tau):=\frac{k}{l}$ equals $\frac{d}{r}$ and is called the \emph{weight} of $(M,\tau)$; see~\cite[\BHAss]{BH_A}.
\end{Definition}

\begin{Definition} (Compare \cite[4.5]{PT})\label{Def1'.2}
\begin{suchthat}
\item 
A \emph{morphism} $f:(M,\tau)\to (M',\tau')$ between Anderson motives of the same characteristic $c^\ast$ is a morphism $f:M\to M'$ of $A_L$-modules which satisfies $f\circ\tau=\tau'\circ\s(f)$.
\item 
If $f:\ulM\to\ulM'$ is surjective, $\ulM'$ is called a \emph{factor motive} of $\ulM$.
\item 
A morphism $f:\ulM\to \ulM'$ is called an \emph{isogeny} if $f$ is injective with torsion cokernel.
\item 
An isogeny is called \emph{separable} (respectively \emph{purely inseparable})
if the induced morphism \\
$\t:\s\coker f\to\coker f$ is an isomorphism (respectively is \emph{nilpotent}, that is, if for some $n$ the morphism $\t\circ\s\t\circ\ldots\circ(\s)^n\t$ is zero).
\end{suchthat}
We denote the set of morphisms between $\ulM$ and $\ulM'$ by $\Hom(\ulM,\ulM')$. It is an $A$-module.
\end{Definition}

If $\ulM$ and $\ulM'$ are pure Anderson motives of different weights then $\Hom(\ulM,\ulM')=\{0\}$ by \cite[\BHApp]{BH_A}. This justifies the terminology \emph{pure}.
The following fact is well known. A proof can be found for instance in \cite[\BHAqq]{BH_A}.

\begin{Theorem}  \label{ThmT.3}\label{CorT.4}
Let $\ulM$ and $\ulM'$ be pure Anderson motives over an arbitrary field $L$. Then $\Hom(\ulM,\ulM')$ is a projective $A$-module of rank $\le rr'$.
The minimal polynomial of every endomorphism of a pure Anderson motive $\ulM$ lies in $A[x]$. \qed
\end{Theorem}

\begin{Corollary} (\cite[\BHArr]{BH_A}) \label{Cor1.11b}
Let $f:\ulM\to\ulM'$ be an isogeny between pure Anderson motives. Then
\begin{suchthat}
\item
there exists an element $a\in A$ which annihilates $\coker f$,
\item 
there exists a dual isogeny $\dual{f}:\ulM'\to\ulM$ such that $f\circ\dual{f}=a\cdot\id_{\ulM'}$ and $\dual{f}\circ f=a\cdot\id_\ulM$.
\end{suchthat}
\end{Corollary}

Next we come to the notion of abelian ($\tau$-)sheaves. It was introduced in~\cite{Hl} in order to construct moduli spaces for pure Anderson motives. We briefly recall the results from \cite{BH_A} on the relation between pure Anderson motives and abelian $\tau$-sheaves. Although our primary interest is on pure Anderson motives we present abelian $\tau$-sheaves here because they can have characteristic point $\chr=\infty\in C$ in contrast to pure Anderson motives, and many results for the later extend to this more general situation. Moreover, some results are proved most naturally via the use of abelian $\tau$-sheaves (e.g.\ \ref{THEOREM-2} and \ref{Prop3.4.1} below). The fact that $\chr=\infty$ is allowed for abelian $\tau$-sheaves was crucial for the uniformization of the moduli spaces of pure Anderson motives in \cite{Hl} and the derived consequences on analytic uniformization of pure Anderson motives.
Let $L\supset\Fq$ be a field and fix a morphism $c: \Spec L\rightarrow C$. Let $\I$ be the ideal sheaf on $C_L$ of the graph of $c$. Let $r$ and $d$ be non-negative integers.

\begin{Definition}[Abelian $\tau$-sheaf]\label{Def1.1}
An \emph{abelian $\tau$-sheaf $\FF = (\F_i,\P_i,\t_i)$ of rank $r$, dimension $d$ and characteristic $c$ over $L$} is a collection of locally free sheaves $\F_i$ on $C_L$ of rank $r$ together with injective morphisms $\P_i,\t_i$ of\/ $\O_{C_L}$-modules $(i\in\Z)$ of the form
\[
\begin{CD}
\cdots @>>> 
\F_{i-1} @>{\P_{i-1}}>> 
\F_{i}   @>{\P_{i}}>> 
\F_{i+1} @>{\P_{i+1}}>> 
\cdots 
\\ & & 
@AA{\t_{i-2}}A 
@AA{\t_{i-1}}A 
@AA{\t_{i}}A 
\\
\cdots @>>> 
\s\F_{i-2} @>{\s\P_{i-2}}>> 
\s\F_{i-1} @>{\s\P_{i-1}}>> 
\s\F_{i}   @>{\s\P_{i}}>> 
\cdots 
\\
\end{CD}
\]
subject to the following conditions:
\begin{suchthat}
\item the above diagram is commutative,
\item there exist integers\/ $k,l>0$ with\/ $ld=kr$ such that the morphism $\P_{i+l-1}\circ\cdots\circ\P_i$ identifies $\F_i$ with the subsheaf\/ $\F_{i+l}(-k\cdot\infty)$ of\/ $\F_{i+l}$ for all $i\in\Z$,
\item $\coker\P_i$ is an $L$-vector space of dimension $d$ for all $i\in\Z$,
\item $\coker\t_i$ is an $L$-vector space of dimension $d$ and annihilated by $\I^d$ for all $i\in\Z$.
\end{suchthat}
We call $\chr:=c(\Spec L)\in C$ the \emph{characteristic point} and say that $\FF$ has \emph{finite} (respectively \emph{generic}) \emph{characteristic} if $\chr$ is a closed (respectively the generic) point. 
If $r\ne0$ we call\/ $\weight(\FF):=\frac{d}{r}$ the \emph{weight of $\FF$}.
\end{Definition}

\noindent
{\it Remark.} \label{PROP.2}
1. If $\FF$ is an abelian $\tau$-sheaf and $D$ a divisor on $C$, then $\FF(D) := (\F_i(D),\P_i\otimesidOCL{D},\t_i\otimesidOCL{D})$
is an abelian $\tau$-sheaf of the same rank and dimension as $\FF$.

2. Let $\FF$ be an abelian $\tau$-sheaf and let $n\in\Z$. We denote by $\FF\II{n} := (\F_{i+n},\P_{i+n},\t_{i+n})$
the \emph{$n$-shifted abelian $\tau$-sheaf of\/ $\FF$} whose collection of\/ $\F$'s, $\P$'s and $\t$'s is just shifted by\/ $n$.

\medskip

\begin{Definition}
A \emph{morphism} $f$ between two abelian $\tau$-sheaves $\FF = (\F_i,\P_i,\t_i)$ and $\FF' = (\F'_i,\P'_i,\t'_i)$ of the same characteristic $c:\Spec L\to C$ is a collection of morphisms $f_i: \F_i \rightarrow \F'_i$ $(i\in\Z)$ which commute with the $\P$'s and the $\t$'s, that is, $f_{i+1}\circ\P_i=\P'_i\circ f_i$ and $f_{i+1}\circ\t_i=\t'_i\circ\s f_i$. We denote the set of morphisms between $\FF$ and $\FF'$ by $\Hom(\FF,\FF')$. It is an $\Fq$-vector space.
\end{Definition}

For example, the collection of morphisms $(\P_i):\, \FF\rightarrow\FF\II{1}$ defines a morphism between the abelian $\tau$-sheaves $\FF$ and $\FF\II{1}$. 

\begin{Definition}
Let $\FF$ and $\FF'$ be abelian $\tau$-sheaves and let $f \in \Hom(\FF,\FF')$ be a morphism. 
Then $f$ is called \emph{injective} (respectively \emph{surjective}, respectively an \emph{isomorphism}), if\/ $f_i$ is injective (respectively surjective, respectively bijective) for all\/ $i\in\Z$.
We call $\FF$ an \emph{abelian factor $\tau$-sheaf} of\/ $\FF'$, if there is a surjective morphism from $\FF'$ onto $\FF$.
\end{Definition}

If $\FF=(\F_i,\P_i,\t_i)$ is an abelian $\tau$-sheaf of rank $r$, dimension $d$, and characteristic $c:\Spec L\to C$ with $\chr=\im c\ne\infty$ then
\begin{equation}\label{Eq1.1}
\ulM(\FF)\;:=\;(M,\tau)\;:=\;\Bigl(\Gamma(C_L\setminus\{\infty\},\F_0)\,,\,\P_0^{-1}\circ\t_0\Bigr)
\end{equation}
is a pure Anderson motive of the same rank and dimension and of characteristic $c^\ast:A\to L$. Conversely we have the following result.

\begin{Proposition} (\cite[\BHAoo]{BH_A}) \label{PropX.1}
Let $(M,\tau)$ be a pure Anderson motive of rank $r$, dimension $d$, and characteristic $c^\ast:A\to L$ over $L$. Then $(M,\tau)=\ulM(\FF)$ for an abelian $\tau$-sheaf $\FF$ over $L$ of same rank and dimension with characteristic $c:=\Spec c^\ast:\Spec L\to \Spec A\subset C$. One can even find the abelian $\tau$-sheaf $\FF$ with $k,l$ relatively prime.
\end{Proposition}

\bigskip

\section{Isogenies and Quasi-Isogenies}
\setcounter{equation}{0}

We recall the basic facts about isogenies from \cite{BH_A}.

\begin{Proposition} (\cite[\BHAnn]{BH_A}) \label{PROP.1.42A}
  Let $f:\FF\to\FF'$ be a morphism between two abelian $\tau$-sheaves $\FF = (\F_i,\P_i,\t_i)$ and\/ $\FF' = (\F'_i,\P'_i,\t'_i)$. Then the following assertions are equivalent:
\begin{suchthat}
\item 
$f$ is injective and the support of all\/ $\coker f_i$ is contained in $D\times\Spec L$ for a finite closed subscheme $D\subset C$,
\item 
$f$ is injective and $\FF$ and $\FF'$ have the same rank and dimension,
\item 
$\FF$ and $\FF'$ have the same weight and the fiber $f_{i,\eta}$ at the generic point $\eta$ of $C_L$ is an isomorphism for some (any) $i\in\Z$.
\end{suchthat}
\end{Proposition}

\begin{Definition}
\begin{suchthat}
\item 
A morphism $f:\FF\to \FF'$ satisfying the equivalent conditions of Proposition~\ref{PROP.1.42A} is called an \emph{isogeny}.
We denote the set of isogenies between $\FF$ and $\FF'$ by $\Isog(\FF,\FF')$.
\item 
An isogeny $f:\FF\to\FF'$ is called \emph{separable} (respectively \emph{purely inseparable}) if for all $i$ the induced morphism $\t_i:\s\coker f_i\to\coker f_{i+1}$ is an isomorphism (respectively is \emph{nilpotent}, that is, $\t_{i}\circ\s\t_{i-1}\circ\ldots\circ(\s)^n\t_{i-n}=0$ for some $n$).
\end{suchthat}
\end{Definition}

The endomorphism rings of abelian $\tau$-sheaves are finite rings. But if we allow the (endo-)morphisms to have ``poles'' we get rings which are related to the endomorphism rings of the associated pure Anderson motives. We make the following:

\begin{Definition}[Quasi-morphism and quasi-isogeny]
Let $\FF$ and $\FF'$ be abelian $\tau$-sheaves.
\begin{suchthat}
\item A \emph{quasi-morphism} $f$ between $\FF$ and $\FF'$ is a morphism $f\in\Hom(\FF,\FF'(D))$ for some effective divisor $D$ on $C$.
\item A \emph{quasi-isogeny} $f$ between $\FF$ and $\FF'$ is an isogeny $f\in\Isog(\FF,\FF'(D))$ for some effective divisor $D$ on $C$. 
\end{suchthat}
\end{Definition}

If $D_1\le D_2$ the composition with the inclusion isogeny $\FF'(D_1)\subset\FF'(D_2)$ defines an injection $\Hom(\FF,\FF'(D_1))\hookrightarrow\Hom(\FF,\FF'(D_2))$. This yields an equivalence relation for quasi-morphisms and quasi-isogenies. We let $\QHom(\FF,\FF')$ and $\QIsog(\FF,\FF')$ be the set of quasi-morphisms, respectively quasi-isogenies, between $\FF$ and $\FF'$ modulo this equivalence relation. 
We write $\QEnd(\FF):=\QHom(\FF,\FF)$ and $\QIsog(\FF):=\QIsog(\FF,\FF)$.

The $Q$-vector spaces $\QHom(\FF,\FF')$ and $\QEnd(\FF)$ are finite dimensional, and 
$\QIsog(\FF)$ is the group of units in the $Q$-algebra $\QEnd(\FF)$, see~\cite[\BHAllmm]{BH_A}.

Two abelian $\tau$-sheaves $\FF$ and $\FF'$ are called \emph{quasi-isogenous} (notation: $\FF\approx\FF'$), if there exists a quasi-isogeny between $\FF$ and $\FF'$.
The relation $\approx$ is an equivalence relation. If $\FF\approx\FF'$, then 
 the $Q$-algebras $\QEnd(\FF)$ and $\QEnd(\FF')$ are isomorphic, and $\QHom(\FF,\FF')$ is free of rank $1$ both as a left module over $\QEnd(\FF')$ and as a right module over $\QEnd(\FF)$.

\begin{Proposition} (\cite[\BHAkk]{BH_A}) \label{CONNECTION}
Let $\FF$ and $\FF'$ be two abelian $\tau$-sheaves of characteristic $\chr\ne\infty$ and let $\ulM(\FF)$ and $\ulM(\FF')$ be their associated pure Anderson motives. Then there is a canonical isomorphism of $Q$-vector spaces
\[
\QHom(\FF,\FF') \;=\; \Hom(\ulM(\FF),\ulM(\FF'))\otimes_A Q\,.
\]
\end{Proposition}

If $\ulM$ and $\ulM'$ are pure Anderson motives, then the elements of $\Hom(\ulM,\ulM')\otimes_A Q$ which admit an inverse in $\Hom(\ulM',\ulM)\otimes_A Q$ are called \emph{quasi-isogenies}. With this definition we can reformulate Propositions~\ref{PropX.1} and \ref{CONNECTION} as follows.

\begin{Corollary}\label{Cor2.9d}
Let $\chr\ne\infty$. Then the functor $\FF\mapsto\ulM(\FF)$ defines an equivalence of categories between 
\begin{suchthat}
\item 
the category with abelian $\tau$-sheaves as objects and with $\QHom(\FF,\FF')$ as the set of morphisms,
\item 
 and the category with pure Anderson motives as objects and with $\Hom(\ulM,\ulM')\otimes_A Q$ as the set of morphisms.
\end{suchthat}
We call these the \emph{quasi-isogeny categories} of abelian $\tau$-sheaves of characteristic different from $\infty$ and of pure Anderson motives, respectively. \qed
\end{Corollary}

\begin{Definition} \label{Def2.6}
Let $\FF$ be an abelian $\tau$-sheaf.
\begin{suchthat}
\item $\FF$ is called \emph{simple}, if\/ $\FF\not=\ZZ$ and\/ $\FF$ has no abelian factor $\tau$-sheaves other than $\ZZ$ and\/ $\FF$.
\item $\FF$ is called \emph{semisimple}, if\/ $\FF$ admits, up to quasi-isogeny, a decomposition into a direct sum $\FF\approx\FF_1\oplus\cdots\oplus\FF_n$ of simple abelian $\tau$-sheaves $\FF_j$ $(1\le j\le n)$.
\end{suchthat}
We make the same definition for a pure Anderson motive.
\end{Definition}

\noindent
{\it Remark.} \label{Prop1.29a}
1. Let $\FF$ be an abelian $\tau$-sheaf with characteristic different from $\infty$. Then $\FF$ is \mbox{(semi-)}simple if and only if the pure Anderson motive $\ulM(\FF)$ is (semi-)simple by \cite[\BHAjj]{BH_A}.

\smallskip

\noindent
2. It is not sensible to try defining \emph{simple} pure Anderson motives via sub-motives, since for example $a\ulM\subset\ulM$ is a proper sub-motive for any $a\in A\setminus A^{\SSC\times}$. 
This shows that pure Anderson motives behave dually to abelian varieties. Namely an abelian variety is called simple if it has no non-trivial abelian subvarieties.

\begin{Theorem} (\cite[\BHAii]{BH_A}) \label{QEND-DIVISION-MATRIX}
Let $\FF$ be an abelian $\tau$-sheaf of characteristic different from $\infty$.
\begin{suchthat}
\item If\/ $\FF$ is simple, then $\QEnd(\FF)$ is a division algebra over $Q$.
\item If\/ $\FF$ is semisimple with decomposition $\FF\approx\FF_1\oplus\cdots\oplus\FF_n$ up to quasi-isogeny into simple abelian $\tau$-sheaves $\FF_j$, then $\QEnd(\FF)$ decomposes into a finite direct sum of full matrix algebras over the division algebras $\QEnd(\FF_j)$ over $Q$.
\end{suchthat}
\end{Theorem}

\medskip

In the following we want to define the \emph{degree} of an isogeny which should be an ideal of $A$ since in the function field case we have substituted $A$ for $\Z$.
Let $f:\ulM\to\ulM'$ be an isogeny between pure Anderson motives. Then the $A_L$-module $\coker f$ is a finite $L$-vector space equipped with a morphism of $A_L$-modules $\t':\s\coker f\to\coker f$. Since $\coker f$ is annihilated by an element of $A$ it decomposes by the Chinese remainder theorem
\[
(\coker f,\t')\;=\;\bigoplus_{v\in\supp(\coker f)}(\coker f,\t')\otimes_A A_v\;=:\;\bigoplus_{v\in\supp(\coker f)}\ulK_v\,.
\]
If $v\ne\chr$ the morphism $\t'$ on $\ulK_v$ is an isomorphism and so Lang's theorem implies that
\[
(\ulK_v\otimes_L L^\sep)^\t\otimes_{\Fq} L^\sep \isoto \ulK_v\otimes_L L^\sep
\]
is an isomorphism; see for instance \cite[Lemma 1.8.2]{Anderson}. In particular
\[
[\BF_v:\Fq]\cdot\dim_{\BF_v}(\ulK_v\otimes_L L^\sep)^\t\;=\;\dim_\Fq(\ulK_v\otimes_L L^\sep)^\t\;=\;\dim_{L^\sep}(\ulK_v\otimes_L L^\sep)\;=\;\dim_L \ulK_v\,.
\]

\smallskip

On the other hand if the characteristic is finite and $v=\chr$, the characteristic morphism $c^\ast:A\to L$ yields $\BF_\chr\subset L$ and determines the distinguished prime ideal 
\[
\Fa_0\;:=\;(b\otimes 1-1\otimes c^\ast(b):b\in\BF_\chr)\subset A_{\chr,L}\,. 
\]
If we set $n:=[\BF_\chr:\Fq]$ and $\Fa_i:=(\s)^i\Fa_0=(b\otimes 1-1\otimes c^\ast(b)^{q^i}:b\in\BF_\chr)$, then we can decompose $A_{\chr,L}=\bigoplus_{i\in\Z/n\Z}A_{\chr,L}/\Fa_i$ and $\t$ is an isomorphism
\[
\s(\ulK_\chr/\Fa_{i-1}\ulK_\chr)\isoto \ulK_\chr/\Fa_i
\]
for $i\ne0$ since $\t$ is an isomorphism on $\ulM$ and $\ulM'$ outside the graph of $c^\ast$. 
(This argument will be used again in Proposition~\ref{PropLS4}.)
 In particular 
\[
[\BF_\chr:\Fq]\cdot\dim_L(\ulK_\chr/\Fa_0\ulK_\chr)\;=\;\dim_L\ulK_\chr\,.
\]

\begin{Definition}\label{Def1.7.6}
We assign to the isogeny $f$ the ideal
\[
\deg(f)\;:=\;\prod_{v\in\supp(\coker f)}v^{(\dim_L\ulK_v)/[\BF_v:\Fq]}\;=\;\chr^{\dim_L(\ulK_\chr/\Fa_0\ulK_\chr)}\cdot\prod_{v\ne\chr}v^{\dim_{\BF_v}(\ulK_v\otimes_L L^\sep)^\t}
\]
of $A$ and call it the \emph{degree of $f$}. We call $\chr^{\dim_L(\ulK_\chr/\Fa_0\ulK_\chr)}$ the \emph{inseparability degree of $f$} and $\prod_{v\ne\chr}v^{\dim_{\BF_v}(\ulK_v\otimes_L L^\sep)^\t}$ the \emph{separability degree of $f$}.
\end{Definition}

\noindent
{\it Remark.} The separability degree of $f$ is the Euler-Poincar\'e characteristic $EP\bigl(\bigoplus_{v\ne\chr}\ulK_v\otimes_L L^\sep\bigr)^\t$; see Gekeler~\cite[3.9]{Gekeler} or Pink and Traulsen~\cite[4.6]{PT}. Recall that the Euler-Poincar\'e characteristic of a finite torsion $A$-module is the ideal of $A$ defined by requiring that $EP$ is multiplicative in short exact sequences, and that $EP(A/v):=v$ for any maximal ideal $v$ of $A$.

\begin{Lemma}\label{Lemma1.7.7}
\begin{suchthat}
\item 
If $f:\ulM\to\ulM'$ and $g:\ulM'\to\ulM''$ are isogenies then $\deg(gf)=\deg(f)\cdot\deg(g)$.
\item 
$\dim_{\Fq}A/\deg(f)\;=\;\dim_L\coker f$.
\end{suchthat}
\end{Lemma}

\begin{proof}
1 is immediate from the short exact sequence
\[
\xymatrix{ 0 \ar[r] & \coker f \ar[r]^g & \coker(gf) \ar[r] & \coker g \ar[r] & 0 }
\]
and 2 is obvious.
\end{proof}

\begin{Proposition}\label{Prop3.28a}
The ideal $\deg(f)$ annihilates $\coker f$.
\end{Proposition}

\begin{proof}
If $v=\chr$ and $a$ is a uniformizer at $\chr$, then multiplication with $a$ is nilpotent on the $L$-vector space $\ulK_\chr/\Fa_0\ulK_\chr$. In particular $a^{\dim_L(\ulK_\chr/\Fa_0\ulK_\chr)}$ annihilates $\ulK_\chr/\Fa_0\ulK_\chr$, and hence also $\ulK_\chr$.

If $v\ne\chr$ and $a$ is a uniformizer at $v$, we obtain analogously that $a^{\dim_{\BF_v}(\ulK_v\otimes_L L^\sep)^\t}$ annihilates the $\BF_v$-vector space $(\ulK_v\otimes_L L^\sep)^\t$ and therefore also the $L$-vector space $\ulK_v$.
\end{proof}

\begin{Proposition}\label{Prop1.7.8}
Let $f:\ulM\to \ulM'$ be an isogeny such that $\deg(f)=aA$ is principal (for example this is the case if $C=\PP^1$ and $A=\Fq[t]$). Then there is a uniquely determined dual isogeny $\dual{f}:\ulM'\to \ulM$ (depending on $a$), which satisfies $f\circ\dual{f}=a\cdot\id_{\ulM'}$ and $\dual{f}\circ f=a\cdot\id_\ulM$.
\end{Proposition}

\begin{proof}
Since $\deg(f)$ annihilates $\coker f$ the proposition is immediate.
\end{proof}

In Theorem~\ref{Prop3.4.1} we will see that any isogeny $f\in\End(\ulM)$ of a semisimple pure Anderson motive over a finite field satisfies the assumption that $\deg(f)$ is principal.

\bigskip


\section{Local Shtuka}\label{SectLS}
\setcounter{equation}{0}

There are mainly two local structures which one can attach to pure Anderson motives and abelian $\tau$-sheaves, namely the \emph{local (iso-)shtuka} and the \emph{Tate module}. We treat the Tate module in the next section. The local (iso-)shtuka is the analogue of the Dieudonn\'e module of the $p$-divisible group attached to an abelian variety. Note however one fundamental difference. While the Dieudonn\'e module exists only if $p$ equals the characteristic of the base field, there is no such restriction in our theory here. And in fact this would even allow to dispense with Tate modules at all and only work with local (iso-)shtuka. Local (iso-)shtuka were introduced in \cite{Hl} under the name \emph{Dieudonn\'e $\Fq\dbl z\dbr$-modules} (respectively \emph{Dieudonn\'e $\Fq\dpl z\dpr$-modules}). They are studied in \cite{Anderson2,HartlPSp,Laumon}. Over a field their definition takes the following form.

\begin{Definition}\label{DefLS1}
Let $v\in C$ be a place of $Q$ and let $L\supset\Fq$ be a field. An \emph{(effective) local $\sigma$-shtuka} at $v$ of rank $r$ over $L$ is a pair $\ulHM=(\hat M,\phi)$ consisting of a free $A_{v,L}$-module $\hat M$ of rank $r$ and an injective $A_{v,L}$-module homomorphism \mbox{$\phi:\s\hat M\to \hat M$}.

A \emph{local $\sigma$-isoshtuka} at $v$ of rank $r$ over $L$ is a pair $\ulHN=(\hat N,\phi)$ consisting of a free $Q_{v,L}$-module $\hat N$ of rank $r$ and an isomorphism $\phi:\s\hat N\isoto\hat N$ of $Q_{v,L}$-modules.
\end{Definition}

\begin{Remark}
Note that so far in the literature \cite{Anderson2,Hl,HartlPSp,Laumon} it is always assumed that $A_v$ has residue field $\Fq$, the fixed field of $\sigma$ on $L$. So in particular $A_{v,L}$ is an integral domain and $Q_{v,L}$ is a field. For applications to pure Anderson motives this is not a problem since we may reduce to this case by Propositions~\ref{PropLS3} and \ref{PropLS4} below.
\end{Remark}

\begin{Definition}\label{DefLS1b}
A local shtuka $\ulHM=(\hat M,\phi)$ is called \emph{\'etale} if $\phi$ is an isomorphism. The \emph{Tate module} of an \'etale local $\sigma$-shtuka $\ulHM$ at $v$ is the $G:=\Gal(L^\sep/L)$-module of $\phi$-invariants
\[
T_v\ulHM\;:=\;\bigl(\ulHM\otimes_{A_{v,L}}A_{v,L^\sep}\bigr)^\phi\,.
\]
The \emph{rational Tate module} of $\ulHM$ is the $G$-module
\[
V_v\ulHM\;:=\;T_v\ulHM\otimes_{A_v}Q_v\,.
\]
\end{Definition}

It follows from \cite[Proposition~6.1]{TW} that $T_v\ulHM$ is a free $A_v$-module of the same rank than $\ulHM$ and that the natural morphism
\[
T_v\ulHM\otimes_{A_v}A_{v,L^\sep}\isoto \ulHM\otimes_{A_{v,L}}A_{v,L^\sep}
\]
is a $G$- and $\phi$-equivariant isomorphism of $A_{v,L^\sep}$-modules, where on the left module $G$ acts on both factors and $\phi$ is $\id\otimes\s$. Since $(L^\sep)^{G}=L$ we obtain:

\begin{Proposition}\label{Prop2.13'}
Let $\ulHM$ and $\ulHM{}'$ be \emph{\'etale} local $\sigma$-shtuka at $v$ over $L$. Then
\begin{suchthat}
\item 
$\DS\ulHM\;=\;(T_v\ulHM\otimes_{A_v}A_{v,L^\sep})^{G}$, the Galois invariants,
\item 
$\DS\Hom_{A_{v,L}[\phi]}(\ulHM,\ulHM{}')\;\isoto\;\Hom_{A_v[G]}(T_v\ulHM,T_v\ulHM{}')\;,\es f\mapsto T_vf$ is an isomorphism.
\end{suchthat}
In particular the Tate module functor yields an equivalence of the category of \'etale local shtuka at $v$ over $L$ with the category of $A_v[G]$-modules, which are finite free over $A_v$. \qed
\end{Proposition}

\begin{proof}
1 and 2 are immediate. Hence clearly the Tate module functor is fully faithful. That it is an equivalence follows analogously to \cite[Proposition 4.1.1]{Katz1}.
\end{proof}

\bigskip

If the residue field $\BF_v$ of $v$ is larger than $\Fq$ one has to be a bit careful with local (iso-)shtuka since $Q_{v,L}$ is then in general not a field. Namely let $\#\BF_v=q^n$ and let $\BF_{q^f}:=\{\alpha\in L:\alpha^{q^n}=\alpha\}$ be the ``intersection'' of $\BF_v$ with $L$. Then
\[
\BF_v\otimes_\Fq L = \prod_{\Gal(\BF_{q^f}/\Fq)}\BF_v\otimes_{\BF_{q^f}}L = 
\prod_{i\in\Z/f\Z}\BF_v\otimes_\Fq L\,/\,(b\otimes 1-1\otimes b^{q^i}:b\in \BF_{q^f})
\]
and $\s$ transports the $i$-th factor to the $(i+1)$-th factor. (Of course, the indexing of the factors depends on a choice of embeddings $\BF_{q^f}\subset\BF_v$ and $\BF_{q^f}\subset L$.) Denote by $\Fa_i$ the ideal of $A_{v,L}$ (or $Q_{v,L}$) generated by $\{b\otimes 1-1\otimes b^{q^i}:b\in \BF_{q^f}\}$. Then
\[
A_{v,L} = \prod_{\Gal(\BF_{q^f}/\Fq)}A_v\wh\otimes_{\BF_{q^f}}L = 
\prod_{i\in\Z/f\Z}A_{v,L}\,/\,\Fa_i
\]
and similarly for $Q_v$.
Note that the factors in this decomposition and the ideals $\Fa_i$ correspond precisely to the places $v_i$ of $C_{\BF_{q^f}}$ lying above $v$.

\begin{Proposition}\label{PropLS3}
Fix an $i$. The reduction modulo $\Fa_i$ induces equivalences of categories
\begin{suchthat}
\item
$\DS (\hat N,\phi)\longmapsto \bigl(\hat N/\Fa_i\hat N\,,\,\phi^f:(\s)^f \hat N/\Fa_i\hat N \to \hat N/\Fa_i\hat N\bigr)$\\[2mm]
between local $\sigma$-isoshtuka at $v$ over $L$ and local $\sigma^f$-isoshtuka at $v_i$ over $L$ of the same rank.
\item 
\vspace{2mm}
$\DS (\hat M,\phi)\longmapsto \bigl(\hat M/\Fa_i\hat M\,,\,\phi^f:(\s)^f \hat M/\Fa_i\hat M \to \hat M/\Fa_i\hat M\bigr)$\\[2mm]
between\/ \emph{\'etale} local $\sigma$-shtuka at $v$ over $L$ and \emph{\'etale} local $\sigma^f$-shtuka at $v_i$ over $L$ preserving Tate modules
\[
T_v(\hat M,\phi) \isoto T_{v_i}(\hat M/\Fa_i\hat M,\phi^f)\,.
\]
\end{suchthat}
\end{Proposition}

\begin{proof}
Since $\s\Fa_i=\Fa_{i+1}$ the isomorphism $\phi$ yields isomorphisms $\s(\hat N/\Fa_i\hat N)\to \hat N/\Fa_{i+1}\hat N$ and similarly for $\hat M$. These allow to reconstruct the other factors from $(\hat N/\Fa_i\hat N,\phi^f)$, and likewise for $\hat\ulM$. The isomorphism between the Tate modules follows from the observation that an element $(x_j)_{j\in\Z/f\Z}$ is $\phi$-invariant if and only if $x_{j+1}=\phi(\s x_j)$ for all $j$ and $x_i=\phi^f((\s)^f x_i)$.
\end{proof}

\noindent
{\it Remark.} The advantage of the \'etale local $\sigma^f$-shtuka at $v_i$ is that it is a free module over the integral domain $A_{v,L}/\Fa_i=A_v\wh\otimes_{\BF_{q^f}}L$, and similarly for local $\sigma^f$-isoshtuka. So the results from \cite{Anderson2,Hl,HartlPSp,Laumon} apply.

\bigskip

Now let $\FF$ be an abelian $\tau$-sheaf and $v\in C$ an arbitrary place of $Q$. We define the \emph{local $\sigma$-isoshtuka of $\FF$ at $v$} as 
\[
\ulN_v(\FF)\;:=\;\Bigl(\F_0\otimes_{\O_{C_L}}Q_{v,L}\,,\,\P_0^{-1}\circ\t_0\Bigr)\,.
\]
If $v\ne\infty$ we define the \emph{local $\sigma$-shtuka of $\FF$ at $v$} as
\[
\ulM_v(\FF)\;:=\;\Bigl(\F_0\otimes_{\O_{C_L}}A_{v,L}\,,\,\P_0^{-1}\circ\t_0\Bigr)\,.
\]
Likewise if $\ulM$ is a pure Anderson motive over $L$ and $v\in \Spec A$ we define the \emph{local $\sigma$-(iso-)shtuka of $\ulM$ at $v$} as
\[
\ulM_v(\ulM)\;:=\;\ulM\otimes_{A_L}A_{v,L} \qquad\text{respectively}\qquad \ulN_v(\ulM)\;:=\;\ulM\otimes_{A_L}Q_{v,L}\,.
\]
These local (iso-)shtuka all have rank $r$. The local shtuka are \'etale if $v\ne\chr$. Note that $\ulN_\infty(\FF)$ does not contain a local $\sigma$-shtuka if $\chr\ne\infty$, since then it is isoclinic of slope $-\weight(\FF)<0$. 

However, if $v=\infty$ the periodicity condition allows to define a different local (iso-)shtuka at $\infty$ which is of slope $\ge0$. Namely, choose a uniformizer $z$ on $C$ at $\infty$ and set $\hat M_i:=\F_i\otimes_{\O_{C_L}} A_{\infty,L}$. Recall the integers $k,l$ from Definition~\ref{Def1.1}/2 and set $\TP:=\P_{l-1}\circ\dots\circ\P_0$. We define the \emph{big local $\sigma$-shtuka of $\FF$ at $\infty$} as 
\begin{equation}\label{EQ.Phi}
\ulTM_\infty(\FF) \,:=\, \hat M_0\oplus\dots\oplus \hat M_{l-1}\qquad\text{with} 
\qquad 
\phi \,:=\, \left(
\raisebox{5ex}{$
\xymatrix @=0pc {
0 \ar@{.}[rrr] \ar@{.}[ddddrrrr] & & & 0 & **{!L !<0.8pc,0pc> =<1.5pc,1.5pc>}\objectbox{\TP^{-1}\circ z^k\t_{l-1}} \\
\t_0 \ar@{.}[dddrrr] & & & & 0 \ar@{.}[ddd]\\
0 \ar@{.}[dd] \ar@{.}[ddrr] \\
\\
0 \ar@{.}[rr] & & 0 & \t_{l-2} & 0
}$}
\qquad\qquad\right)
\end{equation}
We also define the \emph{big local $\sigma$-isoshtuka of $\FF$ at $\infty$} as
\[
\ulTN_\infty(\FF)\;:=\;\ulTM_\infty(\FF)\otimes_{A_{\infty,L}} Q_{\infty,L}\,.
\]
Both have rank $rl$ and depend on the choice of $k,l$ and $z$. If $\chr\ne\infty$ then $\ulTM_\infty(\FF)$ is \'etale. Note that $\ulTM_\infty(\FF)$ and $\ulTN_\infty(\FF)$ were used in \cite{Hl} to construct the uniformization at $\infty$ of the moduli spaces of abelian $\tau$-sheaves.

The big local (iso-)shtuka at $\infty$, $\ulTM_\infty(\FF)$ and $\ulTN_\infty(\FF)$ are always equipped with the endomorphisms 
\begin{equation} \label{EQ.Pi+Lambda}
\P \,:=\, \left(
\raisebox{5.5ex}{$
\xymatrix @=0pc {
0 \ar@{.}[rrr] \ar@{.}[ddddrrrr] & & & 0 & **{!L !<0.8pc,0pc> =<1.5pc,1.5pc>}\objectbox{\TP^{-1}\circ z^k\P_{l-1}} \\
\P_0 \ar@{.}[dddrrr] & & & & 0 \ar@{.}[ddd]\\
0 \ar@{.}[dd] \ar@{.}[ddrr] \\
\\
0 \ar@{.}[rr] & & 0 & \P_{l-2} & 0
}$}
\qquad\qquad\right) , \enspace
\Lambda(\lambda) \,:=\, \left(
\raisebox{5.5ex}{$
\xymatrix @=0pc {
\!\!\lambda \cdot\id_{M_0}\!\!  \\
& \!\!\lambda^q \cdot\id_{M_1}\!\! \ar@{.}[ddrr] \\ 
\\
& & & \!\lambda^{q^{l-1}}\cdot\id_{M_{l-1}}\!\!
}$}
\right)
\forget{
\Lambda(\lambda) \,:=\, \left(
\begin{array}{cccc}
\!\!\lambda \cdot\id_{M_0}\!\!\!\!& & & \\[2mm]
& \!\!\!\!\lambda^q \cdot\id_{M_1}\!\!\!\!& & \\
& & \!\!\!\!\ddots \!\!\!\!& \\[1mm]
& & & \!\!\!\!\lambda^{q^{l-1}}\cdot\id_{M_{l-1}}\!\!
\end{array}
\right)
}
\end{equation}
for all $\lambda\in \Ff_{q^l}\cap L$. They satisfy the relations $\P^l=z^k$ and $\P\circ\Lambda(\lambda^q)=\Lambda(\lambda)\circ\P$. We let $\Delta_\infty$ be ``the'' central division algebra over $Q_\infty$ of rank $l$ with Hasse invariant $-\frac{k}{l}$, or explicitly
\begin{equation}\label{EqDelta}
\Delta_\infty \,:=\, \Ff_{q^l}\dpl z\dpr\II{\P}\,/\,(\P^l-z^k,\, \lambda z-z\lambda,\, \P\lambda^q-\lambda\P\text{ for all }\lambda\in \Ff_{q^l})\,.
\end{equation}
If $\Ff_{q^l}\subset L$ we identify $\Delta_\infty$ with a subalgebra of $\End_{Q_{\infty,L}[\phi]}\bigl(\ulTN_\infty(\FF)\bigr)$ by mapping $\lambda\in\BF_{q^l}\subset\Delta_\infty$ to $\Lambda(\lambda)$.

\medskip

The following two results were proved in \cite[\BHAgghh]{BH_A}.

\begin{Theorem}\label{ThmLS1}
Let $\FF$ and $\FF'$ be abelian $\tau$-sheaves of the same weight over a finite field $L$ and let $v$ be an arbitrary place of $Q$. 
\begin{suchthat}
\item 
Then there is a canonical isomorphism of $Q_v$-vector spaces
\[
\QHom(\FF,\FF')\otimes_Q Q_v \isoto \Hom_{Q_{v,L}[\phi]}\bigl(\ulN_v(\FF),\ulN_v(\FF')\bigr)\,.
\]
\item 
If $v=\infty$ choose an $l$ which satisfies \ref{Def1.1}/2 for both $\FF$ and $\FF'$ and assume $\BF_{q^l}\subset L$. Then there is a canonical isomorphism of $Q_\infty$-vector spaces
\[
\QHom(\FF,\FF')\otimes_Q Q_\infty \isoto \Hom_{\Delta_\infty\wh\otimes_\Fq L[\phi]}\bigl(\ulTN_\infty(\FF),\ulTN_\infty(\FF')\bigr)\,.
\]
\end{suchthat}
\end{Theorem}

\begin{Theorem}\label{ThmLS2}
Let $\ulM$ and $\ulM'$ be pure Anderson motives over a finite field $L$ and let $v\in\Spec A$ be an arbitrary maximal ideal. Then
\[
\Hom(\ulM,\ulM')\otimes_A A_v \isoto \Hom_{A_{v,L}[\phi]}(\ulM_v(\ulM),\ulM_v(\ulM'))\,.
\]
\end{Theorem}

\medskip

Let now the characteristic be finite and $v=\chr$ be the characteristic point. Consider a pure Anderson motive $\ulM$ of characteristic $c$, its local $\sigma$-shtuka $\ulM_\chr(\ulM)=(\hat M,\phi)$ at $\chr$ and the decomposition of the later described before Proposition~\ref{PropLS3}
\[
\ulM_\chr(\ulM)=\prod_{i\in\Z/f\Z}\ulM_\chr(\ulM)/\Fa_i\ulM_\chr(\ulM)\,.
\]
From the morphism $c:\Spec L\to\Spec\BF_\chr\subset C$ we see that $\BF_\chr\subset L$, $f=[\BF_\chr:\Fq]$ and that there is a distinguished place $v_0$ of $C_{\BF_\chr}$ above $v=\chr\ne\infty$, namely the image of $c\times c:\Spec L\to C\times\Spec\BF_\chr$. Then $\phi$ has no cokernel on $\ulM_\chr(\ulM)/\Fa_i\ulM_\chr(\ulM)$ for $i\ne0$ and the reasoning of Proposition~\ref{PropLS3} yields

\begin{Proposition}\label{PropLS4}
The reduction modulo $\Fa_0$ induces an equivalence of categories
\[
\ulM_\chr(\ulM)\longmapsto \bigl(\ulM_\chr(\ulM)/\Fa_0\ulM_\chr(\ulM)\,,\,\phi^f\bigr)
\]
between the local $\sigma$-shtuka at $\chr$ associated with pure Anderson motives of characteristic $c$ and the local $\sigma^f$-shtuka at $v_0$ associated with pure Anderson motives of characteristic $c$. 
The same is true for abelian $\tau$-sheaves. \qed
\end{Proposition}

\noindent
{\it Remark.} Now the fixed field of $\sigma^f$ on $L$ equals $\BF_\chr$, the residue field of $A_\chr$. Also $\ulM_\chr(\FF)/\Fa_0\ulM_\chr(\FF)$ is a module over the integral domain $A_\chr\wh\otimes_{\BF_\chr}L$. So again \cite{Anderson2,Hl,HartlPSp,Laumon} apply to $\bigl(\ulM_\chr(\FF)/\Fa_0\ulM_\chr(\FF),\phi^f\bigr)$.

\begin{Proposition}\label{Prop2.18b}
Let $\ulM$ be a pure Anderson motive over $L$ and let $\ulHM{}'_\chr$ be a local $\sigma^f$-subshtuka of $\ulM_\chr(\ulM)/\Fa_0\ulM_\chr(\ulM)$ of the same rank. Then there is a pure Anderson motive $\ulM'$ and an isogeny $f:\ulM'\to\ulM$ with $\ulM_\chr(f)\bigl(\ulM_\chr(\ulM')/\Fa_0\ulM_\chr(\ulM')\bigr)=\ulHM{}'_\chr$. The same is true for abelian $\tau$-sheaves.
\end{Proposition}

\begin{proof}
Extend $\ulHM{}'_\chr$ to the local $\sigma$-subshtuka $\bigoplus_{i\in\Z/f\Z}\phi^i\bigl((\s)^i\ulHM{}'_\chr\bigr)$ of $\ulM_\chr(\ulM)$ and consider 
\[
\ulK\;:=\;\ulM_\chr(\ulM)\,/\,\bigoplus_{i\in\Z/f\Z}\phi^i\bigl((\s)^i\ulHM{}'_\chr\bigr)\,.
\]
The induced morphism $\phi_K:\s K\to K$ has its kernel and cokernel supported on the graph of $c$. Set $\ulM'=(M',\tau'):=\bigl(\ker(\ulM\to\ulK),\tau|_{M'}\bigr)$. Then $\ulM'$ is a pure Anderson motive with the required properties by \cite[\BHAff]{BH_A}.
\end{proof}

There is a corresponding result at the places $v\ne\chr$ which is stated in Proposition~\ref{Prop2.7b}.


\bigskip

\section{Tate Modules} \label{SectTateModules}
\setcounter{equation}{0}

\begin{Definition} \label{DefTateMod}
If $\FF$ is an abelian $\tau$-sheaf over $L$, respectively $\ulM$ a pure Anderson motive over $L$ and $v\in C$ (respectively $v\in \Spec A$) is a place of $Q$ different from the characteristic point $\chr$, we define
\begin{eqnarray*}
T_v\FF:=T_v(\ulM_v(\FF)) \quad&\text{and}&\quad V_v\FF:=V_v(\ulM_v(\FF)) \quad\text{for }v\ne\infty,\\[2mm]
T_\infty\FF:=T_\infty(\ulTM_\infty(\FF)) \quad&\text{and}&\quad V_\infty\FF:=V_\infty(\ulTM_\infty(\FF))\quad\text{for }v=\infty\ne\chr,\\[2mm]
\text{respectively}\quad T_v\ulM:=T_v(\ulM_v(\ulM)) \quad&\text{and}&\quad V_v\ulM:=V_v(\ulM_v(\ulM))\,.
\end{eqnarray*}
We call $T_v\FF$ (respectively $V_v\FF$) the (rational) $v$-adic Tate module of $\FF$. If $v=\infty$ they both depend on the choice of $k,l$, and $z$; see page~\pageref{EQ.Pi+Lambda}.
\end{Definition}

By \cite[Proposition~6.1]{TW}, $T_v\FF$ (and $V_v\FF$) are free $A_v$-modules (respectively $Q_v$-vector spaces) of rank $r$ for $v\ne\infty$ and $rl$ for $v=\infty$, which carry a continuous $G=\Gal(L^\sep/L)$-action.

\smallskip

Also the Tate modules $T_\infty\FF$ and $V_\infty\FF$ are always equipped with the endomorphisms $\P$ and $\Lambda(\lambda)$ for $\lambda\in \Ff_{q^l}\cap L$ from (\ref{EQ.Pi+Lambda}). 
And if $\Ff_{q^l}\subset L$ we identify the algebra $\Delta_\infty$ from (\ref{EqDelta}) with a subalgebra of $\End_{Q_\infty}(V_\infty\FF)$ by mapping $\lambda\in\BF_{q^l}$ to $\Lambda(\lambda)$.

\medskip

\noindent {\it Remark.} 
Our functor $T_v$ is covariant. In the literature usually the $A_v$-dual of our $T_v \ulM$ is called the $v$-adic Tate module of $\ulM$. With that convention the Tate module functor is contravariant on Anderson motives but covariant on Drinfeld modules and Anderson's abelian $t$-modules~\cite{Anderson} (which both give rise to Anderson motives). Similarly the classical Tate module functor on abelian varieties is covariant. We chose our non-standard convention here solely to avoid perpetual dualizations. This agrees also with the remark after Definition~\ref{Def2.6} that abelian $\tau$-sheaves behave dually to abelian varieties.

\medskip

The following analogues of the Tate conjecture for abelian varieties are due to Taguchi~\cite{Taguchi95b} and Tamagawa~\cite[\S 2]{Tam}.

\begin{Theorem}\label{TATE-CONJECTURE-MODULES}
Let\/ $\ulM$ and $\ulM'$ be pure Anderson motives over a finitely generated field $L$ and let $G:=\Gal(\Lsep/L)$. Let\/ $\chr\ne v\in\Spec A$ be a maximal ideal. Then the Tate conjecture holds:
\[
\Hom(\ulM,\ulM')\otimes_A \Av \;\cong\; \Hom_\AvG(T_v\ulM,T_v\ulM')\,.\qed
\]
\end{Theorem}

\begin{Theorem}\label{TATE-CONJECTURE} (\cite[\BHAee]{BH_A})
Let $\FF$ and $\FF'$ be abelian $\tau$-sheaves over a finitely generated field $L$ and let $G:=\Gal(\Lsep/L)$. Let\/ $v\in C$ be a place different from the characteristic point $\chr$. 
\begin{suchthat}
\item
If $v\ne\infty$ assume $\chr\ne\infty$ or $\weight(\FF)=\weight(\FF')$. Then 
\[
\QHom(\FF,\FF')\otimes_Q\Qv\;\cong\;\Hom_\QvG(\VvFF,\VvFF')\,.
\]
\item
If $v=\infty$ choose an integer $l$ which satisfies \ref{Def1.1}/2 for both $\FF$ and $\FF'$ and assume $\Ff_{q^l}\subset L$. Then 
\[
\QHom(\FF,\FF')\otimes_Q Q_\infty\;\cong\;\Hom_{\Delta_\infty[G]}(V_\infty\FF,V_\infty\FF')\,.
\]
\end{suchthat}
\end{Theorem}

\forget{
\begin{proof}
1. Set $\ulM:=\ulM(\FF)$ and $\ulM':=\ulM(\FF')$. By \ref{CONNECTION} and \ref{TATE-CONJECTURE-MODULES}, we have
\[
\QHom(\FF,\FF')\otimes_Q Q_v\;\cong\;\Hom(\ulM,\ulM')\otimes_AQ_v \;\cong\;
\Hom_\QvG(V_v\ulM,V_v\ulM')\,.
\]

\smallskip\noindent
2. Let $D\subset C$ be a finite closed subscheme as in Section~\ref{SectRelation} with $\chr,\infty\notin D$ and set $\ulM:=\ulM^{(D)}(\FF)$ and $\ulM':=\ulM^{(D)}(\FF')$. By \ref{CONNECTION} and \ref{TATE-CONJECTURE-MODULES}, we have
\[
\QHom(\FF,\FF')\otimes_Q Q_\infty\;\cong\;\Hom_{\P,\Lambda}(\ulM,\ulM')\otimes_{\TA}Q_\infty \;\cong\;
\Hom_{\Delta_\infty[G]}(V_\infty \ulM,V_\infty \ulM')\,.
\]
Here the last isomorphism comes from the fact that the commutation with $\P$ and $\Lambda(\lambda)$ are linear conditions on $\Hom(\ulM,\ulM')$ and $\Hom(\ulM,\ulM')\otimes_{\TA}Q_\infty\cong\Hom_{Q_\infty[G]}(V_\infty \ulM,V_\infty \ulM')$ thus cutting out isomorphic subspaces.
\end{proof}
}

\medskip

As expected, there is the following relation between Tate modules and isogenies.

\begin{Proposition} (\cite[\BHAdd]{BH_A}) \label{Prop2.7b}
\begin{suchthat}
\item 
Let $f:\ulM'\to\ulM$ be an isogeny between pure Anderson motives then $T_vf(T_v\ulM')$ is a $G$-stable lattice in $V_v\ulM$ contained in $T_v\ulM$.
\item 
Conversely if $\ulM$ is a pure Anderson motive and $\Lambda_v$ is a $G$-stable lattice in $V_v\ulM$ contained in $T_v\ulM$, then there exists a pure Anderson motive $\ulM'$ and a separable isogeny $f:\ulM'\to \ulM$ with $T_vf(T_v\ulM')=\Lambda_v$.
\end{suchthat}
\end{Proposition}

\begin{Proposition}\label{FACTORSHEAF-FACTORSPACE}
Let\/ $\FF'$ be an abelian factor $\tau$-sheaf of\/ $\FF$. Then $\VvFF'$ is a $G$-factor space of\/ $\VvFF$. The same holds if $\ulM'$ is a  factor motive of a pure Anderson motive $\ulM$.
\end{Proposition}

\begin{proof}
Let $f\in\Hom(\FF,\FF')$ be surjective and let $\ulHM$ and $\ulHM'$ be the (big, if $v=\infty$) local $\sigma$-shtuka of $\FF$, respectively $\FF'$, at $v$. Then the induced morphism $\ulM_v(f)\in\Hom(\ulHM,\ulHM')$ is surjective and $\ulHM'':=\ker \ulM_v(f)$ is also a local $\sigma$-shtuka at $v$. We get an exact sequence of local $\sigma$-shtuka which we tensor with $A_{v,L^\sep}$ yielding
\[
\bigexact{0}{}{\ulHM''\otimes_{A_{v,L}}A_{v,L^\sep}}{}{\ulHM\otimes_{A_{v,L}}A_{v,L^\sep}}{\ulM_v(f)}{\ulHM'\otimes_{A_{v,L}}A_{v,L^\sep}}{}{0\,.}
\]
The Tate module functor is left exact, because considering the morphism of $A_{v,L^\sep}$-modules
\[
1-\t:\es\ulHM\otimes_{A_{v,L}}A_{v,L^\sep} \longto \ulHM\otimes_{A_{v,L}}A_{v,L^\sep}
\]
we have by definition $T_v\ulHM = \ker(1-\t)$, and the desired left exactness follows from the snake lemma. After tensoring with $\otimes_\Av\Qv$ we get
\[
\bigexact{0}{}{V_v\ulHM''}{}{V_v\ulHM}{V_v f}{V_v\ulHM'\;.}{}{}
\]
Counting the dimensions of these $\Qv$-vector spaces, we finally also get right exactness, as desired.
\end{proof}

\forget{
\begin{proof}
Let $f\in\Hom(\FF,\FF')$ be surjective, let $\ulM:=\ulM^{(D)}(\FF)$ \? and let $\ulM':=\ulM^{(D)}(\FF')$. Then the induced morphism $\ulM^{(D)}(f)\in\Hom(\ulM,\ulM')$ is surjective and $\ulM'':=\ker \ulM^{(D)}(f)$ is a $\t$-module on $A$. Thus we get the exact sequence
\[
\bigexact{0}{}{M''}{}{M}{\ulM^{(D)}(f)}{M'}{}{0\,.}
\]
The exactness being preserved since $M'$ is locally free, we consider the following diagram
\[
\xymatrix{
0 \ar[r] &
M''/v^nM'' \otimes_L\Lsep \ar[r] &
M/v^nM \otimes_L\Lsep \ar[r] &
M'/v^nM' \otimes_L\Lsep \ar[r] &
0 \\
0 \ar@{-->}[r] &
(M''/v^nM''\otimes_L\Lsep)^{\textstyle\t} \ar@{-->}[r]\ar[u] &
(M/v^nM\otimes_L\Lsep)^{\textstyle\t} \ar@{-->}[r]\ar[u] &
(M'/v^nM'\otimes_L\Lsep)^{\textstyle\t} \ar[u] \;.&
\\
}
\]
The $\t$-invariant functor is left exact, because considering the morphism of $A/{v^n}$-modules
\[
1-\t:\,(M/v^nM)\otimes_L\Lsep\rightarrow(M/v^nM)\otimes_L\Lsep
\]
we have by definition $((M/v^nM)\otimes_L\Lsep)^{\textstyle\t} = \ker\, 1-\t$, and the desired left exactness follows from the snake lemma. Since the projective limit preserves left exactness as well, we get after tensoring with $\otimes_\Av\Qv$
\[
\bigexact{0}{}{V_v\ulM''}{}{V_v\ulM}{V_v\ulM^{(D)}(f)}{V_v\ulM'\,.}{}{}
\]
Counting the dimensions of these $\Qv$-vector spaces, we finally also get right exactness, as desired.
\end{proof}

}


\bigskip

\section{The Frobenius Endomorphism}
\setcounter{equation}{0}

Suppose that the characteristic is finite, that is, the characteristic point $\chr$ is a closed point of $C$ with finite residue field $\Ff_\chr$, and the map $c:\Spec L\to C$ factors through the finite field $\chr=\Spec\Ff_\chr$.

\begin{Definition}[$s$-Frobenius on abelian $\tau$-sheaves] \label{DefFrob}
Let $\FF$ be an abelian $\tau$-sheaf with finite characteristic point $\chr=\Spec \Ff_\chr$ and let $s=q^e$ be a power of the cardinality of $\Ff_\chr$. We define the \emph{$s$-Frobenius on $\FF$} by
\[
\pi \;:=\; (\pi_i):\, (\sigma^\ast)^e\FF\rightarrow\FF\II{e}, \quad 
\pi_i \;:=\; \t_{i+e-1}\circ\cdots\circ(\s)^{e-1}\t_{i}:\, (\sigma^\ast)^e\F_{i}\rightarrow\F_{i+e}\;.
\]
Clearly $\pi$ is an isogeny. Observe that $\Ff_\chr\subset\Fs$ implies that $(\sigma^\ast)^e\FF$ has the same characteristic as $\FF$. 
\end{Definition}

Similarly if $\chr\in\Spec A$ is a closed point we define

\begin{Definition}[$s$-Frobenius on pure Anderson motives]\label{Def2.19b}
Let $\ulM$ be a pure Anderson motive with finite characteristic point $\chr=\Spec \Ff_\chr$ and let $s=q^e$ be a power of the cardinality of $\Ff_\chr$. We define the \emph{$s$-Frobenius isogeny on $\ulM$} by
\[
\pi\;:=\;\t\circ\ldots\circ (\s)^{e-1}\t:\,(\sigma^\ast)^e\ulM\to\ulM\,.
\]
\end{Definition}

\begin{Remark}
Classically for (abelian) varieties $X$ over a field $K$ of characteristic $p$ one defines the Frobenius morphism $X\to\phi^\ast X$ where $\phi$ is the $p$-Frobenius on $K$. There $p$ equals the cardinality of the ``characteristic field'' $\im(\Z\to K)=\Ff_p$. 
In view of the dual behavior of abelian $\tau$-sheaves and pure Anderson motives our definition is a perfect analogue since here we consider the $s$-Frobenius for $s$ being the cardinality of (a power of) the ``characteristic field'' $\im(c^\ast:A\to L)=\BF_\chr$.
\end{Remark}

Now we suppose $L=\Fs$ to be a finite field with $s=q^e$ $(e\in\N$). Let $\FS$ denote a fixed algebraic closure of $\Fs$ and set $G=\GalFSFs$. It is topologically generated by ${\rm Frob}_s:x\mapsto x^s$. The following results for the Frobenius endomorphism of $\t$-modules can be found in Taguchi and Wan~\cite[\S 6]{TW}.

\begin{Proposition}\label{ABSOLUTE-GALOIS}
Let $\ulM$ be a pure Anderson motive over $\Fs$ of rank $r$ and let\/ $\chr\ne v\in\Spec A$ be a maximal ideal.
\begin{suchthat}
\item The generator ${\rm Frob}_s$ of $G$ acts on $T_v\ulM$ like\/ $(T_v\pi)^{-1}$.
\item Let\/ $\Psi:\,\AvG\rightarrow\End_\Av(T_v\ulM)$ denote the continuous morphism of $\Av$-modules which is induced by the action of\/ $G$ on $T_v\ulM$. Then\/ $\im\Psi = \Av\II{T_v\pi}$.
\end{suchthat}
\end{Proposition}

\begin{proof}
1 was proved in \cite[Ch. 6]{TW} and 2 follows from the continuity of $\Psi$.
\end{proof}

\noindent {\it Remark.}
The inversion of $T_v\pi$ in the first statement results from the dual definition of our Tate module.

\begin{Proposition}\label{PI-IS-QISOG}
Let\/ $\FF$ be an abelian $\tau$-sheaf over $L=\Fs$ with $s=q^e$ and let\/ $\pi$ be its $s$-Frobenius. Then $(\sigma^\ast)^e\FF=\FF$. Let\/ $v\in C$ be a place different from $\infty$ and from the characteristic point $\chr$. 
\begin{suchthat}
\item The $s$-Frobenius $\pi$ can be considered as a quasi-isogeny of\/ $\FF$.
\item The generator ${\rm Frob}_s$ of $G$ acts on $T_v\FF$ like\/ $(T_v\pi)^{-1}$.
\item The image of the continuous morphism of $\Qv$-vector spaces $\QvG\rightarrow\End_\Qv(V_v\FF)$ is $\Qv\II{V_v\pi}$.
\item $\ulM(\pi)$ coincides with the $s$-Frobenius on the pure Anderson motive $\ulM(\FF)$ from definition~\ref{Def2.19b}.
\end{suchthat}
\end{Proposition}

\begin{proof}
1. Due to the periodicity condition, we have $\FF\II{e}\subset\FF(nk\cdot\infty)$ for a sufficiently large $n\in\N$, since $\F_{i+e}\subset\F_{i+nl}=\F_i(nk\cdot\infty)$ for $e\le nl$. Thus $\pi\in\Hom(\FF,\FF(nk\cdot\infty))$, and therefore $\pi\in\QEnd(\FF)$. By \ref{PROP.1.42A}, we have $\pi\in\QIsog(\FF)$.\\
2 and 3 again follow from \cite[Ch. 6]{TW} and the continuity of $\Psi$; see~\cite[Proposition 2.29]{BH} for more details. \\
4 follows from the definition of $\pi$ and the commutation of the $\P$'s and the $\t$'s.
\end{proof}

\bigskip

\bigskip

\section{The Poincar\'e-Weil Theorem}
\setcounter{equation}{0}

In this section we study the analogue for pure Anderson motives and abelian $\tau$-sheaves of the Poincar\'e-Weil theorem. Originally, this theorem states that every abelian variety is semisimple, that is, isogenous to a product of simple abelian varieties, see \cite[Corollary of Theorem~II.1/6]{La}. Unfortunately, we cannot expect a full analogue of this statement for abelian $\tau$-sheaves or pure Anderson motives as our next example illustrates. On the positive side we show that every abelian $\tau$-sheaf or pure Anderson motive over a finite field becomes semisimple after a finite base field extension.

\begin{Example} \label{Ex3.1}
Let $C=\PP^1_\Fq$, $C\setminus\{\infty\}=\Spec\Fq\II{t}$ and $\zeta:=c^*(1/t)\in\Fq^\times$. We construct an abelian $\tau$-sheaf $\FF$ over $L=\Fq$ with $r=d=2$ which is not semisimple. Let 
\[\textstyle
\Delta=\tmatr{1}{0}{0}{1} + \tmatr{\alpha}{\gamma}{\beta}{\delta}\cdot t
\]
with $\alpha,\beta,\gamma,\delta\in\Fq$.
To obtain characteristic $c$ we need $\det\Delta=(1-\zeta t)^2$, and thus we require the conditions $\alpha+\delta=-2\zeta$ and $\alpha\delta-\beta\gamma=\zeta^2$. We set $\F_i:=\O_{C_L}(i\cdot\infty)^{\oplus 2}$, we let $\P_i$ be the natural inclusion, and  we let $\t_i:=\Delta$. Then $\FF$ is an abelian $\tau$-sheaf with $r=d=2$ and $k=l=1$. The associated pure Anderson motive is $\ulM=(L[t]^{\oplus 2},\Delta)$.

We see that $\FF$ is not simple. If $\Delta=\smatr{1-\zeta t}{0}{0}{1-\zeta t}$ then $\FF$ is semisimple as a direct sum of two simple abelian $\tau$-sheaves. Otherwise, if $\Delta\ne\smatr{1-\zeta t}{0}{0}{1-\zeta t}$ which is the case for example if $\beta\ne 0$, consider
\[\textstyle 
\tilde{\Delta} \,:=\, \tmatr{\beta}{\delta+\zeta}{0}{1}^{-1} \!\!\cdot\Delta\cdot \s\tmatr{\beta}{\delta+\zeta}{0}{1} \,=\, \tmatr{1-\zeta t}{0}{t}{1-\zeta t}
\]
and the abelian $\tau$-sheaf with $\wt\F_i=\O_{C_L}(i\cdot\infty)^{\oplus 2}$ and $\tilde\t_i=\tilde\Delta$ which is isomorphic to $\FF$.
There is an exact sequence 
\[
\begin{array}{c}
\bigexact{0}{}{\FF'}{\varphi}{\TFF}{\psi}{\FF''}{}{0} \\[1ex]
\scriptstyle \t'\,=\,1-\zeta t \qquad\quad\Tt\quad\qquad \t''\,=\,1-\zeta t
\end{array}
\]
with $\varphi:\, 1\mapsto\tvect{1}{0}$ and $\psi:\,\tvect{x}{y}\mapsto y$ where $\FF'=\FF''$ is the abelian $\tau$-sheaf with $\F'_i=\O_{C_L}(i\cdot\infty)$ and $\t'_i=1-\zeta t$. If $\TFF$ were semisimple, then there would be a quasi-morphism $\omega:\,\FF''\rightarrow\TFF$ with $\psi\circ\omega=\id_{\FF''}$, hence $\omega:\, y\mapsto\tvect{e}{1}\cdot y$ for some $e\in \Fq(t)$. Thus, a necessary condition for the semisimplicity of $\FF$ is
\[\textstyle
(1-\zeta t)\cdot\s(y)\cdot \tvect{e}{1}
\;=\;
\tmatr{1-\zeta t}{0}{t}{1-\zeta t}\cdot\tvect{\s(e)}{1}\cdot\s(y)
\]
which is equivalent to the condition
\[
e-\s(e)\;=\;\frac{t}{1-\zeta t}\enspace.
\]
But this cannot be true since $e-\s(e)=0$, thus $\FF$ is \emph{not} semisimple. However, this last formula is satisfied if $e=\lambda\cdot\frac{t}{1-\zeta t}$ for $\lambda\in\Ff_{q^q}$ with $\lambda^q-\lambda=-1$. That means that after field extension $\Fq(\lambda)\,/\,\Fq$ we get $\FF\cong{\FF'}\mbox{\,}^{\oplus 2}$  and we have $\QEnd(\FF)=M_2(\QEnd(\FF'))=M_2(Q)$. Note that this phenomenon generally appears, and we will state and prove it in Theorem \ref{Thm3.8b}.
\end{Example}

From now on we fix a place $v\in\Spec A$ which is different from the characteristic point $\chr$ of $c$. For a morphism $f\in\QHom(\FF,\FF')$ between two abelian $\tau$-sheaves $\FF$ and $\FF'$ we denote its image $V_vf\in\Hom_\QvG(\VvFF,\VvFF')$ just by $f_v$. If $\FF$ is defined over $\Fs$ this applies in particular to the $s$-Frobenius endomorphism $\pi$ of $\FF$ (Definition~\ref{DefFrob}).

Let $\FF$ be an abelian $\tau$-sheaf over the finite field $L=\Fs$. We set
\begin{equation}\label{Eq3.1}
\begin{array}{l@{\;}c@{\;}l@{\qquad\qquad}l@{\;}c@{\;}l}
E &:=& \QEnd(\FF) \ni\pi & E_v &:=& \End_\QvG(\VvFF) \ni\pi_v \\[0.5ex]
F &:=& Q\II{\pi} \subset E & F_v &:=& \im(\QvG\rightarrow \End_\Qv(\VvFF)) \\
\end{array}
\end{equation}
with $\QvG\rightarrow \End_\Qv(\VvFF)$ induced by the action of $G$ on $\VvFF$. Clearly, we have $F\subset E$ and $F_v\subset E_v$ by Proposition~\ref{PI-IS-QISOG}/3. By \cite[\BHAmm]{BH_A}, we know that $\dim_Q E<\infty$. Thus $\pi$ is algebraic over $Q$. We denote its minimal polynomial by $\mu_\pi\in Q[x]$, and the characteristic polynomial of the endomorphism $\pi_v$ of $V_v\FF$ by $\chi_v\in Q_v[x]$. If $\chr\ne\infty$, Theorem~\ref{ThmT.3} shows that $\pi$ is integral over $A$, $\mu_\pi\in A[x]$. The zeroes of $\pi$ in $\Spec A[\pi]$ all lie above $\chr$ because $\pi$ is an isomorphism locally at all $v\ne\chr$; compare with \cite[\BHAcc]{BH_A}.

Due to the Tate conjecture, our situation can be represented by the following diagram where we want to fit the missing bottom right arrow with an isomorphism.
\[
\xymatrix{
\,E^{\,\!}_{\,\!} \ar[r] & 
\,E\xQv\, \ar[r]^<<{\sim} &
\,E_v \\
\,F^{\,\!}_{\,\!} \ar[r]\ar[u] & 
\,F\xQv\, \ar@{-->}[r]^<<{\sim}\ar[u] &
\,F_v \,.\!\!\ar[u] \\
}
\]

\begin{Lemma}\label{Lemma3.4}
The natural morphism between $F\xQv$ and\/ $F_v$ is an isomorphism.
\end{Lemma}

\begin{proof}
Consider the isomorphism $\psi:\,E\xQv\cong E_v\subset\End_\Qv(\VvFF)$ and set $\varphi:=\psi|_{F\xQv}$. Then $\varphi$ is injective and maps into $F_v$. Since $\im\varphi=\Qv\II{\pi_v}$, the surjectivity follows from Proposition~\ref{PI-IS-QISOG}.
\end{proof}

To evaluate the dimension of $E$ we need the following notation.

\begin{Definition}\label{Def3.3}
Let\/ $K$ be a field. Let\/ $f,g\in K\II{x}$ be two polynomials and let
\[
f\, = \!\!\!\prod\limits_{\genfrac{}{}{0pt}{}{\scriptstyle\mu\in K\II{x}}{\mbox{\rm\scriptsize irred.}}}\! \mu^{m(\mu)}
,\qquad
g\, = \!\!\!\prod\limits_{\genfrac{}{}{0pt}{}{\scriptstyle\mu\in K\II{x}}{\mbox{\rm\scriptsize irred.}}}\! \mu^{n(\mu)}
\]
be their respective factorizations in powers of irreducible polynomials. Then we define the integer
\[
r_K(f,g)\; := \!\!\!\prod\limits_{\genfrac{}{}{0pt}{}{\scriptstyle\mu\in K\II{x}}{\mbox{\rm\scriptsize irred.}}}\! m(\mu)\cdot n(\mu)\cdot\deg\mu \;.
\]
\end{Definition}

\noindent {\it Remark.}
In contrast to characteristic zero, we have for $\charakt(K)\ne0$ in general different values of the integer $r_K$ for different ground fields $K$. Namely, if $K\subset L$ then $r_K(f,g)\le r_L(f,g)$ with equality if and only if all irreducible $\mu\in K\II{x}$ which are contained both in $f$ and in $g$ have no multiple factors in $L\II{x}$. This is satisfied for example if the greatest common divisor of $f$ and $g$ has only separable irreducible factors, or if $L$ is separable over $K$. See \ref{LAST-EXAMPLE} below for an example where $r_K(f,g)<r_L(f,g)$.

\medskip

Before we discuss semisimplicity criteria in \ref{PROP.EQ} -- \ref{Cor3.8c}, let us compute the dimension of $\QHom(\FF,\FF')$.

\begin{Lemma}\label{PROP.15}\label{PI-END-SEMISIMPLE}
Let $v$ be a place of $Q$ different from $\chr$ and $\infty$. Let $\FF$ and $\FF'$ be abelian $\tau$-sheaves over $\Fs$ and assume that $\pi_v$ and $\pi'_v$ are semisimple. Factor their characteristic polynomials $\chi_v \,=\, \mu_1^{m_1}\cdot\ldots\cdot\mu_n^{m_n}$ and $\chi'_v\,=\, \mu_1^{m'_1}\cdot\ldots\cdot\mu_n^{m'_n}$ with distinct monic irreducible polynomials $\mu_1,\dots,\mu_n\in\Qv\II{x}$ and $m_i,m'_i\in\N_0$. Then
\begin{suchthat}
\item
$\DS\Hom_{Q_v[G]}(V_v\FF,V_v\FF')\,\cong\,\bigoplus_{i=1}^nM_{m_i'\times m_i}\bigl(Q_v[x]/(\mu_i)\bigr)$ as $Q_v$-vector spaces,
\item 
$\DS\End_{Q_v[G]}(V_v\FF)\,\cong\,\bigoplus_{i=1}^nM_{m_i}\bigl(Q_v[x]/(\mu_i)\bigr)$ as $Q_v$-algebras, and
\item
$\dim_\Qv\Hom_\QvG(\VvFF,\VvFF') \,=\, r_\Qv(\chi_v,\chi'_v)$\,.
\end{suchthat}
\end{Lemma}

\begin{proof}
Clearly 2 and 3 are consequences of 1 which we now prove. Since $\pi_v$ and $\pi'_v$ are semisimple, we have the following decomposition of $Q_v[G]$-modules
\[
\VvFF \,\cong\, \bigoplus_{i=1}^n\,(\Qv\II{x}/(\mu_i))^{\oplus m_i}, \qquad
\VvFF'\,\cong\, \bigoplus_{i=1}^n\,(\Qv\II{x}/(\mu_i))^{\oplus m'_i}
\]
where $\Qv\II{x}/(\mu_i)=:K_i$ are fields. Obviously, we only have non-zero $Q_v[G]$-morphisms $K_i\rightarrow K_j$ if $i=j$, since otherwise $\mu_i(\pi)\ne0$ in $K_j$. Since $\pi_v$ operates on $K_i^{\oplus m_i}$ as multiplication by the scalar $x$, the lemma follows.
\end{proof}

\begin{Theorem}\label{Thm3.5a}
Let $\FF$ and $\FF'$ be abelian $\tau$-sheaves of the same weight over $\BF_s$ and assume that the endomorphisms $\pi_v$ and $\pi'_v$ of $V_v\FF$ and $V_v\FF'$ are semisimple at a place $v\ne\chr,\infty$ of $Q$. Let $\chi_v$ and $\chi'_v$ be their characteristic polynomials. Then
\[
\dim_Q\QHom(\FF,\FF')\;=\;r_{Q_v}(\chi_v,\chi'_v).
\]
\end{Theorem}

\begin{proof}
This follows from the lemma and the Tate conjecture, Theorem~\ref{TATE-CONJECTURE}.
\end{proof}

\begin{Corollary}\label{Cor3.11b}
Let $\FF$ be an abelian $\tau$-sheaf of rank $r$ over $\Fs$ with Frobenius endomorphism $\pi$ and let $\mu_\pi$ be the minimal polynomial of $\pi$. Assume that $F=Q[x]/(\mu_\pi)$ is a field and set $h:=[F:Q]=\deg\mu_\pi$. Then
\begin{suchthat}
\item
$h|r$ and $\dim_Q \QEnd(\FF) =\frac{r^2}{h}$ and $\dim_F\QEnd(\FF)=\frac{r^2}{h^2}$.
\item 
For any place $v$ of $Q$ different from $\chr$ and $\infty$ we have $\QEnd(\FF)\otimes_Q Q_v\cong M_{r/h}(F\otimes_Q Q_v)$ and $\chi_v=(\mu_\pi)^{r/h}$ independent of $v$.
\end{suchthat}
\end{Corollary}

\begin{proof}
Since $F$ is a field, $\pi_v$ is semisimple by \ref{PROP.EQ} below. So general facts of linear algebra imply that $\mu_\pi=\mu_1\cdot\ldots\cdot\mu_n$ with pairwise different irreducible monic polynomials $\mu_i\in Q_v[x]$ and $\chi_v=\mu_1^{m_1}\cdot\ldots\cdot\mu_n^{m_n}$ with $m_i\ge1$. We set $K_i=Q_v[x]/(\mu_i)$ and use the notation from (\ref{Eq3.1}). By Lemma~\ref{PROP.15} the semisimple $Q_v$-algebra $E_v$ decomposes $E_v\cong \bigoplus_{i=1}^n E_i$ into the simple constituents $E_i=M_{m_i}(K_i)$. By \cite[Th\'eor\`eme~5.3/1 and Proposition~5.4/12]{Bou}, $E_i=E_v\cdot e_i$ where $e_i$ are the central idempotents with $K_i=F_v\cdot e_i$. Thus there are epimorphisms of $K_i$-vector spaces
\[
\QEnd(\FF)\otimes_F K_i \,=\,E_v\otimes_{F_v}K_i \,\onto\, E_i\,.
\]
This shows that $m_i^2\le \dim_F E$. So by Lemma~\ref{PROP.15}
\begin{eqnarray*}
[F:Q]\cdot\dim_F E\es=\es\dim_{Q_v}E_v&=&\sum_{i=1}^n m_i^2\deg\mu_i\es\le \\[-2mm]
\le\es\dim_F E\cdot\sum_{i=1}^n\deg\mu_i &=& \dim_F E\cdot\deg\mu_\pi\es =\es [F:Q]\cdot\dim_F E\,.
\end{eqnarray*}
Therefore $m_i^2=\dim_F E$ for all $i$. Since $r=\deg\chi_v=\sum_i m_i\deg\mu_i=\sqrt{\dim_F E}\cdot[F:Q]$. We find $r=m_ih$ and $\dim_FE=\frac{r^2}{h^2}$, proving 1. For 2 we use that 
\[
E_v\es\cong\es\bigoplus_i\, M_{r/h}\bigl(Q_v[x]/(\mu_i)\bigr)\es=\es M_{r/h}\bigl(\bigoplus_i Q_v[x]/(\mu_i)\bigr)\es=\es M_{r/h}\bigl(Q_v[x]/(\mu_\pi)\bigr)\,.
\]

\end{proof}

\medskip

Next we investigate when $\pi_v$ is semisimple.

\begin{Remark}\label{SEPARABLE}
Notice that the completion $\Qv$ is separable over $Q$. Namely, in terms of \cite[IV.7.8.1--3]{EGA}, we can state that $\O_{C,v}$ is an excellent ring. Thus the formal fibers of $\widehat{\O}_{C,v} \longrightarrow {\O_{C,v}}$ and therefore $Q_v = \widehat{\O}_{C,v}\!\otimes_{\O_{C,v}}\!Q \longrightarrow Q$ are geometrically regular. This means that $Q_v\otimes_Q K$ is regular for every finite field extension $K$ over $Q$. Since \quotes{regular} implies \quotes{reduced}, we conclude that $Q_v$ is separable over $Q$.
\end{Remark}

\begin{Proposition}\label{PROP.EQ}
In the notation of (\ref{Eq3.1}) the following statements are equivalent:
\begin{suchthat}
\item $\pi$ is semisimple.
\item $F$ is semisimple.
\item $F\xQv\cong F_v$ is semisimple.
\item $\pi_v$ is semisimple.
\item $E\xQv\cong E_v$ is semisimple.
\item $E$ is semisimple.
\end{suchthat}
\end{Proposition}

\begin{proof}
1. and 2. are equivalent by definition. So we show the equivalences from 2. to 6. 

Let $F$ be semisimple. Since $\Qv$ is separable over $Q$, we conclude that $F\xQv\cong \Qv\II{\pi_v}$ is semisimple by \cite[Corollaire~7.6/4]{Bou}.
Hence $\pi_v$ is semisimple by definition, and we showed in Lemma~\ref{PI-END-SEMISIMPLE}/2 that then $E_v\cong E\xQv$ is semisimple. Again by \cite[Corollaire~7.6/4]{Bou} this implies that $E$ is semisimple. Since $F\subset Z(E)$ is a finite dimensional $Q$-subalgebra of the center of $E$, we conclude by \cite[Corollaire de Proposition~6.4/9]{Bou} that $F$ is semisimple, and our proof is complete.
\end{proof}

\noindent {\it Remark.}
If more generally $\FF$ is defined over a finitely generated field, then one cannot consider $\pi$, $\pi_v$, nor $F$. Nevertheless 5 and 6 remain equivalent and are still implied by 3 due to the following well-known lemma. Namely $E_v$ is the commutant of $F_v$ in $\End_{Q_v}(V_v\FF)$. 
We thank O.\ Gabber for mentioning this fact to us and we include its proof for lack of reference.

\begin{Lemma}\label{Lemma3.5b}
Let $B$ be a central simple algebra of finite dimension over a field $K$ and let $F$ be a semisimple $K$-subalgebra of $B$. Then the commutant of $F$ in $B$ is semisimple.
\end{Lemma}

\begin{proof}
Let $F=\bigoplus_i F_i$ be the decomposition into simple constituents and let $e_i$ be the corresponding central idempotents, that is, $F_i=Fe_i$. Consider $B_i=e_i Be_i$ which is again central simple over $K$ by \cite[Corollaire~6.4/4]{Bou}, since if $I\subset B_i$ is a non-zero two sided ideal then $BIB$ contains $1$ and so $I$ contains the unit $e_i$ of $B_i$. By \cite[Th\'eor\`eme~10.2/2]{Bou} the commutant $E_i$ of $F_i$ in $B_i$ is simple. Clearly the commutant $E$ of $F$ in $B$ satisfies $E_i=e_iEe_i=Ee_i$ and $E=\bigoplus_i E_i$ proving the lemma.
\end{proof}

\begin{Corollary} \label{CorFisCenter}
Let\/ $\FF$ be an abelian $\tau$-sheaf over $\Fs$ of rank $r$ with semisimple Frobenius endomorphism\/ $\pi$. 
Then the algebra $F=Q(\pi)$ is the center of the semisimple algebra $E=\QEnd(\FF)$.
\end{Corollary}

\begin{proof}
Since $F_v$ is semisimple, we know by \cite[Proposition~5.1/1]{Bou} that the $F_v$-module $\VvFF$ is semisimple. The commutant of $F_v$ in $\End_{Q_v}(\VvFF)$ is $E_v$ by definition. Trivially $\VvFF$ is of finite type over $E_v$. Thus, by the theorem of bicommutation \cite[Corollaire~4.2/1]{Bou}, the commutant of $E_v$ in $\End(\VvFF)$ is again $F_v$. We conclude $Z(E_v) = E_v \cap F_v = F_v$ and we have $F\xQv = F_v = Z(E_v) = Z(E)\xQv$ by \cite[Corollaire de Proposition~1.2/3]{Bou}. Considering the dimensions, we obtain $\dim_Q F= \dim_Q Z(E)$.
Since $F\subset Z(E)$ and the dimensions are finite, we finish by $F=Z(E)$.
\end{proof}

\begin{Theorem}\label{Thm3.8} 
Let $\FF$ be an abelian $\tau$-sheaf over a finite field $L$.
\begin{suchthat}
\item If $\QEnd(\FF)$ is a division algebra over $Q$ then $\FF$ is simple. If in addition $\chr\ne\infty$ then both statements are equivalent.
\item If the characteristic point $\chr$ is different from $\infty$ then
$\FF$ is semisimple if and only if $\QEnd(\FF)$ is semisimple.
\end{suchthat}
\end{Theorem}

\begin{proof}
1. Let $\QEnd(\FF)=E$ be a division algebra and let $f\in\Hom(\FF,\FF')$ be the morphism onto a non-zero factor sheaf $\FF'$ of $\FF$. We show that $f$ is an isomorphism. We know by \ref{FACTORSHEAF-FACTORSPACE} that \mbox{$f_v\in\Hom_\QvG(\VvFF,\VvFF')$} is surjective. By the semisimplicity of $E$ and Proposition~\ref{PROP.EQ}, $F_v$ is semisimple, and therefore $\VvFF$ is a finitely generated semisimple $F_v$-module. Thus we get a morphism $g_v\in\Hom_{Q_v[G]}(\VvFF',\VvFF)$ with \mbox{$f_v\circ g_v=\id_{\VvFF'}$}. Consider the integral Tate modules $T_v\FF$ and $T_v\FF'$. We can find some $n\in\N$ such that 
\[
v^n g_v \in \Hom_{\AvG}(T_v\FF',T_v\FF) \;\cong\; \Hom(\ulM(\FF'),\ulM(\FF))\otimes_A\Av
\]
and we choose $g\in\Hom(\ulM(\FF'),\ulM(\FF))\subset\QHom(\FF',\FF)$ with $g\equiv v^n g_v$ modulo $v^m$ for a sufficiently large $m>n$. If $g\circ f=0$ in $E$, then $f\circ g\circ f=0$, and therefore $f\circ g=0$ in $\QEnd(\FF')$ due to the surjectivity of $f$. This would imply
\[
v^n\cdot\id_{\VvFF'} \;=\; v^n(f_v\circ g_v) \;=\; f_v\circ(v^n g_v) \;\equiv\; f\circ g \;=\; 0 
\qquad\mbox{(modulo $v^m$)}
\]
which is a contradiction. Thus $g\circ f\not=0$ is invertible in $E$, and therefore $f$ is injective. By that, $f$ gives the desired isomorphism between $\FF'$ and $\FF$. The second assertion follows from Theorem~\ref{QEND-DIVISION-MATRIX}.

\noindent 
2. We already saw one direction in Theorem~\ref{QEND-DIVISION-MATRIX}/2. So let now $\QEnd(\FF)$ be semisimple and let 
\[
\QEnd(\FF)=\bigoplus_{j=1}^m\, M_{\lambda_j}(E_j)
\]
be the decomposition into full matrix algebras $M_{\lambda_j}(E_j)$ over division algebras $E_j$ over $Q$ $(1\le j\le m)$. For each $j$ we find $\lambda_j$ distinct idempotents $e_{j,1},\dots,e_{j,\lambda_j}\in M_{\lambda_j}(E_j)$ such that $e_{j,\alpha}\cdot\QEnd(\FF)\cdot e_{j,\alpha} = E_j$ for all $1\le\alpha\le\lambda_j$ with $\sum_{\alpha=1}^{\lambda_j} e_{j,\alpha}=1$ in $M_{\lambda_j}(E_j)$. Let $e_1,\dots,e_n$ denote all these idempotents, $n=\sum_{j=1}^m \lambda_j$, and choose a divisor $D$ on $C$ such that $e_i\in\Hom(\FF,\FF(D))$ for all $1\le i\le n$. Then $\sum_{i=1}^n e_i=\id_\FF$ in $\QEnd(\FF)$ and therefore 
\[
\begin{CD}
\FF\; @>{\sum_i e_i}>> \;{\displaystyle\bigoplus_{i=1}^n\, \im e_i}\; \subset \,\FF(D)\;.
\end{CD}
\]
The image $\FF_i:=\im e_i$ is an abelian $\tau$-sheaf by \cite[\BHAbb]{BH_A} because $\chr\ne\infty$. Since $\sum_i e_i$ is injective it is an isogeny by \ref{PROP.1.42A}. Since $\QEnd(\FF_i) = e_i\cdot\QEnd(\FF)\cdot e_i$ is a division algebra, $\FF_i$ is a simple abelian $\tau$-sheaf by 1. Thus $\FF\approx\FF_1\oplus\cdots\oplus\FF_n$ gives the decomposition into a direct sum of simple abelian $\tau$-sheaves $\FF_i$ as desired. 
\end{proof}

\begin{Remark}\label{Rem3.10b}
Unfortunately the theorem fails if $L$ is not finite, as Example~\ref{Ex3.10c} below shows. The reason is, that then $E_v$ may still be semisimple while the image $F_v$ of $Q_v[G]$ in $\End_{Q_v}(V_v\FF)$ is not. Nevertheless, if one adds the assumption that $F_v$ is semisimple, the assertions of Theorem~\ref{Thm3.8} remain valid over an arbitrary field $L$. (See also the remark after Proposition~\ref{PROP.EQ}.)
\end{Remark}

\begin{Example}\label{Ex3.10c}
We construct a pure Anderson motive $\ulM$ over a non-finite field $L$ which is not semisimple, but has $\End(\ulM)=A$. Any associated abelian $\tau$-sheaf $\FF$ has $\QEnd(\FF)=Q$. Let $C=\PP^1_\Fq$, $A=\Fq[t]$ with $q>2$, and $L=\Fq(\alpha)$ where $\alpha$ is transcendental over $\Fq$. Let $M=A_L^{\oplus 2}$ and $\t=\matr{\alpha t}{0}{t}{t}$. Then $\ulM=(M,\t)$ is a pure Anderson motive of rank and dimension $2$. Clearly $\ulM$ is not simple, since $\ulM'=(A_L,\t'=t)$ is a factor motive by projecting onto the second coordinate. We will see below that $\ulM$ is not even semisimple.

Let $\matr{e}{g}{f}{h}\in M_2(A_L)$ be an endomorphism of $\ulM$, that is,
\[
\matr{\alpha\s e+\s g}{\s g}{\;\;\alpha\s f+\s h}{\;\s h}\es=\es\matr{\alpha e}{\alpha g}{\;\;e+f}{\;\;g+h}\,.
\]
Choose $\beta\in\BF_q(\alpha)^\alg\setminus\BF_q(\alpha)$ satisfying $\beta^{q-1}=\alpha$ (for $\beta\notin\BF_q(\alpha)$ we use $q>2$). Then $\s g=\alpha g$ implies $g\in\beta\cdot\BF_q[t]$. Since also $g\in\BF_q(\alpha)[t]$ we must have $g=0$. Now $\s e=e$ and $\s h=h$ yielding $e,h\in\BF_q[t]$.

Let $\gamma\in\BF_q(\alpha)^\alg\setminus\BF_q(\beta)$ with $\gamma^q-\gamma=\beta$ and set $\tilde f:=\beta f-\gamma\cdot(e-h)$. Then $\alpha \s f-f=e-\s h=e-h$ implies $\s\tilde f-\tilde f=\beta^q\s f-\beta f-(\gamma^q-\gamma)(e-h)=\beta(\alpha\s f-f-(e-h))=0$. Thus $\tilde f\in\BF_q[t]$ and $\gamma\cdot(e-h)\in\BF_q(\beta)[t]$. So we must have $e=h$ and then $\beta f=\tilde f\in\BF_q[t]$ implies $f=0$. 
This shows that $\End(\ulM)=\Fq[t]=A$.

\forget{
Using the inclusion $A_L=\Fq(\alpha)[t]\subset\Fq(t)\dpl\alpha\dpr$ we write $g=\sum_{i=N}^\infty g_i\alpha^i$ with $N\in\Z$, $g_i\in\Fq(t)$, and $g_N\ne0$. Then $\s g=\sum_{i=N}^\infty g_i\alpha^{qi}$. Comparing the coefficients of $\alpha^i$ in the equation 
\[
\sum_{i=N}^\infty g_i\alpha^{qi}\;=\;\s g\;=\;\alpha g\;=\;\sum_{i=N}^\infty g_i\alpha^{i+1}
\]
and observing $q>2$ one easily sees that all $g_i$ must be zero. Therefore $g=0$ and $\s e=e$, $\s h=h$, hence $e,h\in\Fq[t]$. Now we write $f=\sum_{i=N}^\infty f_i\alpha^i\in\Fq(t)\dpl\alpha\dpr$ with $f_N\ne0$ to solve the remaining equation $\alpha\s f+\s h=e+f$, that is,
\[
\sum_{i=N}^\infty f_i\alpha^i-\sum_{i=N}^\infty f_i\alpha^{qi+1}\;=\;h-e\,.
\]
It follows that $N=0$, $f_0=h-e=f_i$ whenever $i=1+q+\ldots+q^j$ for some $j\in\N_0$ and $f_i=0$ otherwise. But for $e\ne h$ this solution $f$ does not belong to $\Fq(\alpha)[t]$ due to the well known characterization
\[
\Fq(t)(\alpha)\;=\;\bigl\{\,\sum_{i=N}^\infty f_i\alpha^i\in\Fq(t)\dpl\alpha\dpr:\;\exists\,m,n\in\Z_{>0}\text{ with }f_{i+n}=f_i\;\forall\,i\ge m\,\bigr\}\,.
\]
We conclude $e=h$ and $f=0$, that is, $\End(\ulM)=\Fq[t]=A$.
}

The same argument shows that $\ulM$ is not even semisimple. Namely, the projection $\ulM\to\ulM'$ has no section $\ulM'\to\ulM, 1\mapsto{f\choose 1}$, since there is no solution $f$ for the equation $\alpha t\s f+t=tf$.

It is also not hard to compute $F_v$ for instance at the place $v=(t-1)$. Let $z=t-1$ and $\beta\in L^\sep$ with $\beta^{q-1}=\alpha$, and consider the basis ${y/\beta\choose 0},{x\choose y}$ of the Tate module $T_v(\ulM)$ with 
\[
{x\choose y}\;=\;\sum_{i=0}^\infty{x_i\choose y_i}z^i\quad\text{and}\quad x_i,y_i\in L^\sep\,,\;y_0\ne0\,.
\]
They are subject to the equations $y=t\s y=(1+z)\s y$ and $x=\alpha t\s x+t\s y=\alpha(1+z)\s x+y$, that is,
\begin{eqnarray*}
y_i-y_i^q&=&y_{i-1}^q\,,\text{ and}\\[1mm]
x_i-\alpha x_i^q&=&\alpha x_{i-1}^q+y_i\,.
\end{eqnarray*}
There are elements $\gamma$ and $\delta$ of $G=\Gal(L^\sep/L)$ operating as $\gamma(y_i)=y_i$, $\gamma(x_i)=x_i$, $\gamma(\beta)=\beta/\eta$ for an $\eta\in\Fq^\times\setminus\{1\}$, respectively as $\delta(y_i)=y_i$, $\delta(\beta)=\beta$, $\delta(x_i)=x_i+y_i/\beta$. With respect to our basis of $T_v(\ulM)$ they correspond to matrices $\gamma_v=\matr{\eta}{0}{0}{1}$ and $\delta_v=\matr{1}{0}{1}{1}$. We conclude that $F_v$ is the $Q_v$-algebra of upper triangular matrices. Its commutant in $M_2(Q_v)$ equals $Q_v\cdot\Id_2\cong\End(\ulM)\otimes_A Q_v$.
\end{Example}

\noindent
{\it Remark.} If $q=2$ any pure Anderson motive of rank $\rk\ulM=2$ on $A=\BF_q[t]$, which is not semisimple has  $\End(\ulM)\supsetneq A$. One easily sees this by choosing a basis of $\ulM$ for which $\t$ has the form $\matr{\alpha (t-\theta)^d}{0}{\ast}{\beta(t-\theta)^d}$ with $\alpha,\beta,\theta\in L$. Then $\matr{0}{0}{\beta/\alpha}{0}$ is an endomorphism.

However, we expect that also for $q=2$ there are examples similar to \ref{Ex3.10c} (of $\rk\ulM\ge3$), although we have not tried to find one.

\bigskip

Let\/ $\FF$ be an abelian $\tau$-sheaf over $\Fs$ and let\/ $\Fss/\Fs$ be a finite field extension. The base extension
\[
\FF\xFss := (\F_i\otimes_{\O_{\CFs}}\!\!\!\O_{C_{\Fs\mbox{\raisebox{-0.5ex}{$\scriptscriptstyle'$}}}}, \P_i\otimes1,\t_i\otimes1)
\]
is an abelian $\tau$-sheaf over $\Fss$ with $\pi'=(\pi\otimes1)^t$ for $s'=s^t$, and we have a canonical isomorphism between $\VvFF$ and $\VvFF'$.

\bigskip

For the next result recall that an endomorphism $\varphi$ of a finite dimensional vector space $V$ over a field $K$ is called \emph{absolutely semisimple} if for every field extension $K'/K$ the endomorphism $\varphi\otimes1\in\End_{K'}(V\otimes_K K')$ is semisimple. The following characterization is taken from \cite[Proposition~9.2/4 and Proposition~9.2/5]{Bou}.

\begin{Lemma}\label{XY99}
Let\/ $K$ be a field and let\/ $V$ be a finite dimensional $K$-vector space. Let\/ $\varphi\in\End_K(V)$ be an endomorphism. 
\begin{suchthat}
\item $\varphi$ is absolutely semisimple, if and only if there exists a perfect field extension\/ $K'/K$ such that $\varphi\otimes1\in\End_{K'}(V\otimes_K K')$ is semisimple.
\item $\varphi$ is absolutely semisimple, if and only if its minimal polynomial is separable.
\end{suchthat}
\end{Lemma}

\medskip

\begin{Theorem}\label{Thm3.8b}
Let $\FF$ be an abelian $\tau$-sheaf over the finite field $\Fs$. Then there exists a finite field extension\/ $\Fss/\Fs$ whose degree is a power of $\charakt\Fs$ such that\/ $\FF\xFss$ has an absolutely semisimple Frobenius endomorphism. Thus if moreover\/ $\chr\ne\infty$ then $\FF\xFss$ is semisimple.
\end{Theorem}

\noindent
{\it Remark.} It suffices to take $[\BF_{s'}:\BF_s]$ as the smallest power of $\charakt\BF_s$ which is $\ge\rk\FF$.

\begin{proof}
Let $s'=s^t$ for some arbitrary $t\in\N$. Let $\FF':=\FF\xFss$ be the abelian $\tau$-sheaf over $\Fss$ induced by $\FF$. Let $v\in\Spec A$ be a place different from $\chr$. Over $\Qv^\alg$ we can write $\pi_v\in\End_\Qv(\VvFF)$ in Jordan normal form 
\[
\makebox[5em][r]{$B^{-1}\,(\pi_v\otimes1)\,B$} 
\;=\; 
\left(\begin{array}{c@{\;\;\;}c@{\;\;\;}c@{\;\;\;}c}
\lambda_1 & * & & 0 \\
& \lambda_2 & \ddots & \\
& & \ddots & * \\
0 & & & \lambda_r \\
\end{array}\right)
\]
for $B\in GL_r(\Qv^\alg)$ and for some $\lambda_j\in\Qv^\alg$, $1\le j\le r$. Thus, by a suitable choice of $t\in\N$ as a power of $\charakt\Fq$ (as in the remark), we can achieve that $\pi'_v=(\pi_v\otimes1)^t$ is of the form
\[
\makebox[5em][r]{$B^{-1}\,\pi'_v\,B$} 
\;=\; 
\left(\begin{array}{c@{\;\;\;}c@{\;\;\;}c@{\;\;\;}c}
\lambda_1^t & & & 0 \\
& \lambda_2^t & & \\
& & \ddots & \\
0 & & & \lambda_r^t \\
\end{array}\right)\;.
\]
Since $\Qv^\alg$ is perfect, we conclude by \ref{XY99}/1 that $\pi'_v$ and thus $\pi'$ is absolutely semisimple.
\end{proof}

The following corollary illustrates that, in contrast to endomorphisms of vector spaces, there is no need of the term \quotes{absolutely semisimple} for abelian $\tau$-sheaves or pure Anderson motives over finite fields.

\begin{Corollary} \label{Cor3.8c}
Let\/ $\FF$ be an abelian $\tau$-sheaf over $\Fs$ of characteristic different from $\infty$. If\/ $\FF$ is semisimple, then\/ $\FF\xFss$ is semisimple for every finite field extension $\Fss/\Fs$. The same is true for pure Anderson motives.
\end{Corollary}

\begin{proof}
Let $\FF$ be semisimple and let $\Fss/\Fs$ be a finite field extension with $s'=s^t$. We set $\FF':=\FF\xFss$. By \ref{Thm3.8} and \ref{PROP.EQ}, we know that $\QEnd(\FF)\xQv \cong \End_{\Qv\II{\pi_v}}(\VvFF)$ is semisimple. Since $\Qv\II{\pi_v^t}\subset\Qv\II{\pi_v}$ we conclude by \cite[Corollaire de Proposition~6.4/9]{Bou} that $\Qv\II{\pi_v^t}$ is semisimple, as well. As $\VvFF'=\VvFF$, we have $\pi'_v=\pi_v^t$, and therefore $\pi'_v$ is semisimple. Thus, by \ref{PROP.EQ}, $\QEnd(\FF')$ is semisimple and $\FF'$ is semisimple by \ref{Thm3.8}/2.
\end{proof}


\bigskip

\section{Zeta Functions and Reduced Norms} \label{SectZeta}
\setcounter{equation}{0}

In this section we generalize Gekeler's results \cite{Gekeler} on Zeta functions for Drinfeld modules to pure Anderson motives.
But let us begin by recalling a few facts about reduced norms; see for instance \cite[\S 9]{Re}. Let $\ulM$ be a semisimple pure Anderson motive over a finite field and let $\pi$ be its Frobenius endomorphism. Then $F=Q(\pi)$ is the center of the semisimple algebra $E$ by Corollary~\ref{CorFisCenter}. Write $F=\bigoplus_i F_i$ and $E=\bigoplus_i E_i$ where the $F_i$ are fields and $E_i$ is central simple over $F_i$. Note that by~\ref{Thm3.8} the pure Anderson motive $\ulM$ decomposes correspondingly up to isogeny $\ulM\approx\bigoplus_i\ulM_i$ with $E_i=\End(\ulM_i)\otimes_AQ$. We apply \ref{Cor3.11b} to $\ulM_i$ and obtain $\sum_i[E_i:F_i]^{1/2}\cdot[F_i:Q]=r$. Let $f\in E$ and write it as $f=\sum_i f_i$ with $f_i\in E_i$. Choose for each $i$ a splitting field $K_i$ of $E_i$ with $\alpha_i:E_i\otimes_{F_i}K_i\isoto M_{n_i}(K_i)$ where $n_i^2=[E_i:F_i]$. The \emph{reduced norm of $f$} is then defined by
\[
N(f)\;:=\;nr_{E/Q}(f)\;:=\;\prod_i N_{F_i/Q}\bigl(\det\alpha_i(f_i\otimes 1)\bigr),
\]
where $N_{F_i/Q}$ is the usual field norm. The reduced norm is an element of $Q$ which is independent of the choices of $K_i$ and $\alpha_i$. It satisfies $N(a)=a^r$ for all $a\in Q$, and $N(f)\ne0$ if and only if $f\in E^{\SSC\times}$, that is, $f$ is a quasi-isogeny. If $f\in \End(\ulM)$ or more generally $f$ is contained in a finite $A$-algebra then $N(f)\in A$ since $A$ is normal.

\begin{Theorem}\label{Thm3.4.1a}
Let $\FF$ be a semisimple abelian $\tau$-sheaf over a finite field $L$ and let $f\in \QEnd(\FF)$ be a quasi-isogeny. Then
for any place $v\ne\chr,\infty$ of $Q$ we have $N(f)=\det V_vf$, the determinant of the endomorphism $V_vf\in\End_{Q_v}(V_v\FF)$. For $v=\infty\ne\chr$ we have $N(f)^l=\det V_\infty f$, where $l$ comes from Definition~\ref{DefTateMod} and satisfies $\dim_{Q_\infty}V_\infty\FF=l\cdot\rk\FF$.
\end{Theorem}

\begin{proof}
Clearly, if $t$ is a power of $q$ then $N(f^t)=\det V_vf^t$ implies $N(f)=\det V_vf$ since $1$ is the only $t$-th root of unity in $Q_v$ for $v\ne\infty$, and likewise for $v=\infty$.
Writing $V_vf$ in Jordan canonical form over $Q_v^\alg$ we find as in the proof of Theorem~\ref{Thm3.8b} a power $t$ of $q$ such that $V_vf^t$ is absolutely semisimple over $Q_v$ and hence its minimal polynomial is separable by \ref{XY99}.
Then $F_v(f^t)$ and $F(f^t)$ are semisimple by \cite[Proposition 9.1/1 and Corollaire 7.7/4]{Bou}. We now replace $f$ by $f^t$ and thus assume that $F(f)$ is semisimple. 

As is well known there is a semisimple commutative subalgebra $H=\bigoplus_i H_i$ of $E$ containing $F(f)$ with $\dim_{F_i}H_i=n_i$ and hence $\dim_QH=r$. Then $nr_{E/Q}(f)$ equals the determinant of the $Q$-endomorphism $\tilde f:x\mapsto fx$ of $H$. The reason for this is that $H_i\otimes_{F_i}K_i$ is still semisimple and commutative if we choose a splitting field $K_i$ which is separable over $F_i$. By Lemma~\ref{Lemma3.4.1b} below $H_i\otimes_{F_i}K_i$ is isomorphic to $K_i^{n_i}$ as left $H_i\otimes_{F_i}K_i$-modules, and this implies that $nr_{E_i/F_i}(f_i)=\det\alpha_i(f_i)=\det\tilde f_i$, the determinant of the $F_i$-endomorphism $\tilde f_i:x\mapsto f_i x$ of $H_i$, and $N(f)=\det\tilde f$ the determinant of the $Q$-endomorphism $\tilde f$ of $H$.

If $v\ne\infty$ then again by Lemma~\ref{Lemma3.4.1b}, $H_v$ is $H_v$-isomorphic to $V_v\FF$ and $N(f)=\det\tilde f=\det V_vf$.

If $v=\infty$ we embed $E_\infty^{\oplus l}$ into $\End_{Q_{\infty,L}[\phi]}\bigl(\ulTN_\infty(\FF)\bigr)$. Namely, if $(f^{(0)},\ldots,f^{(l-1)})\in E_\infty^{\oplus l}$, where $f^{(m)}=\bigl(f^{(m)}_i:\F_i\otimes_{\O_{C_L}}Q_{\infty,L}\to\F_i\otimes_{\O_{C_L}}Q_{\infty,L}\bigr)$, we set
\[
g_{ij}\;:=\;\left\{
\begin{array}{ll}
\P_{i-1}\circ\ldots\circ\P_j\circ f_j^{(i-j)} & \text{if }0\le j\le i\le l-1 \\[2mm]
z^k\P_i^{-1}\circ\ldots\circ\P_{j-1}^{-1}\circ f_j^{(l+i-j)} & \text{if }0\le i<j\le l-1\;.
\end{array}
\right.
\]
Then $g_{ij}:\F_j\otimes_{\O_{C_L}}Q_{\infty,L}\to\F_i\otimes_{\O_{C_L}}Q_{\infty,L}$ and a straightforward computation shows that the homomorphism $g=(g_{ij})_{i,j=0\ldots l-1}$ commutes with $\phi$ from (\ref{EQ.Phi}) on page~\pageref{EQ.Phi}, that is, $g$ is an element of $\End_{Q_{\infty,L}[\phi]}\bigl(\ulTN_\infty(\FF)\bigr)=\End_{Q_\infty[G]}(V_\infty\FF)$; use Proposition~\ref{Prop2.13'}. Now we apply Lemma~\ref{Lemma3.4.1b} to $H_\infty^{\oplus l}\subset E_\infty^{\oplus l}\subset\End_{Q_\infty}(V_\infty\FF)$, and we compute $N(f)^l=(\det\tilde f)^l=\det_{Q_\infty}\bigl(H_\infty^{\oplus l}\to H_\infty^{\oplus l}, h\mapsto fh\bigr)=\det V_\infty f$ as desired.
\end{proof}

\begin{Lemma}\label{Lemma3.4.1b}
Let $K$ be a field and let $H\subset M_n(K)$ be a semisimple commutative $K$-algebra with $\dim_K H=n$. Then as a (left) module over itself $H$ is isomorphic to $K^n$.
\end{Lemma}

\begin{proof}
Decomposing $H$ into a direct sum of fields $\bigoplus_\kappa L_\kappa$ and $K^n$ into a direct sum $\bigoplus_\lambda V_\lambda$ of simple $H$-modules, each $V_\lambda$ is isomorphic to an $L_{\kappa(\lambda)}$. The injectivity of $H\to M_n(K)$ and $\dim_K H=n$ imply that $H$ is isomorphic to $\bigoplus_\lambda \End_{L_{\kappa(\lambda)}}(V_\lambda)$ and a fortiori isomorphic as left module over itself to $K^n$.
\end{proof}

\begin{Theorem}\label{Prop3.4.1}
Let $\ulM$ be a semisimple pure Anderson motive of rank $r$ over a finite field $L$ and let $f\in \End(\ulM)$ be an isogeny. Then
\begin{suchthat}
\item 
$\dim_L\coker N(f)=r\cdot\dim_L\coker f$.
\item 
The ideal $\deg(f)=N(f)\cdot A$ is principal and has a canonical generator.
\item 
There exists a canonical dual isogeny $\dual{f}\in\End(\ulM)$ satisfying $f\circ\dual{f}=N(f)=\dual{f}\circ f$.
\end{suchthat}
\end{Theorem}

\noindent
{\it Remark.} 1. This shows that $N(1-\pi^n)\in A$ is the analogue for pure Anderson motives of the number of rational points $X(\BF_{q^n})=\deg(1-{\rm Frob}_q^n)\in\Z$ on an abelian variety $X$ over the finite field $\Fq$; see also Theorem~\ref{ThmZeta} below.

2. The dual isogeny satisfies $\dual{(fg)}=\dual{g}\dual{f}$, because $N(fg)=N(f)N(g)$. Note however, that we cannot expect that $\dual{(f+g)} =\dual{f}+\dual{g}$ unless $r=2$ because for $f=a\in A$ we have $N(a)=a^r$ and $\dual{a}=a^{r-1}$.

\begin{proof}
1. Clearly for any $a\in A$ we have $\dim_L\ulM/a\ulM=r\cdot\dim_\Fq A/(a)=-r\cdot\infty(a)$ where $\infty(a)$ denotes the $\infty$-adic valuation of $a$. Now let $\FF$ be an abelian $\tau$-sheaf with $\ulM=\ulM(\FF)$, and let $f:\FF\to\FF(n\cdot\infty)$ for some $n$ be the isogeny induced by $f$. Using Theorem~\ref{Thm3.4.1a} we compute the dimension
\begin{eqnarray*}
l\cdot\dim_L\coker f & =& nrl-\dim_L{\TS\bigoplus_{j=0}^{l-1}\bigl(\F_j(n\cdot\infty)/f_j(\F_j)}\bigr)_\infty\es\\[1mm]
& = & nrl-\dim_L \ulTM_\infty\bigl(\FF(n\!\cdot\!\infty)\bigr)/\ulTM_\infty(f)\bigl(\ulTM_\infty(\FF)\bigr)\\[1mm]
& = & nrl-\dim_\Fq\bigl(T_\infty\FF(n\cdot\infty)/T_\infty f(T_\infty \FF)\bigr)\\[1mm]
& = & -\infty(\det V_\infty f)\es =\es -l\cdot\infty\bigl(N(f)\bigr)\,.
\end{eqnarray*}
Here the first equality follows from the identities $\F_j(n\cdot\infty)/f_j(\F_j)=\bigl(\F_j(n\cdot\infty)/f_j(\F_j)\bigr)_\infty\oplus\coker f$ and $\dim_L\bigl(\F_j(n\cdot\infty)/f_j(\F_j)\bigr)=\deg\F_j(n\cdot\infty)-\deg f_j(\F_j)=nr$. The second equality is the definition of $\ulTM_\infty$, and the third follows from the isomorphism $\ulTM_\infty(\FF)\otimes_{A_{\infty,L}}A_{\infty,L^\sep}\cong T_\infty\FF\otimes_{A_\infty}A_{\infty,L^\sep}$. The fourth equality follows from the elementary divisor theorem. From this we obtain 1.

\smallskip
\noindent
2. Let $v\ne\chr$ be a maximal ideal of $A$. Using Theorem~\ref{Thm3.4.1a} we compute the $v$-adic valuation of $N(f)$
\[
v(N(f))\;=\;v(\det T_vf)\;=\;\dim_{\BF_v}\bigl(T_v\ulM/T_vf(T_v\ulM)\bigr)\;=\;\dim_{\BF_v}\bigl((\coker f)_v\otimes_L L^\sep\bigr)^\t\;=\;v(\deg f)\,.
\]
Again the second equality follows from the elementary divisor theorem, the third equality comes from the fact that the $\t$-invariants of the $v$-primary  part $(\coker f)_v\otimes_L L^\sep$ are isomorphic to $T_v\ulM/T_vf(T_v\ulM)$, and the last equality is the definition of $\deg f$. From 1 and Lemma~\ref{Lemma1.7.7} we obtain
\begin{eqnarray*}
r\cdot\dim_\Fq A/\deg(f)&=& r\cdot\dim_L\coker f \es=\es \dim_L\coker N(f)\es\\
&=&\dim_L\bigl((A/N(f))^r\otimes_\Fq L\bigr)\es=\es r\cdot\dim_\Fq A/N(f)\,.
\end{eqnarray*}
From the identity $\dim_\Fq A/\Fa=\sum_v[\BF_v:\Fq]\cdot v(\Fa)$ for any ideal $\Fa\subset A$ we conclude $\chr(\deg f)=\chr(N(f))$ and therefore $\deg (f)=N(f)\cdot A$.

\smallskip
\noindent
Finally 3 is immediate since $N(f)$ annihilates $\coker f$ by Proposition~\ref{Prop3.28a}.
\end{proof}

\begin{Remark}\label{Rem3.32b}
We do not know of a proof of 1 and 2 for arbitrary pure Anderson motives which does not make use of the associated abelian $\tau$-sheaf $\FF$. In the special case when $\ulM$ comes from a Drinfeld module, Gekeler~\cite[Lemma 3.1]{Gekeler} argued that both sides of the equation in 2 are extensions to $E$ of the $\infty$-adic valuation on $Q$. But this argument fails in general, since there may be more than one such extension as one sees from Example~\ref{Ex3.15} below.
\end{Remark}

\begin{Corollary}\label{Cor3.4.1c}
Let $\ulM$ be a semisimple pure Anderson motive of dimension $d$ over a finite field $L$ and let $\pi$ be its Frobenius endomorphism. Let $v\ne\chr$ be a maximal ideal of $A$ and let $\chi_v$ be the characteristic polynomial of $\pi_v$. Then
\begin{suchthat}
\item
$\chi_v\in A[x]$ is independent of $v$ and $\chi_v(a)\cdot A=\det V_v(a-\pi)\cdot A=\deg(a-\pi)$ for every $a\in A$,
\item 
$\chr^{d\cdot[L:\BF_\chr]}=\deg(\pi)=\chi_v(0)\cdot A=N(\pi)\cdot A$ is principal.
\end{suchthat}
\end{Corollary}

\begin{proof}
1 is a direct consequence of Theorems~\ref{Thm3.4.1a} and \ref{Prop3.4.1} and the Lagrange interpolation theorem applied to the fact that $\chi_v(a)=N(a-\pi)=\chi_w(a)\in A$ for all $a\in A$.\\
2 follows from the fact that $\coker\pi$ is supported on $\chr$ and from the equation $\dim_L\coker\pi=[L:\Fq]\cdot\dim_L\coker\t=d\cdot[L:\Fq]$.
\end{proof}

\begin{Definition}\label{DefZetaFkt}
We define the \emph{Zeta function} of a pure Anderson motive $\ulM$ over a finite field $\Fs$ as
\[
Z_\ulM(t)\;:=\;\prod_{0\le i\le r} \det(1-t\wedge^i\pi_v)^{(-1)^{i+1}}
\]
where $\chr\ne v\in \Spec A$ is a maximal ideal and $\wedge^i\pi_v\in\End_{Q_v}(\wedge^iV_v\ulM)$. 
\end{Definition}

By \ref{Cor3.4.1c}/1 the Zeta function $Z_\ulM(t)$ is independent of the place $v$ and lies in $Q(t)$. This also follows from work of B\"ockle~\cite{Boeckle02} and Gardeyn~\cite[\S7]{Gardeyn2}. The name ``Zeta function'' is justified by the following theorem (see also the remark after Theorem~\ref{Prop3.4.1}).

\begin{Theorem} \label{ThmZeta}
If $\ulM$ is semisimple and $\sum_i a_it^i$ is the power series expansion of $t\frac{d}{dt}\log Z_\ulM(t)$, then $a_i=N(1-\pi^i)\in A$.
\end{Theorem}

\begin{proof}
By standard arguments $a_i=\det(1-\pi_v^i)$; see \cite[Lemma 5.6]{Gekeler}. Now our assertion follows from Theorem~\ref{Thm3.4.1a} 
\end{proof}

This Zeta function satisfies the Riemann hypothesis:

\begin{Theorem}\label{ThmRH}
In an algebraic closure of $Q_\infty$ all eigenvalues of $\wedge^i\pi_v\in\End_{Q_v}(\wedge^iV_v\ulM)$ have the same absolute value $(\#\Fs)^{i\weight(\ulM)}$.
\end{Theorem}

\begin{proof}
This was proved by Goss~\cite[Theorem~5.6.10]{Go} for $i=1$ and follows for the remaining $i$ by general arguments of linear algebra.
\end{proof}


\bigskip

\section{A Quasi-Isogeny Criterion}
\setcounter{equation}{0}

Similarly to the theory for abelian varieties, the characteristic polynomials of the Frobenius endomorphisms on the associated Tate modules play an important role for the study of abelian $\tau$-sheaves. For example, we can decide on quasi-isogeny of two abelian $\tau$-sheaves $\FF$ and $\FF'$ just by considering these characteristic polynomials.

\begin{Theorem}\label{THEOREM-1}
Let\/ $\FF$ and\/ $\FF'$ be abelian $\tau$-sheaves over $\Fs$ with respective Frobenius endomorphisms\/ $\pi$ and\/ $\pi'$, and let $\mu_\pi$ and $\mu_{\pi'}$ be their minimal polynomials over $Q$. Let\/ $v\in C$ be a place different from $\infty$ and $\chr$. Let\/ $\chi_v$ and\/ $\chi'_v$ be the characteristic polynomials of\/ $\pi_v$ and\/ $\pi'_v$, respectively, and let $G:=\Gal(L^\sep/L)$. Assume in addition that $\chr\ne\infty$, or that $\FF$ and $\FF'$ have the same weight. 
\begin{suchthat}
\item Consider the following statements:
\begin{suchthat}
\item[\labelenumi 1. ] $\FF'$ is quasi-isogenous to an abelian factor $\tau$-sheaf of\/ $\FF$.
\item[\labelenumi 2. ] $\VvFF'$ is $G$-isomorphic to a $G$-factor space of\/ $\VvFF$.
\item[\labelenumi 3. ] $\chi'_v$ divides $\chi_v$ in $Q_v[x]$.
\item[\labelenumi 4. ] $\mu_{\pi'}$ divides $\mu_\pi$ in $Q[x]$ and $\rk\FF'\le\rk\FF$
\end{suchthat}
\noindent
$\begin{array}{rc@{\:\:\:}c@{\:\:\:}c@{\:\:}c@{\:\:}c@{\:\:}c@{\:\:}c@{\:\:\quad}l}
\mbox{We have}&\mbox{1.1} &\Rightarrow& \mbox{1.2} &\Rightarrow& \mbox{1.3} & \mbox{and} & \mbox{1.4} & \makebox[0.5\textwidth][l]{always,} \\
& & & \mbox{1.2} &\Leftarrow & \mbox{1.3} & & & \mbox{if\/ $\pi_v$ and\/ $\pi'_v$ are semisimple,} \\
& & & \mbox{1.2} & \Leftarrow & \mbox{1.3} &\Leftarrow& \mbox{1.4} & \mbox{if\/ $\mu_\pi$ is irreducible in\/ $Q[x]$,}\\
&\mbox{1.1} &\Leftarrow & \mbox{1.2} & & & & & \mbox{if the characteristic is different from $\infty$.}
\end{array}$
\smallskip
\item Consider the following statements:
\begin{suchthat}
\item[\labelenumi 1. ] $\FF$ and $\FF'$ are quasi-isogenous.
\item[\labelenumi 2. ] $\VvFF$ and $\VvFF'$ are $G$-isomorphic.
\item[\labelenumi 3. ] $\chi_v=\chi'_v$.
\item[\labelenumi 4. ] $\mu_\pi=\mu_{\pi'}$ and $\rk\FF=\rk\FF'$.
\item[\labelenumi 5. ] There is an isomorphism of $Q$-algebras $\QEnd(\FF)\cong\QEnd(\FF')$ mapping $\pi$ to $\pi'$.
\item[\labelenumi 6. ] There is a $Q_v$-isomorphism $\QEnd(\FF)\otimes_Q Q_v\cong\QEnd(\FF')\otimes_Q Q_v$ mapping $\pi_v$ to $\pi'_v$.
\item[\labelenumi 7. ] If $\chr\ne\infty$ also consider the statement $Z_{\ulM(\FF)}=Z_{\ulM(\FF')}$.
\end{suchthat}
\noindent
$\begin{array}{rc@{\:\:}c@{\:\:}c@{\:\:}c@{\:\:}c@{\:\:}c@{\:\:}c@{\:\:}l@{\:\:}c@{\:\:}c@{\:\:\quad}l}
\mbox{\hspace{-1ex}We have}&\mbox{2.1} &\Leftrightarrow& \mbox{2.2} &\Rightarrow& \multicolumn{4}{l}{\mbox{\hspace{-1.07ex}2.3\,,\,2.4\,,\,2.5}} & & & \makebox[0.5\textwidth][l]{always,} \\
& & & & & & & \multicolumn{2}{r}{\mbox{2.5}} &\Rightarrow &\mbox{2.6} & \makebox[0.5\textwidth][l]{always,} \\
& & & & & \multicolumn{5}{l}{\hspace{-1.07ex}\mbox{2.3} \quad\Leftrightarrow \quad \mbox{2.7}} & & \mbox{if the characteristic is different from $\infty$,} \\
& & & \mbox{2.2} &\Leftarrow & \mbox{2.3}& &\Leftarrow & & &\mbox{2.6} & \mbox{if\/ $\pi_v$ and\/ $\pi'_v$ are semisimple,} \\
& & & \mbox{2.2} & \Leftarrow & \mbox{2.3}&\Leftarrow& \mbox{2.4} & \multicolumn{2}{l}{\hspace{-1.1ex}\Leftarrow} & \mbox{2.6} & \mbox{if $\mu_\pi$ and $\mu_{\pi'}$ are irreducible in $Q[x]$.}
\end{array}
$
\end{suchthat}
\end{Theorem}

\begin{proof}
1. For the implication $1.1\Rightarrow 1.2$ without loss of generality, $\FF'$ can itself be considered as abelian factor $\tau$-sheaf of $\FF$ and the implication follows from Proposition~\ref{FACTORSHEAF-FACTORSPACE}. The implication $1.2\Rightarrow 1.3$ is obvious.

For $1.2\Rightarrow 1.4$ note that $\mu_\pi$ is also the minimal polynomial of $\pi_v$ over $Q_v$ by Lemma~\ref{Lemma3.4}. By Proposition~\ref{PI-IS-QISOG} statement 1.2 implies $\mu_{\pi}(\pi'_v)=0$, whence 1.4.

For $1.3\Rightarrow 1.2$ let $\pi_v$ and $\pi'_v$ be semisimple. Let $\chi_v=\mu_1\cdot\ldots\cdot\mu_n$ and $\chi'_v=\mu'_1\cdot\ldots\cdot\mu'_{n'}$ be the factorization in $Q_v[x]$ into irreducible factors and set $V_i:=Q_v[x]/(\mu_i)$ and $V'_i:=Q_v[x]/(\mu'_i)$.
Then we can decompose $\VvFF=V_1\oplus\cdots\oplus V_n$ and $\VvFF'=V'_1\oplus\cdots\oplus V'_{n'}$. Since $\chi'_v$ divides $\chi_v$, we can now easily construct a surjective $G$-morphism from $\VvFF$ onto $\VvFF'$ which gives the desired result.

Next if $\mu_\pi$ is irreducible, 1.4 implies $\mu_{\pi'}=\mu_\pi$ and 1.3 follows from Corollary~\ref{Cor3.11b}. It further follows from Proposition~\ref{PROP.EQ} that $\pi_v$ and $\pi'_v$ are semisimple and this implies 1.2 by the above.

For $1.2\Rightarrow 1.1$ we first do not assume that $\chr\ne\infty$. Let $f_v:\, \VvFF\rightarrow\VvFF'$ be a surjective morphism of $\QvG$-modules. We may multiply $f_v$ by a suitable power of $v$ to get a morphism $f_v:\, T_v\FF\rightarrow T_v\FF'$ of the integral Tate modules which is not necessarily surjective, but satisfies $v^n T_v\FF'\subset f_v(T_v\FF)$ for a sufficiently large $n$. Let $\ulM:=\bigl(\Gamma(C_L\setminus\{\infty\},\F_0),\,\P_0^{-1}\circ\t\bigr)$. This is a ``$\t$-module on $A$'' in the sense of \cite[\BHAuu]{BH_A}. If $\chr\ne\infty$ then $\ulM$ is the pure Anderson motive $\ulM(\FF)$ associated with $\FF$ in (\ref{Eq1.1}). Also let $\ulM':=\bigl(\Gamma(C_L\setminus\{\infty\},\F'_0),\,\P_0'{}^{-1}\circ\t'\bigr)$. By \cite[\BHAtt]{BH_A} (or Theorem~\ref{TATE-CONJECTURE-MODULES} if $\chr\ne\infty$), $f_v$ lies inside $\Hom(\ulM,\ulM')\otimes_A A_v$, so we can approximate $f_v$ by some $f\in\Hom(\ulM,\ulM')$ with $T_v(f)\equiv f_v$ modulo $v^{n+1} T_v\ulM'$. Since $v^n T_v\ulM'\subset f_v(T_v \ulM)$ we find inside $\im T_v(f)$ generators of $v^n T_v\ulM'/v^{n+1} T_v\ulM'$. They generate an $\Av$-submodule of $v^nT_v\ulM'$ whose rank must at least be $r'$ since $v^n T_v\ulM'/v^{n+1} T_v\ulM'\cong (\Av/v\Av)^{r'}$. Thus $\im T_v(f)$ has rank $r'$. Either by assumption or by \cite[\BHApp]{BH_A} if $\chr\ne\infty$, both $\FF$ and $\FF'$ have the same weight. So by \cite[\BHAkk/1]{BH_A}, $f$ comes from a quasi-morphism $f\in\QHom(\FF,\FF')$, that is, a morphism $f:\FF\to\FF'(D)$ for a suitable divisor $D$. Now we finally assume that the characteristic is different from $\infty$. By \cite[\BHAbb]{BH_A}, the image $\im \bigl(f:\FF\to\FF'(D)\bigr)$ is an abelian factor $\tau$-sheaf of $\FF$ and $\im f\rightarrow\FF'(D)$ is an injective morphism between abelian $\tau$-sheaves of the same rank and weight, hence an isogeny by Proposition~\ref{PROP.1.42A}.

\medskip

\noindent
2. A large part of 2 follows from 1. We prove the rest. To show $2.2\Rightarrow 2.1$ without the hypothesis on the characteristic, we just replace the last argument of the proof of $1.2\Rightarrow 1.1$ by the following: Since $r=\dim_{Q_v}V_v\FF=\dim_{Q_v}V_v\FF'=r'$, the morphism $f:\FF\to\FF'(D)$ is an injective morphism between abelian $\tau$-sheaves of the same rank and weight, hence an isogeny by Proposition~\ref{PROP.1.42A}.

For the implication $2.1\Rightarrow 2.5$ let $g\in\QIsog(\FF,\FF')$. Then the map $\QEnd(\FF)\to\QEnd(\FF')$ sending $f\mapsto gfg^{-1}$ is an isomorphism with $\pi'=g\pi g^{-1}$. The implication $2.5\Rightarrow 2.6$ is obvious.

For the implication $2.3\Rightarrow 2.7$ note that knowledge of $\chi_v$ yields the knowledge of $\det(1-t\wedge^i\pi_v)$ and thus of $Z_{\ulM(\FF)}$ by linear algebra. Conversely we know from Theorem~\ref{ThmRH} that all zeroes of $\det(1-t\wedge^i\pi_v)$ have absolute value $s^{-i\weight(\FF)}$ in an algebraic closure of $Q_\infty$. So we can recover $\chi_v$ from $Z_{\ulM(\FF)}$ by simply looking at this absolute value. This proves $2.3\Leftarrow 2.7$.

Next if $\pi_v$ and $\pi'_v$ are semisimple $2.6\Rightarrow 2.3$ follows from Lemma~\ref{PROP.15}/2, and $2.3\Rightarrow 2.2$ was already established in 1.

Finally if $\mu_\pi$ and $\mu_{\pi'}$ are irreducible, 2.4 follows from 2.6 by Corollary~\ref{Cor3.11b} since $\mu_\pi$ is also the minimal polynomial of $\pi_v$ over $Q_v$ by Lemma~\ref{Lemma3.4}. Also 2.3 follows from 2.4 by Corollary~\ref{Cor3.11b} and $\pi_v$ and $\pi'_v$ are semisimple, so $2.3\Rightarrow 2.2$ by the above.
\end{proof}


\bigskip

\section{The Endomorphism $Q$-Algebra}
\setcounter{equation}{0}

In this section we study the structure of $\QEnd(\FF)$ for a semisimple abelian $\tau$-sheaf $\FF$ over a finite field and calculate the local Hasse invariants of $\QEnd(\FF)$ as a central simple algebra over $Q(\pi)$. For a detailed introduction to central simple algebras, Hasse invariants and the Brauer group, we refer to \cite[Ch. 7, \S\S 28--31]{Re}.

\begin{Theorem}\label{THEOREM-2}
Let\/ $\FF$ be an abelian $\tau$-sheaf over the finite field $\Fs$ of rank $r$ with semisimple Frobenius endomorphism\/ $\pi$, that is, $Q(\pi)$ is semisimple. Let\/ $v\in C$ be a place different from $\infty$ and from the characteristic point $\chr$. Let\/ $\chi_v$ be the characteristic polynomial of\/ $\pi_v$. 
\begin{suchthat}
\item \label{THEOREM-2.1}
The algebra $F=Q(\pi)$ is the center of the semisimple algebra $E=\QEnd(\FF)$.\vspace{\parsep}
\item \label{THEOREM-2.3}
We have $\quad r \,\le\, \II{E:Q} \,=\, r_\Qv(\chi_v,\chi_v) \,\le\, r^2\;.$
\item \label{THEOREM-2.4}
Consider the following statements:
\begin{suchthat}
\item[\ref{THEOREM-2.4}.1. ] $E=F$.
\item[\ref{THEOREM-2.4}.2. ] $E$ is commutative.
\item[\ref{THEOREM-2.4}.3. ] $\II{F:Q} = r$.
\item[\ref{THEOREM-2.4}.4. ] $\II{E:Q} = r$.
\item[\ref{THEOREM-2.4}.5. ] $\chi_v$ has no multiple factor in $\Qv\II{x}$.
\item[\ref{THEOREM-2.4}.6. ] $\chi_v$ is separable.
\end{suchthat}
\noindent
$\begin{array}{rc@{\:\:}c@{\:\:}c@{\:\:}c@{\:\:}c@{\:\:}c@{\:\:}c@{\:\:}c@{\:\:}c@{\:\:}c@{\:\:}c@{\:\:\quad}l}
\text{We have} &\mbox{\ref{THEOREM-2.4}.1} &\Leftrightarrow& \mbox{\ref{THEOREM-2.4}.2} &\Leftrightarrow& \mbox{\ref{THEOREM-2.4}.3} &\Leftrightarrow& \mbox{\ref{THEOREM-2.4}.4} &\Leftrightarrow& 
\mbox{\ref{THEOREM-2.4}.5} &\Leftarrow & \mbox{\ref{THEOREM-2.4}.6} & \makebox[0.55\textwidth][l]{always,} \\
& & & & & & & & & \mbox{\ref{THEOREM-2.4}.5} &\Rightarrow& \mbox{\ref{THEOREM-2.4}.6} & \mbox{if\/ $\pi_v$ is absolutely semisimple.}
\end{array}$
\smallskip

\item \label{THEOREM-2.5}
Consider the following statements:
\begin{suchthat}
\item[\ref{THEOREM-2.5}.1. ] $F=Q$.
\item[\ref{THEOREM-2.5}.2. ] $E$ is a central simple algebra over $Q$.
\item[\ref{THEOREM-2.5}.3. ] $\II{E:Q} = r^2$.
\item[\ref{THEOREM-2.5}.4. ] $\chi_v$ is the $r$-th power of a linear polynomial in $\Qv\II{x}$.
\item[\ref{THEOREM-2.5}.5. ] $\chi_v$ is purely inseparable.
\end{suchthat}
$\begin{array}{rc@{\:\:}c@{\:\:}c@{\:\:}c@{\:\:}c@{\:\:}c@{\:\:}c@{\:\:}c@{\:\:}c@{\:\:\quad}l}
\text{We have}& \mbox{\ref{THEOREM-2.5}.1} &\Leftrightarrow& \mbox{\ref{THEOREM-2.5}.2} &\Leftrightarrow& \mbox{\ref{THEOREM-2.5}.3} &\Leftrightarrow& 
 \mbox{\ref{THEOREM-2.5}.4} &\Rightarrow& \mbox{\ref{THEOREM-2.5}.5} & \makebox[0.55\textwidth][l]{always,} \\
& & & & & & & \mbox{\ref{THEOREM-2.5}.4} &\Leftarrow & \mbox{\ref{THEOREM-2.5}.5} & \mbox{if\/ $\pi_v$ is absolutely semisimple.}
\end{array}$
\smallskip
If\/ \ref{THEOREM-2.5}.2 holds and moreover the characteristic point $\chr:=c(\Spec\Fs)\in\CFs$ is different from $\infty$, $E$ is characterized by\/ $\inv_\infty E=\weight(\FF)$, $\inv_\chr E=-\weight(\FF)$ and\/ $\inv_v E=0$ for any other place $v\in C$.
\item \label{THEOREM-2.6}
In general the local Hasse invariants of $E$ at the places $v$ of $F$ equal $\inv_v E=-\frac{[\BF_v:\BF_q]}{[\BF_s:\BF_q]}\cdot v(\pi)$. In particular
\[
\inv_v E\,=\,\left\{\begin{array}{ll} 0 & \text{ if }v\nmid\chr\infty\,,\\
\weight(\FF)\cdot[F_v:Q_\infty] & \text{ if }v|\infty \text{ and }\chr\ne\infty\,.
\end{array}\right.
\]
(Here $F_v$ denotes the completion of $F$ at the place $v$ and $\BF_v$ is the residue field of the place $v$.)
\end{suchthat}
\end{Theorem}

\begin{Remark}\label{Rem3.14b}
If $\chr\ne\infty$ and $\FF$ is an elliptic sheaf, that is, $d=1$ and $\ulM(\FF)$ is the Anderson motive of a Drinfeld module, Gekeler~\cite[Theorem~2.9]{Gekeler} has shown that there is exactly one place $v$ of $F$ above $\chr$, and exactly one place $w$ of $F$ above $\infty$, and that $\inv_w E=[F:Q]\cdot\weight(\FF)$ and $\inv_v E=-[F:Q]\cdot\weight(\FF)$.
Note that Gekeler actually computes the Hasse invariants of the endomorphism algebra of the Drinfeld module. So his invariants differ from ours by a minus sign, since passing from Drinfeld modules to abelian $\tau$-sheaves is a contravariant functor, see \cite[Theorem~3.2.1]{BS}.
\end{Remark}

\begin{Corollary}\label{Cor3.13b}
Let $\FF$ be an abelian $\tau$-sheaf over the smallest possible field $L=\Fq$ such that $\QEnd(\FF)$ is a division algebra. Then $\QEnd(\FF)$ is commutative and equals $Q(\pi)$.
\end{Corollary}

\begin{proof}
$\QEnd(\FF)$ is a central division algebra over $F$ by Theorem~\ref{THEOREM-2}, which splits at all places of $F$ by \ref{THEOREM-2}/\ref{THEOREM-2.6}, hence equals $F$.
\end{proof}

\begin{proof}[\Proofof{of Theorem \ref{THEOREM-2}}]
\ref{THEOREM-2.1} was already proved in Corollary~\ref{CorFisCenter}. 

\smallskip\noindent
\ref{THEOREM-2.3}. Let
\[
\chi_v \,=\, \prod_{i=1}^n \,\mu_i^{m_i} \quad\in\Qv\II{x}
\]
with distinct irreducible $\mu_i\in\Qv\II{x}$ and $m_i>0$ for $1\le i\le n$. Then $\sum_{i=1}^n m_i\cdot\deg\mu_i=\deg\chi_v=r$, and by Theorem \ref{Thm3.5a} we have $\II{E:Q}=r_\Qv(\chi_v,\chi_v)=\sum_{i=1}^n m_i^2\cdot\deg\mu_i$. The result now follows from the obvious inequalities
\begin{equation}\label{Eq9.1}
r\;=\;\sum_{i=1}^n m_i\cdot\deg\mu_i \;\stackrel{(1)}{\le}\; 
\sum_{i=1}^n m_i^2\cdot\deg\mu_i \;\stackrel{(2)}{\le}\; 
\left(\sum_{i=1}^n m_i\cdot\deg\mu_i\right)^2 \;=\;r^2\, .
\end{equation}

\smallskip\noindent
\ref{THEOREM-2.4}. Since $F=Z(E)$, the equivalence $\ref{THEOREM-2.4}.1\Leftrightarrow \ref{THEOREM-2.4}.2$ is evident. We have equality in (1) of Equation (\ref{Eq9.1}) if and only if $m_i=1$ for all $1\le i\le s$ which establishes the equivalence $\ref{THEOREM-2.4}.4\Leftrightarrow \ref{THEOREM-2.4}.5$. In order to prove $\ref{THEOREM-2.4}.5\Rightarrow \ref{THEOREM-2.4}.3$ we consider the minimal polynomial $\mu_v$ of $\pi_v$ over $\Qv$. If $\chi_v$ has no multiple factor, then $\mu_v=\chi_v$ and therefore $\II{F:Q}=\II{\Qv(\pi_v):\Qv}=r$. Next $\ref{THEOREM-2.4}.3\Rightarrow \ref{THEOREM-2.4}.1$ because $F\subset E$ and $(\dim_{Q_v}F_v)(\dim_{Q_v}E_v)=\dim_{Q_v}\End_{Q_v}(V_v\FF)=r^2$ by \cite[Th\'eor\`eme 10.2/2]{Bou}, since $E_v$ is the commutant of $F_v$ in $\End_{Q_v}(V_v\FF)$. Note that $\ref{THEOREM-2.4}.3\Rightarrow \ref{THEOREM-2.4}.1$ also follows from Lemma~\ref{Lemma3.4.1b}. Conversely $\ref{THEOREM-2.4}.1\Rightarrow \ref{THEOREM-2.4}.4$ because $E=F$ implies $r\ge\II{\Qv(\pi_v):\Qv}=\II{F:Q}=\II{E:Q}\ge r$. For $\ref{THEOREM-2.4}.5\Rightarrow \ref{THEOREM-2.4}.6$ we use Lemma~\ref{XY99}/2 as we know that $\chi_v=\mu_v$. $\ref{THEOREM-2.4}.6\Rightarrow \ref{THEOREM-2.4}.5$ is clear.

\smallskip\noindent
\ref{THEOREM-2.5}. If $F=Q$, then $E$ is simple with center $Q$, so $E$ is a central simple algebra over $Q$. Since $F=Z(E)$, the converse is obvious. This shows $\ref{THEOREM-2.5}.1\Leftrightarrow \ref{THEOREM-2.5}.2.$ We have equality in (2) of (\ref{Eq9.1}) if and only if $n=1$, $\deg\mu_1=1$ and $m_1=r$ which establishes $\ref{THEOREM-2.5}.3\Leftrightarrow \ref{THEOREM-2.5}.4.$ In order to connect $\ref{THEOREM-2.5}.1\Leftrightarrow \ref{THEOREM-2.5}.2$ with $\ref{THEOREM-2.5}.3\Leftrightarrow \ref{THEOREM-2.5}.4$ let $\chi_v$ be a power of a linear polynomial. By \cite[Proposition~9.1/1]{Bou} the minimal polynomial of $\pi_v$ over $\Qv$ is linear and thus $F=Q$. The converse is trivial. For $\ref{THEOREM-2.5}.5\Rightarrow \ref{THEOREM-2.5}.4$ we use again \ref{XY99}/2 to see that $\mu_v$ is linear. $\ref{THEOREM-2.5}.4\Rightarrow \ref{THEOREM-2.5}.5$ is clear.

The statement about the Hasse invariants follows from \ref{THEOREM-2.6}. Nevertheless, we give a separate proof in case $(k,l)=1$ using Tate modules, since this is much shorter here and exhibits a different technique than \ref{THEOREM-2.6}.
By the Tate conjecture \ref{TATE-CONJECTURE}, $E\otimes_Q Q_v$ is isomorphic to $\End_{Q_v}(V_v\FF)\cong M_r(Q_v)$ for all places $v\in C$ which are different from $\chr$ and $\infty$, so the Hasse invariants of $E$ at these places are $0$. Since the sum of all Hasse invariants is $0$ (modulo $1$), we only need to calculate $\inv_\infty E$.

As a first step, we show that $\Ff_{q^l}$ is contained in $\Fs$. In our situation, $\pi$ lies inside $Q$. Thus, by \ref{ThmRH} we get $s^{k/l}=|\,\pi\,|_\infty=q^m$ for some $m\in\Z$ as $|\,Q_\infty^{\times}\,|_\infty=q^\Z$. Since $q^e=s$, we conclude that $e\cdot k/l=m\in\Z$ and hence $l \mid e$, since $k$ and $l$ are assumed to be relatively prime. Therefore $\Ff_{q^l}\subset\Ff_{q^e}=\Fs$.

Consider the rational Tate module $V_\infty(\FF)$ at $\infty$ and the isomorphism of $Q_\infty$-algebras 
\[
E\otimes_Q Q_\infty\,\cong\,\End_{\Delta_\infty[G]}(V_\infty\FF)\,=\,\End_{\Delta_\infty}(V_\infty\FF)
\]
from Theorem~\ref{TATE-CONJECTURE}.
Since $\dim_{Q_\infty}\Delta_\infty=l^2$ and $\dim_{Q_\infty}V_\infty\FF=rl$, we conclude that $V_\infty\FF$ is a left $r/l$-dimensional $\Delta_\infty$-vector space and hence isomorphic to $\Delta_\infty^{r/l}$. Thus we have 
\[
E\otimes_Q Q_\infty \,\cong\, \End_{\Delta_\infty}(\Delta_\infty^{r/l}) \,=\, M_{r/l}(\End_{\Delta_\infty}(\Delta_\infty))\,=\,M_{r/l}(\Delta_\infty^{\op})\,.
\]
\forget{
We claim that $\End_{\Delta_\infty}(\Delta_\infty) = \Delta_\infty^{\op}$ denoting the opposite division algebra of $\Delta_\infty$. Let $m\in\End_{\Delta_\infty}(\Delta_\infty)$ and consider $f:=m(1)\in\Delta_\infty$. For $g\in\Delta_\infty$ we have
\[
m(g) \,=\, m(g\cdot 1) \,=\, g\cdot m(1) \,=\, g\cdot f
\]
since $m$ is a $\Delta_\infty$-endomorphism. Thus, $m$ is uniquely determined by $f$, and the map $\alpha:\, \End_{\Delta_\infty}(\Delta_\infty)\rightarrow\Delta_\infty$, $m\mapsto f$ is an isomorphism of $Q_\infty$-vector spaces. Let now $n\in\End_{\Delta_\infty}(\Delta_\infty)$ and let $g:=n(1)$. Then
\[
(m\cdot n)(1) \,=\, m(n(1)) \,=\, m(g) \,=\, g\cdot f\;.
\]
This means that, under $\alpha$, the multiplication in $\End_{\Delta_\infty}(\Delta_\infty)$ passes to the opposite multiplication which shows our claim. 
}
Our proof now completes by $\inv_\infty E \,=\, \inv\Delta_\infty^{\op}\,=\, -\inv\Delta_\infty \,=\, \frac{k}{l} \,=\, \weight(\FF)$\,.

\smallskip\noindent
\ref{THEOREM-2.6}. We prove the general case using local (iso-)shtuka rather than Tate modules which were used in \ref{THEOREM-2.5}. Our method is inspired by Milne's and Waterhouse' computation for abelian varieties~\cite[Theorem~8]{WM}. However in the function field case this method can be used to calculate the Hasse invariant at \emph{all} places, whereas in the number field case it applies only to the place which equals the characteristic of the ground field.
Let $w$ be a place of $Q$ and let $\ulN_w:=\ulN_w(\FF)$ be the local $\sigma$-isoshtuka of $\FF$ at $w$. Let $\BF_w$ be the residue field of $w$ and $\BF_{q^f}=\BF_w\cap\Fs$ the intersection inside an algebraic closure of $\Fq$. Let $\Fa_0$ be the ideal $(b\otimes 1-1\otimes b:b\in\BF_{q^f})$ of $Q_w\otimes_\Fq \Fs$ and let $R:=(Q_w\otimes_{\BF_q}\BF_s/\Fa_0)[T]=Q_w\otimes_{\BF_{q^f}}\Fs[T]$ be the non-commutative polynomial ring with $T\cdot(a\otimes b)=(a\otimes b^{q^f})\cdot T$ for $a\in Q_w$ and $b\in\BF_s$. Since $Q_w\otimes_{\BF_{q^f}}\Fs$ is a field, $R$ is a non-commutative principal ideal domain as studied by Jacobson~\cite[Chapter 3]{Jacobson}. Its center is the commutative polynomial ring $Q_w[T^g]$ where $g=[\Fs:\BF_{q^f}]=\frac{e}{f}$. From Theorem~\ref{ThmLS1} and Proposition~\ref{PropLS3} we get isomorphisms
\[
\QEnd(\FF)\otimes_Q Q_w\,\cong\,\End_{Q_w\otimes_\Fq\Fs[\phi]}(\ulN_w)\,\cong\,\End_R(\ulN_w/\Fa_0\ulN_w)
\]
where $T$ operates on $\ulN_w/\Fa_0\ulN_w$ as $\phi^f$.

By \cite[Theorem~3.19]{Jacobson} the $R$-module $\ulN_w/\Fa_0\ulN_w$ decomposes into a finite direct sum indexed by some set $I$
\begin{equation}\label{Eq3.13.6}
\ulN_w/\Fa_0\ulN_w\,\cong\,\bigoplus_{v\in I}\ulN_v^{\oplus n_v}
\end{equation}
of indecomposable $R$-modules $\ulN_v$ with $\ulN_v\not\cong \ulN_{v'}$ for $v\ne v'$. The annihilator of $\ulN_v$ is a two sided ideal of $R$ generated by a central element $\mu_v\in Q_w[T^g]$ by \cite[\S 3.6]{Jacobson}, which can be chosen to be monic. In particular (\ref{Eq3.13.6}) is an isomorphism of $Q_w[T^g]$-modules and $\mu_v$ is the minimal polynomial of $T^g$ on $\ulN_v$ by \cite[Lemma 3.1]{Jacobson}. Therefore the least common multiple $\mu$ of the $\mu_v$ is the minimal polynomial of $T^g$ on $\ulN_w/\Fa_0\ulN_w$. Note that $T^g$ operates on $\ulN_w/\Fa_0\ulN_w$ as the Frobenius $\pi$, hence $\mu=mipo_{\pi|\FF}$ and $F=Q(\pi)=Q[T^g]/(\mu)$, where we write $mipo$ for the minimal polynomial. By the semisimplicity of $\pi$ (and Proposition~\ref{PROP.EQ}) $\mu$ has no multiple factors in $Q_w[T^g]$. Since the $\mu_v$ are powers of irreducible polynomials by \cite[Theorem~3.20]{Jacobson} we conclude that all $\mu_v$ are themselves irreducible in $Q_w[T^g]$. Again \cite[Theorem~3.20]{Jacobson} implies that $\mu_v\ne \mu_{v'}$ since $\ulN_v\not\cong \ulN_{v'}$ and
\[
\mu\,=\,mipo_{\pi|\FF}\,=\,\prod_{v\in I}\mu_v\quad\text{inside}\quad Q_w[T^g]\,.
\]
Thus $F\otimes_Q Q_w=Q_w[T^g]/(\mu)=\prod_{v\in I}Q_w[T^g]/(\mu_v)=\prod_{v|w}F_v$. So $I$ is the set of places of $F$ dividing $w$ and $F_v=Q_w[T^g]/(\mu_v)$ is the completion of $F$ at $v$, justifying our notation. Let $\pi_v$ be the image of $\pi$ in $F_v$. Its minimal polynomial over $Q_w$ is $\mu_v$. This implies that $E\otimes_Q Q_w$ decomposes further
\[
E\otimes_Q Q_w\,=\,\bigoplus_{v\in I}\End_R(\ulN_v^{\oplus n_v})\,=\,\bigoplus_{v\in I}E\otimes_F F_v
\]
and $E\otimes_F F_v\cong\End_R(\ulN_v^{\oplus n_v})$. 

Now fix a place $v$ above $w$ and consider the diagram of field extensions
\[
\xymatrix @R-1pc @C-1pc {& & \BF_v\BF_s\ar@{-}[d] \\
\BF_v\ar@{-}[ddd]\ar@{-}[urr]^{g/h}\ar@{-}[dr] & & \BF_w\BF_s\ar@{-}[dd]^i\\
& \BF_w\BF_s\cap\BF_v \ar@{-}[ur]^{g/h}\ar@{-}[d]\\
& \BF_w(\BF_v\cap\BF_s)\ar@{-}[d]^i & \BF_s\\
\BF_w \ar@{-}[d]^i\ar@{-}[ur]^{h\es} & \BF_v\cap\BF_s\ar@{-}[ur]_{g/h}\\
**{!L !<0.8pc,0pc> =<1.5pc,1.5pc>}\objectbox{\BF_{q^f}=\BF_w\cap\BF_s=\BF_w\cap(\BF_v\cap\BF_s)} \ar@{-}[ru]^h \ar@{-}[d]^f\\
\BF_q
}
\]
Let $h:=[\BF_v\cap\Fs:\BF_{q^f}]=\gcd([\BF_v:\BF_{q^f}],g)$. Let $i:=[\BF_w:\BF_{q^f}]$. From the formulas
\begin{eqnarray*}
[\BF_w\BF_s:\BF_w]&=&[\BF_s:\BF_{q^f}]\es=\es g,\\[1mm]
[\BF_w(\BF_v\cap\BF_s):\BF_w]&=&[\BF_v\cap\BF_s:\BF_{q^f}]\es=\es h,\\[1mm]
[\BF_w\BF_s:(\BF_w\BF_s\cap\BF_v)]&=&[\BF_v\BF_s:\BF_v]\es=\es[\BF_s:\BF_v\cap\BF_s]\es=\es{\TS\frac{g}{h}},\quad\text{and}\\[1mm]
\BF_w(\BF_v\cap\BF_s)&\subset&\BF_w\BF_s\cap\BF_v,
\end{eqnarray*}
we obtain $\BF_w\BF_s\cap\BF_v=\BF_w(\BF_v\cap\BF_s)=\BF_{q^{fhi}}$.
Let $F_{v,L}$ be the compositum of $Q_w\otimes_{\BF_{q^f}}\Fs$ and $F_v$ in an algebraic closure of $Q_w$. Note that $F_{v,L}$ is well defined since $\Fs/\BF_{q^f}$ is Galois. Let $F_{v,L}[T']$ be the non-commutative polynomial ring with
\[
T'\cdot(a\otimes b)=(a\otimes b^{q^{fhi}})\cdot T'\quad\text{and}\quad T'\cdot x= x\cdot T'
\]
for $a\in Q_w$, $b\in\Fs$, and $x\in F_v$ and set $\Delta_v=F_{v,L}[T']/\bigl((T')^{g/h}-\pi_v^i\bigr)$.
Observe that the commutation rules of $T'$ are well defined since $(Q_w\otimes_{\BF_{q^f}}\Fs)\cap F_v$ has residue field $\BF_w\Fs\cap\BF_v=\BF_{q^{fhi}}$ and is unramified over $Q_w$, because $Q_w\otimes_{\BF_{q^f}}\Fs$ is. Moreover, the extension $F_{v,L}/F_v$ is unramified of degree $[\BF_v\BF_s:\BF_v]=\frac{g}{h}$ and $\wt T:=(T')^{[\BF_v:\Fq]/fhi}$ is its Frobenius automorphism. Since $\wt T^{g/h}=\pi_v^{[\BF_v:\Fq]/fh}$ in $\Delta_v$, our $\Delta_v$ is just the cyclic algebra $\bigl(F_{v,L}/F_v,\wt T,\pi_v^{[\BF_v:\Fq]/fh}\bigr)$ and has Hasse invariant $\frac{[\BF_v:\Fq]}{[\Fs:\Fq]}\cdot v(\pi_v)$; compare \cite[p.\ 266]{Re}. We relate $\Delta_v$ to $E\otimes_F F_v$. Firstly by \cite[Theorem~3.20]{Jacobson} there exists a positive integer $u$ such that $\ulN_v^{\oplus u}\cong R/R \mu_v(T^g)$. Therefore
\[
M_u(E\otimes_F F_v)\,\cong\,M_u\bigl(\End_R(\ulN_v^{\oplus n_v})\bigr)\,=\,\End_R(\ulN_v^{\oplus un_v})\,=\,M_{n_v}\bigl((R/R\mu_v(T^g))^\op\bigr)\,.
\]
Secondly we choose integers $m$ and $n$ with $m>0$ and $mi+ng=1$. We claim that the morphism $R/R\mu_v(T^g)\to M_h(\Delta_v)$, which maps
\[
a\otimes b\longmapsto\left(\begin{array}{cccc}a\otimes b\\ & a\otimes b^{q^f}\\ & & \ddots \\ & & & a\otimes b^{q^{f(h-1)}} \end{array}\right)\qquad\text{and}\qquad T\longmapsto 
\pi_v^n\cdot\left( \raisebox{4ex}{$
\xymatrix @R=0.3pc @C=0.7pc {0 \ar@{.}[drdr] & 1\ar@{.}[dr]\\
& & 1 \\
(T')^m & & 0}$} \right)
\]
for $a\in Q_w$ and $b\in\Fs$, is an isomorphism of $F_v$-algebras. It is well defined since it maps $T\cdot(a\otimes b)$ and $(a\otimes b^{q^f})\cdot T$ to the same element because $(T')^m=(T')^{1/i}$ in $\Gal(F_{v,L}/F_v)$, and it maps $T^g=(T^h)^{g/h}$ to $\pi_v^{ng}(T')^{mg/h}\cdot\Id_h=\pi_v\cdot\Id_h$. Since $R\mu_v(T^g)\subset R$ is a maximal two sided ideal the morphism is injective. To prove surjectivity we compare the dimensions as $Q_w$-vector spaces. We compute
\begin{eqnarray*}
\dim_{F_v} M_h(\Delta_v)&=&\TS h^2\cdot(\frac{g}{h})^2\es=\es g^2\,,\\[2mm]
\dim_{Q_w\otimes_{{\SC\BF}_{q^f}}\Fs}\bigl(R/R\mu_v(T^g)\bigr) &=&g\cdot\deg \mu_v\es=\es g\cdot[F_v:Q_w]\,, \quad\text{and}\\[2mm]
\dim_{Q_w}\bigl(R/R\mu_v(T^g)\bigr) &=&g^2\cdot[F_v:Q_w]\es=\es\dim_{Q_w}M_h(\Delta_v)\,.
\end{eqnarray*}
Altogether $M_u(E\otimes_F F_v)\cong M_{hn_v}(\Delta_v^\op)$ and $\inv_v E=-\inv_v\Delta_v=-\frac{[\BF_v:\Fq]}{[\Fs:\Fq]}\cdot v(\pi_v)$ as claimed.

It remains to convert this formula into the special form asserted for $v\nmid\chr\infty$ or $v|\infty$. If $v|\infty$ and $\chr\ne\infty$, let $e_v$ be the ramification index of $F_v/Q_\infty$. Then we get from Theorem~\ref{ThmRH} the formula $q^{e\weight(\FF)}=|\pi|_\infty=q^{-v(\pi_v)/e_v}$, since the residue field of $Q_\infty$ is $\Fq$. This implies as desired
\[
-\frac{[\BF_v:\Fq]}{[\Fs:\Fq]}\cdot v(\pi_v)\,=\,-\frac{[\BF_v:\Fq]\cdot  (-e_ve\cdot\weight(\FF))}{e}\,=\,\weight(\FF)\cdot[F_v:Q_\infty]
\]

Finally if $w\ne\chr,\infty$ is a place of $Q$, the local $\sigma$-shtuka $\ulM_w(\FF)$ at $w$ is \'etale. So $\mu=mipo_{\pi|\FF}$ has coefficients in $A_w$ with constant term in $A_w^\times$. Therefore $v(\pi_v)=0$ for all places $v$ of $F$ dividing $w$.
\end{proof}

\begin{Example}\label{LAST-EXAMPLE}
Let $C=\PP^1_\Fq$, $C\setminus\{\infty\}=\Spec\Fq\II{t}$ and $L=\Fq$. Let $d$ be a positive integer. Let $\F_i:=\O(d\lceil\frac{i}{2}\rceil\cdot\infty)\oplus\O(d\lceil\frac{i-1}{2}\rceil\cdot\infty)$ for $i\in\Z$ and let $\t:=\matr{0}{1}{t^d}{0}$. Then $\FF=(\F_i,\P_i,\t_i)$ is an abelian $\tau$-sheaf of rank $2$, dimension $d$, and characteristic $\chr=V(t)\in\PP^1$ over $\Fq$. Hence the Frobenius endomorphism $\pi$ equals $\t$. If $d$ is odd then $\FF$ is primitive (that means $(d,r)=1$) and therefore simple by \cite[\BHAaa]{BH_A}. In particular, $\pi$ is semisimple. We have
\[
\mu_\pi\,=\,\chi_v \,=\, x^2-t^d \,=\, (x-\sqrt{t^d})(x+\sqrt{t^d})
\]
which means that $\pi_v$ is \emph{not} absolutely semisimple in characteristic $2$. Moreover, we calculate $r_\Qv(\chi_v,\chi_v)=1\cdot1\cdot2=2$ whereas in the field extension $\Qv(\sqrt{t})\,/\,\Qv$ we have
\[
r_{\Qv(\sqrt{t})}(\chi_v,\chi_v) \,=\, \left\{
\begin{array}{l@{\quad}l} 
2\cdot2\cdot1=4 & \mbox{in characteristic 2,} \\ 
1\cdot1\cdot1+1\cdot1\cdot1=2 & \mbox{in characteristic different from 2.} 
\end{array}
\right.
\]
Although the later has no further significance it illustrates the remark after Definition~\ref{Def3.3}. By Theorem \ref{THEOREM-2}/\ref{THEOREM-2.4}. we have $E=F=Q(\pi)$ commutative and $\II{E:Q}=2=r$. Moreover, $|\,\pi\,|_\infty=|\,\sqrt{t^d}\,|_\infty=q^{d/2}$ and $\chi_v$ is irreducible. But $\chi_v$ is not separable in characteristic 2.

If $d=2n$ is even then the minimal polynomial of $\pi$ is
\[
\mu_\pi\,=\,\chi_v \,=\, x^2 -t^d\,=\, (x-t^{d/2})(x+t^{d/2})\,.
\]
So $\pi$ is semisimple if and only if $\charakt(\Fq)\ne2$. In this case $\FF$ is quasi-isogenous to the abelian $\tau$-sheaf $\FF'$ with $\F'_i=\O_{C_L}(in\cdot\infty)^{\oplus 2}$ and $\t'_i=\matr{-t^n}{0}{0}{t^n}$. The quasi-isogeny $f:\FF'\to\FF$ is given by $f_{0,\eta}=\matr{-t^n}{1}{t^n}{1}:\F'_{0,\eta}\isoto\F_{0,\eta}$. The abelian $\tau$-sheaf $\FF'$ equals the direct sum $\FF^{(1)}\oplus\FF^{(2)}$ where $\F_i^{(j)}=\O_{C_L}(in\cdot\infty)$ and $\t_i^{(j)}=(-1)^j t^n$. Note that $\FF^{(1)}$ and $\FF^{(2)}$ are not isogenous over $\Fq$, since the equation $-t^n\cdot\s(g)=g\cdot t^n$ has no solution $g\in Q$ for $\charakt(\Fq)\ne2$. Therefore 
\[
Q\oplus Q\,=\,\bigoplus_{j=1}^2 \QEnd(\FF^{(j)})\,\cong\,E\,=\,F\,=\,Q[x]/(x^2-t^{2n})\,\cong\,Q\oplus Q\,.
\]

Now we consider the same abelian $\tau$-sheaf over $L=\Ff_{q^2}$. This means $\pi=\t^2=t^d\in Q$ and therefore $\chi_v=(x-t^d)^2$. Thus $\pi$ is semisimple. By Theorem \ref{THEOREM-2}/\ref{THEOREM-2.5} we have $F=Q(\pi)=Q$ and $E$ is central simple over $Q$ with $\II{E:Q}=4$ and $\inv_\infty E=\inv_\chr E=\frac{d}{2}$. Moreover, $|\,\pi\,|_\infty=|\,t^d\,|_\infty=q^d$. In this case, $\pi_v$ is absolutely semisimple. Note that if $d$ is even and $\charakt(\Fq)=2$ this is another example for Theorem~\ref{Thm3.8b}. 

If $d$ is odd then $\FF$ is still primitive, whence simple and $E$ is a division algebra. If $d=2n$ is even then the abelian $\tau$-sheaves $\FF^{(1)}$ and $\FF^{(2)}$ defined above are isomorphic $\FF^{(1)}\isoto\FF^{(2)},1\mapsto\lambda$ where $\lambda\in\BF_{q^2}$ satisfies $\lambda^{q-1}=-1$. Therefore $M_2(Q)=M_2\bigl(\QEnd(\FF^{(1)})\bigr)\cong E$ in accordance with the Hasse invariants just computed.
\end{Example}

\begin{Example}\label{Ex3.15}
We compute another example which displays other phenomena. Let $C=\PP^1_\Fq$ and let $C\setminus\{\infty\}=\Spec\Fq[t]$. Let $\F_i=\O_{C_L}(\lceil\frac{i-1}{2}\rceil\cdot\infty)^{\oplus2}\oplus\O_{C_L}(\lceil\frac{i}{2}\rceil\cdot\infty)^{\oplus2}$, let $\P_i$ be the natural inclusion, and let $\t_i$ be given by the matrix
\[
T\,:=\,\left( \begin{array}{cccc}0&0&0&a\\0&b&1&0\\t&0&-b&0\\0&t&0&0 \end{array}\right)\qquad\text{with}\quad a,b\in\Fq\setminus\{0\}\,.
\]
Then $\FF$ is an abelian $\tau$-sheaf of rank $4$ and dimension $2$ with $l=2,k=1$ and characteristic $\chr=V(t)\in\PP^1$. One checks that the minimal polynomial of the matrix $T$ is $x^4-b^2x^2-at^2$ which is irreducible over $Q$ if $\charakt(\Fq)\ne2$, since it has neither zeroes in $\Fq[t]$ nor quadratic factors in $Q[x]$. If $\charakt(\Fq)=2$ then the minimal polynomial is a square and $\FF$ is not semisimple.

For $L=\Fq$ and $2\nmid q$ we obtain $\pi=\t$ semisimple and $E=F=Q(\pi)=Q[x]/(x^4-b^2x^2-at^2)$.

For $L=\BF_{q^2}$ we have $\pi=\t^2$ and the minimal polynomial of $\pi$ over $Q$ is $x^2-b^2x-at^2$, which is irreducible also in characteristic $2$ since it has no zeroes in $\Fq[t]$. Hence $\pi$ is semisimple, $F$ is a field with $[F:Q]=2$ and $[E:F]=4$ by Corollary~\ref{Cor3.11b}. This again illustrates Theorem~\ref{Thm3.8b}. We compute the decomposition of $\infty$ and $\chr$ in $F$.

\smallskip

\noindent
\emph{Decomposition of $\chr$:} Modulo $t$ the polynomial $x^2-b^2x-at^2$ has two zeroes $x=b^2$ and $x=0$ in $\Fq$. So by Hensel's lemma $F\otimes_Q Q_\chr\cong F_v\oplus F_{v'}$ splits with $F_v\cong F_{v'}\cong Q_\chr$ and $v(\pi)=0$ and $v'(\pi)=v'(at^2)=2$. Thus the Hasse invariants of $E$ are $\inv_v E=\inv_{v'}E=0$.

\smallskip

\noindent
\emph{Decomposition of $\infty$:} Set $y=\pi/t$. Then $y^2-\frac{b^2}{t}y-a=0$.\\[1mm]
\emph{Case (a).} If $2|q$ then $(y-a^{q/2})^2-\frac{b^2}{t}(y-a^{q/2})-\frac{b^2}{t}a^{q/2}=0$, that is, $\infty$ ramifies in $F$, $F\otimes_Q Q_\infty=F_w$ with $w(\frac{\pi}{t}-a^{q/2})=1$ and $w(\frac{1}{t})=2\cdot\infty(\frac{1}{t})=2$. So $[F_w:Q_\infty]=2$ and $\inv_w E=0$.

\smallskip

\noindent
\emph{Case (b).} If $2\nmid q$ and $\sqrt{a}\in\Fq$ then the polynomial $y^2-\frac{b^2}{t}y-a$ has two zeroes $y=\pm\sqrt{a}$ modulo $\frac{1}{t}$. So by Hensel's lemma $F\otimes_Q Q_\infty\cong F_w\oplus F_{w'}$ splits with $[F_w:Q_\infty]=[F_{w'}:Q_\infty]=1$. Thus the local Hasse invariants of $E$ are $\inv_w E=\inv_{w'}E=\frac{1}{2}$. As was remarked in \ref{Rem3.14b} such a distribution of the Hasse invariants can occur only if $d\ge2$.

\smallskip

\noindent
\emph{Case (c).} If $2\nmid q$ and $\sqrt{a}\notin\Fq$ then $y^2-\frac{b^2}{t}y-a$ is irreducible modulo $\frac{1}{t}$ and $\infty$ is inert in $F$, $F\otimes_Q Q_\infty=F_w$ with $[F_w:Q_\infty]=2$. Thus the Hasse invariant of $E$ is $\inv_w E=0$.

\medskip

In case (b) $E$ is a division algebra and $\FF$ is simple. In cases (a) and (c) $E\cong M_2(F)$ and $\FF$ is quasi-isogenous to $(\FF')^{\oplus2}$ for an abelian $\tau$-sheaf $\FF'$ of rank $2$, dimension $1$ and $\QEnd(\FF')=F$. This surprising result is due to the fact that $\FF'$, being of dimension $1$, is associated with a Drinfeld module and thus of the form $\F'_i=\O_{C_L}(\lceil\frac{i}{2}\rceil\cdot\infty)\oplus\O_{C_L}(\lceil\frac{i-1}{2}\rceil\cdot\infty)$ with $\t'_i=\matr{c}{d}{t}{0}$ and $c,d\in\BF_{q^2}$. Then $\pi'=(\t')^2=\matr{c^{q+1}+d^qt}{c^qd}{ct}{dt}$ has minimal polynomial $x^2-(c^{q+1}+(d+d^q)t)x+d^{q+1}t^2$ which must be equal to $x^2-b^2x-at^2$. This is possible only if $d+d^q=0$ and $d^{q+1}=-a$. So either $d\in\Fq$ and $2|q$ and we are in case (a), or $d\in\BF_{q^2}\setminus\Fq$, $d^q=-d$, and $a=d^2$. The later implies $2\nmid q$ and $\sqrt{a}=d\notin\Fq$ and we are in case (c). If we choose $c=b$ in case (c) a quasi-isogeny $f:\FF\to(\FF')^{\oplus2}$ over $\BF_{q^2}$ is given for instance by
\[
\left(\begin{array}{cccc} 
d & a & -bd/t & 0 \\
0 & 0 & -d & a \\
0 & 0 & d/t & a/t \\
1 & -d & 0 & bd/t 
\end{array}\right)\,.
\]
\end{Example}


\bigskip

\section{Kernel Ideals for Pure Anderson Motives} \label{Sect1.7}
\setcounter{equation}{0}

In this section we investigate which orders of $E$ can arise as endomorphism rings $\End(\ulM)$ for pure Anderson motives $\ulM$. For this purpose we define for each right ideal of the endomorphism ring $\End(\ulM)$ an isogeny with target $\ulM$ and discuss its properties. This generalizes Gekeler's results for Drinfeld modules \cite[\S 3]{Gekeler} and translates the theory of Waterhouse \cite[\S 3]{Wat} for abelian varieties to the function field case. These two sources are themselves the translation, respectively the higher dimensional generalization of Deuring's work on elliptic curves~\cite{Deuring}.

Let $\ulM$ be a pure Anderson motive over $L$ and abbreviate $R:=\End(\ulM)$. Let $I\subset R$ be a right ideal which is an $A$-lattice in $E:=R\otimes_A Q$. This is equivalent to saying that $I$ contains an isogeny, since every lattice contains some isogeny $a\cdot\id_\ulM$ for $a\in A$ and conversely the existence of an isogeny $f\in I$ implies that the lattice $f\cdot\dual{f}\cdot R$ is contained in $I$.

\begin{Definition}\label{Def1.7.1}
\begin{suchthat}
\item 
Let $\ulM^I$ be the pure Anderson sub-motive of $\ulM$ whose underlying $A_L$-module is $\sum_{g\in I}\im(g)$. This is indeed a pure Anderson motive, since if $I=f_1R+\ldots+f_nR$ are arbitrary generators, then $\ulM^I$ equals the image of the morphism 
\[
(f_1,\ldots,f_n):\ulM\oplus\ldots\oplus\ulM\es\longto\es\ulM\,.
\]
As $I$ contains an isogeny, $\ulM^I$ has the same rank as $\ulM$ and the natural inclusion is an isogeny which we denote $f_I:\ulM^I\to \ulM$.
\item 
If $I=\{\,f\in R:\im(f)\subset\ulM^I\,\}$ then $I$ is called a \emph{kernel ideal for $\ulM$}.
\end{suchthat}
\end{Definition}

The later terminology is borrowed from Waterhouse~\cite[\S 3]{Wat}. Since $\{\,f\in R:\im(f)\subset\ulM^I\,\}$ is the right ideal annihilating $\coker f_I$ one should maybe use the name ``cokernel ideal'' instead.

\begin{Proposition}\label{Prop3.21b}
Let $I\subset R$ be a right ideal which is a lattice, and consider the right ideal $J:=\{\,f\in R: \im(f)\subset\ulM^I\,\}\subset R$ containing $I$. Then $\ulM^{J}=\ulM^I$. In particular, $J$ is a kernel ideal for $\ulM$. We call $J$ the \emph{kernel ideal for $\ulM$ associated with $I$}.
\end{Proposition}

\begin{proof}
Obviously $J$ is a right ideal and $\ulM^{J}\subset\ulM^I$ by definition of $J$. Conversely $\ulM^I\subset\ulM^{J}$ since $I\subset J$.
\end{proof}

\begin{Lemma}\label{Lemma1.7.2}
\begin{suchthat}
\item 
For any $g\in I$, $f_I^{-1}\circ g:\ulM\to\ulM^I$ is a morphism and $g=f_I\circ(f_I^{-1}\circ g)$.
\item 
If $I=gR$ is principal, $g$ an isogeny, then $f_I^{-1}\circ g:\ulM\to\ulM^I$ is an isomorphism and $I$ is a kernel ideal.
\end{suchthat}
\end{Lemma}

\begin{proof}
1 is obvious since the image of $g$ lies inside $\ulM^I$.\\
2. Clearly $f_I^{-1}\circ g$ is injective since $g$ is an isogeny and surjective by construction, hence an isomorphism. To show that $I$ is a kernel ideal let $f\in R$ satisfy $\im(f)\subset\ulM^I$. Consider the diagram
\[
\xymatrix @C+2pc {\ulM \ar@{-->}[d]_h \ar[r]^{f_I^{-1}\circ f} & \ulM^I \ar[r]^{f_I} & \ulM\\
\ulM \ar[ur]_{f_I^{-1}\circ g}
}
\]
and let $h:=(f_I^{-1}\circ g)^{-1}\circ(f_I^{-1}\circ f)$. Then $f=gh\in I$ as desired.
\end{proof}

\noindent
{\it Example.} If $a\in A$ and $I=aR$, then $\ulM^I=a\ulM$ and $\coker f_I=\ulM/a\ulM$. More generally if $\Fa\subset A$ is an ideal and $I=\Fa R$ then $\ulM^I=\Fa\ulM$ and $\coker f_I=\ulM/\Fa\ulM$.

\begin{Proposition}\label{Prop1.7.3}
Let $I\subset R$ and $J\subset \End(\ulM^I)$ be right ideals which are lattices in $E$. Then also the product $K:=f_I\cdot J\cdot f_I^{-1}\cdot I$ is a right ideal of $R$ and a lattice in $E$ and $f_K^{-1}\circ f_I\circ f_J$ is an isomorphism of $(\ulM^I)^J$ with $\ulM^K$
\[
(\ulM^I)^J \xrightarrow{\es f_J\;} \ulM^I \xrightarrow{\es f_I\;} \ulM \xleftarrow{\;f_K\es} \ulM^K\,.
\]
\end{Proposition}

\begin{proof}
  If $f\in I$ and $g\in J$ then the morphism $f_I^{-1}\circ f:\ulM\to\ulM^I$ can be composed with $f_I\circ g$ to yield an element of $R$. Since $I$ and $J$ contain isogenies, $K$ is a right ideal and contains an isogeny. Clearly the images of $f_I\circ f_J$ and $f_K$ in $\ulM$ coincide since they equal the sum $\sum_{i,j}f_I\circ g_j\circ f_I^{-1} \circ f_i(\ulM)$ for sets of generators $\{f_i\}$ of $I$ and $\{g_j\}$ of $J$.
\end{proof}

\begin{Theorem}\label{Thm1.7.4}
Let $I,J\subset\End(\ulM)=:R$ be right ideals which are lattices in $E:=R\otimes_A Q$ and consider the following assertions:
\begin{suchthat}
\item 
$I$ and $J$ are isomorphic $R$-modules,
\item 
the pure Anderson motives $\ulM^I$ and $\ulM^J$ are isomorphic.
\end{suchthat}
Then 1 implies 2 and if moreover $I$ and $J$ are kernel ideals, also 2 implies 1.
\end{Theorem}

\begin{proof}
$1\Rightarrow 2$. Since $I$ and $J$ are lattices, the $R$-isomorphism $I\to J$ extends to an $E$-isomorphism of $E$ and is thus given by left multiplication with a unit $g\in E^\times$, that is, $J=gI$. There is an $a\in A$ such that $ag\in I\subset R$. Then $\im(ag)\subset\ulM^I$, that is, $f_I^{-1}\circ ag:\ulM\to\ulM^I$ is an isogeny.

Let $K$ be the right ideal $f_I\cdot\bigl(f_I^{-1}\circ ag\circ f_I\cdot\End(\ulM^I)\bigr)\cdot f_I^{-1}\cdot I$ of $R$. We claim that $\ulM^K\cong\ulM^{(ag)I}$. Namely, $\ulM^{(ag)I}\subset\ulM^K$ since $agI\subset K$. Conversely if $f\in I$, $h\in\End(\ulM^I)$, and $m\in\ulM$, then we find $m':=f_I\circ h\circ f_I^{-1}\circ f(m)\in\ulM^I$, that is, $m'=\sum_i f_i(m_i)$ for suitable $f_i\in I$ and $m_i\in \ulM$. It follows that $ag(m')=\sum_i agf_i(m_i)\in\ulM^{(ag)I}$ and therefore $\ulM^{(ag)I}=\ulM^K$. 

Applying Lemma~\ref{Lemma1.7.2} and Proposition~\ref{Prop1.7.3} now yields an isomorphisms $\ulM^I\cong\ulM^K=\ulM^{(ag)I}$. Likewise we obtain $\ulM^J\cong\ulM^{aJ}$ and the equality $aJ=agI$ then implies $\ulM^J\cong\ulM^I$ as desired.

\smallskip
\noindent
$2\Rightarrow 1$. Let $I$ and $J$ be kernel ideals and let $u:\ulM^I\to\ulM^J$ be an isomorphism. There is an $a\in A$ with $a\ulM\subset\ulM^I$. Therefore $g:=f_J\circ u\circ(f_I^{-1}\circ a):\ulM\to \ulM$ is an isogeny. 

We claim that $gI=aJ$, that is, left multiplication by $a^{-1}g$ is an isomorphism of $I$ with $J$. Let $f\in I$, then $h:=f_J\circ u\circ(f_I^{-1}\circ f)\in R$ has $\im(h)\subset\ulM^J$. So $h\in J$ since $J$ is a kernel ideal, and $g f=ah\in aJ$, since $a$ commutes with all morphisms. Conversely let $h\in J$, then $f:=f_I\circ u^{-1}\circ(f_J^{-1}\circ h)\in R$ has $\im(f)\subset \ulM^I$. So $f\in I$ since $I$ is a kernel ideal, and $ah=gf\in gI$ as desired.
\end{proof}

\begin{Proposition}\label{Prop1.7.5}
Let $I\subset R$ be a right ideal which is a lattice in $E$. Then $f_I\cdot\End(\ulM^I)\cdot f_I^{-1}$ contains the left order $\O=\{\,f\in E:fI\subset I\,\}$ of $I$ and equals it if $I$ is a kernel ideal.
\end{Proposition}

\noindent
{\it Remark.} Recall that $\End(\ulM^I)\otimes_A Q$ is identified with $E$ by mapping $h\in \End(\ulM^I)$ to $f_I\circ h\circ f_I^{-1}$.

\begin{proof}
Let $f\in \O$ and $g\in I$. Then $f g\in I$ and $f_I^{-1}\circ f\circ f_I\circ(f_I^{-1}\circ g)=f_I^{-1}\circ f g$ is a morphism from $\ulM$ to $\ulM^I$. If $g$ varies, the images of $f_I^{-1}\circ g$ exhaust all of $\ulM^I$. Hence $f_I^{-1}\circ f\circ f_I$ is indeed an endomorphism of $\ulM^I$. 
Conversely let $I$ be a kernel ideal and let $f=f_I\circ h\circ f_I^{-1}\in f_I\cdot\End(\ulM^I)\cdot f_I^{-1}$. If $g\in I$ then $f\circ g=f_I\circ h\circ(f_I^{-1}\circ g)\in R$ has $\im(f\circ g)\subset \ulM^I$. So $fg\in I$ as desired.
\end{proof}

We will now draw conclusions about the endomorphism ring $R$ similar to Waterhouse' results \cite{Wat} on abelian varieties by simply translating his arguments.

\begin{Theorem}\label{ThmW3.13}
Every maximal order in $E$ occurs as the endomorphism ring $f\cdot\End(\ulM')\cdot f^{-1}\subset E$ of a pure Anderson motive $\ulM'$ isogenous to $\ulM$ via an isogeny $f:\ulM'\to \ulM$.
\end{Theorem}

\begin{proof}
Let $S$ be a maximal order of $E$. Then the lattice $R$ contains $aS$ for some $a\in A$. Consider the right ideal $I=aS\cdot R$ whose left order contains $S$. By Proposition~\ref{Prop1.7.5}, $f_I\cdot\End(\ulM^I)\cdot f_I^{-1}$ contains the left order of $I$. Since $S$ is maximal we find $S=f_I\cdot\End(\ulM^I)\cdot f_I^{-1}$.
\end{proof}

\begin{Theorem}\label{ThmW3.14}
If $E$ is semisimple and $\End(\ulM)$ is a maximal order in $E$, so is $f_I\cdot\End(\ulM^I)\cdot f_I^{-1}$ for any right ideal $I\subset R$.
\end{Theorem}

\begin{proof}
By \cite[Theorem 21.2]{Re} the left order of $I$ is also maximal and then Proposition~\ref{Prop1.7.5} yields the result.
\end{proof}


From now on we assume that $L$ is a finite field and we set $e:=[L:\Fq]$. Let $\pi$ be the Frobenius endomorphism of $\ulM$.

\begin{Proposition}\label{PropW3.5}
The order $R$ in $E$ contains $\pi$ and $\deg(\pi)/\pi$.
\end{Proposition}

\begin{proof}
Clearly the isogeny $\pi$ belongs to $R$. Let now $a\in\deg(\pi)$. Then $a$ annihilates $\coker\pi$ by \ref{Prop3.28a} and so there is an isogeny $f:\ulM\to\ulM$ with $\pi\circ f=a$. The image $a/\pi$ of $f$ in $E$ belongs to $R$.
\end{proof}

\begin{Proposition}\label{ThmW3.15}
If $\ulM$ is a semisimple pure Anderson motive over a finite field and $\End(\ulM)$ is a maximal order in $E=\End(\ulM)\otimes_A Q$, then every right ideal $I\subset\End(\ulM)$, which is a lattice, is a kernel ideal for $\ulM$, and $\deg(f_I)=N(I):=\bigl(N(f):f\in I\bigr)$.
\end{Proposition}

\begin{proof}
(cf.\ \cite[Theorem 3.15]{Wat}) Let $f\in I$, then $f=f_I\circ f_I^{-1}f$ and $N(f)\in\deg(f)\subset\deg(f_I)$ by Lemma~\ref{Lemma1.7.7}. Therefore $N(I)\subset\deg(f_I)$. Let $R'$ be the left order of $I$. It is maximal by \cite[Theorem 21.2]{Re}. For a suitable $a\in A$ the set $J':=\{\,x\in E: xI\subset aR\,\}$ is a right ideal in $R'$ and a lattice in $E$ and satisfies $J'\cdot I=aR$ by \cite[Theorem 22.7]{Re}. Let $J:=f_I^{-1}J'f_I\subset\End(\ulM^I)$ be the induced right ideal of $\End(\ulM^I)=f_I^{-1} R'f_I$; see \ref{Prop1.7.5}. Then $\coker f_I\circ f_J=\coker f_{J'I}=\coker a$ by Proposition~\ref{Prop1.7.3}. Therefore Theorem~\ref{Prop3.4.1} and \cite[24.12 and 24.11]{Re} imply
\[
N(a)\cdot A=N(J')\cdot N(I)\subset(\deg f_J)(\deg f_I)=\deg(a)=N(a)\cdot A\,.
\]
By the above we must have $N(I)=\deg(f_I)$ since $A$ is a Dedekind domain. If $I$ were not a kernel ideal its associated kernel ideal would be a larger ideal with the same norm. But this is impossible by \cite[24.11]{Re}.
\end{proof}

Like for abelian varieties there is a strong relation between the ideal theory of orders of $E$ and the investigation of isomorphy classes of pure Anderson motives isogenous to $\ulM$. We content ourselves with the following result which is analogous to Waterhouse~\cite[Theorem 6.1]{Wat}. The interested reader will find many other results without much difficulty.

\begin{Theorem}\label{ThmW6.1}
Let $\ulM$ be a simple pure Anderson motive of rank $r$ and dimension $d$ over the smallest possible field $\Fq$. Then
\begin{suchthat}
\item 
$\End(\ulM)$ is commutative and $E:=\End(\ulM)\otimes_A Q=Q(\pi)$.
\item 
All orders $R$ in $Q(\pi)$ containing $\pi$ are endomorphism rings of pure Anderson motives isogenous to $\ulM$. Any such order automatically contains $N(\pi)/\pi=N_{Q(\pi)/Q}(\pi)/\pi$.
\item 
For each such $R$ the isomorphism classes of pure Anderson motives isogenous to $\ulM$ with endomorphism ring $R$ correspond bijectively to the isomorphism classes of $A$-lattices in $E$ with order $R$.
\end{suchthat}
\end{Theorem}

\begin{proof} 
1 follows from \ref{Thm3.8} and \ref{Cor3.13b}.

\noindent
2. Let $R$ be an order in $Q(\pi)$ containing $\pi$ and let $v\ne\chr$ be a maximal ideal of $A$. Since $[E:Q]=r$ and $E_v$ is semisimple, there is by Lemma~\ref{Lemma3.4.1b} an isomorphism $E_v\isoto V_v\ulM$ of (left) $E_v$-modules given by $f\mapsto f(x)$ for a suitable $x\in V_v\ulM$. It identifies $R_v:=R\otimes_A A_v$ with a $\pi$-stable lattice $\Lambda_v=R_v\cdot x$ in $V_v\ulM$, which without loss of generality is contained in $T_v\ulM$. By Proposition~\ref{Prop2.7b} 
there is an isogeny $f:\ulM'\to \ulM$ of pure Anderson motives with $T_vf(T_v\ulM')=\Lambda_v$. By Theorem~\ref{TATE-CONJECTURE-MODULES} we conclude
\[
\End(\ulM')\otimes_A A_v =\End_{A_v[\pi]}(\Lambda_v)=R_v\,.
\]
For $v=\chr$ note that $Q_{\chr,L}=Q_\chr$ since $L=\Fq$. In particular $\BF_\chr=\BF_q$. Since $\dim_{Q_\chr}\ulN_\chr(\ulM)=r=[E:Q]$, Theorem~\ref{ThmLS2} together with Lemma~\ref{Lemma3.4.1b} show that $E_\chr$ is isomorphic to $\ulN_\chr(\ulM)$ as left $E_\chr$-modules. Since $R$ contains $\pi$, the image of $R_\chr:=R\otimes_A A_\chr$ in $\ulN_\chr(\ulM)$ is a local $\sigma$-subshtuka $\ulHM{}'$ of $\ulM_\chr(\ulM)$ of the same rank. (If it is not contained in $\ulM_\chr(\ulM)$, multiply it with a suitable $a\in A$.) Then Proposition~\ref{Prop2.18b} yields an isogeny of pure Anderson motives $f:\ulM'\to \ulM$ such that $\ulM_\chr(f)\bigl(\ulM_\chr(\ulM')\bigr)=\ulHM{}'$ and
\[
\End(\ulM')\otimes_A A_\chr=\End_{A_\chr[\phi]}(\ulHM{}')=R_\chr
\]
by Theorem~\ref{ThmLS2}.
Since each of these operations only modifies $\End(\ulM)$ at the respective place $v$, this shows that we may modify $\ulM$ at all places to obtain a pure Anderson motive $\ulM'$ with $\End(\ulM')=R$. Now the last statement follows from Proposition~\ref{PropW3.5} and Theorem~\ref{Prop3.4.1}.

\smallskip
\noindent
3. Let $R$ be such an order. By what we proved in 2 there is a pure Anderson motive $\ulTM$ for which all $T_v\ulTM\cong R_v$ and $\ulM_\chr(\ulTM)\cong R_\chr$. Let $I\subset R$ be a (right) ideal which is an $A$-lattice in $E$ and consider the isogeny $f_I:\ulTM^I\to\ulTM$. Under the above isomorphisms $T_v f_I(T_v\ulTM^I)\cong I\otimes_A A_v=:I_v$ and $\ulM_\chr f_I(\ulM_\chr\ulTM^I)\cong I\otimes_A A_\chr=:I_\chr$. Conversely if $f:\ulM'\to\ulTM$ is an isogeny then $\ulM_v f(\ulM_v\ulM')$ is a (left) $R_v$-module because $R=\End(\ulTM)$, hence isomorphic to an $R_v$-ideal $I_v$. This shows that any isogeny $f:\ulM'\to\ulTM$ is of the form $f_I:\ulTM^I\to\ulTM$.

If now $f\in R$ satisfies $\im(f)\subset\ulTM^I$ then $f\in I_\chr$ and $f\in I_v$ for all $v$ and therefore $f\in I$. This shows that every $I$ is a kernel ideal for $\ulTM$. By Proposition~\ref{Prop1.7.5}, $\End(\ulTM^I)$ is the (left) order of $I$. Since every lattice with order $R$ in $E$ is isomorphic to an ideal of $R$, we have
\[
\xymatrix @R-1pc {\{\,A\text{-lattices in }E\text{ with order }R\,\}/_\sim \ar@{=}[d] \\
 \{\,I\subset R\text{ Ideals with order }R\,\}/_\sim \es \ar[r]^{\sim\qquad\quad} &
\es \bigl\{\,\ulTM^I\xrightarrow{f_I}\ulTM\to\ulM\text{ with }\End(\ulTM^I)=R\,\bigr\}/_\sim
}
\]
and the assertion now follows from Theorem~\ref{Thm1.7.4}.
\end{proof}


\vfill
 
\noindent
\parbox[t]{.45\textwidth}{
Matthias Bornhofen  \\ 
Am Rösslewald 4   \\ 
D -- 79874 Breitnau \\ 
Germany  \\[0.1cm] }
\parbox[t]{.5\textwidth}{ 
Urs Hartl  \\ 
Mathematisches Institut  \\ 
Universit{\"a}t M\"unster  \\ 
Einsteinstr.\ 62  \\ 
D -- 48149 M\"unster\\ 
Germany  \\[0.1cm] 
http:/\!/www.math.uni-muenster.de/u/urs.hartl/}

\end{document}